\documentclass[11pt]{article}
\usepackage[authoryear]{natbib}

\usepackage{amsfonts,amsmath,amssymb,amsthm} 
\usepackage{fullpage,graphicx,subcaption}
\usepackage{tikz,fouriernc,tocloft}
\usepackage{multirow}

\usepackage{hyperref}
\usepackage{floatpag}
\floatpagestyle{empty}

\theoremstyle{plain}
\newtheorem{theorem}{Theorem}
\newtheorem{lemma}[theorem]{Lemma}

\theoremstyle{definition}

\newtheorem{remark}{Remark}

\let\hat\widehat
\let\tilde\widetilde
\let\bar\overline

\def\E{\mathbb{E}}
\def\P{\mathbb{P}}
\def\R{\mathbb{R}}

\def\Cov{\mathrm{Cov}}
\def\Var{\mathrm{Var}}

\def\tr{\mathrm{tr}}

\def\sign{\mathrm{sign}}

\def\hbeta{\hat{\beta}}

\usepackage{calrsfs}
\DeclareMathAlphabet{\pazocal}{OMS}{zplm}{m}{n}
\def\cE{\pazocal{E}}
\def\cL{\pazocal{L}}
\def\cM{\pazocal{M}}
\def\cP{\pazocal{P}}
\def\cT{\pazocal{T}}

\def\figdir{figs}

\begin{document}

\title{Uniform Asymptotic Inference and the Bootstrap After Model Selection}
\author{Ryan Tibshirani, Alessandro Rinaldo, Rob Tibshirani, and  
  Larry Wasserman} 
\date{Carnegie Mellon University and Stanford University}
\maketitle

\begin{abstract}
Recently, \citet{tibshirani2016exact} proposed a method for making
inferences about parameters defined by model selection, in a typical 
regression setting with normally distributed errors. Here, we study
the large sample properties of this method, without assuming
normality.  We prove that the test statistic of
\citet{tibshirani2016exact} is asymptotically valid, as the number of
samples $n$ grows and the dimension $d$ of the regression problem
stays fixed. Our asymptotic result holds uniformly over a wide class
of nonnormal error distributions. We also propose an efficient
bootstrap version of this test that is provably (asymptotically)
conservative, and in practice, often delivers shorter
intervals than those from the original normality-based approach.
Finally, we prove that the test statistic of
\citet{tibshirani2016exact} does not enjoy uniform validity in a
high-dimensional setting, when the dimension $d$ is allowed grow. 
\end{abstract}

\section{Introduction}
\label{sec:intro}

There has been a recent surge of work on conducting formally valid
inference in a regression setting after a model selection event has  
occurred, see
\citet{berk2013valid,lockhart2014significance,tibshirani2016exact,
lee2016exact,fithian2014optimal,bachoc2014valid}, just to name a few. 
Our interest in this paper stems in particular from the work of
\citet{tibshirani2016exact}, who presented a method to produce valid
p-values and confidence intervals for adaptively fitted coefficients
from any given step of a sequential regression procedure like  
forward stepwise regression (FS), least angle regression (LAR), or the
lasso (the lasso is meant to be thought of as tracing out a sequence
of models along its solution path, as the penalty parameter descends
from  $\lambda=\infty$ to $\lambda=0$).    
These authors use a statistic that is carefully crafted to be pivotal
after conditioning on the model selection event.  This idea is not
specific to the sequential regression setting, and is an example
of a broader framework that we might call 
{\it selective pivotal inference}, applicable in many other settings, 
as in, e.g., \citet{taylor2016inference,lee2016exact,
lee2014exact,loftus2014significance,reid2017post, 
choi2014selecting,fithian2014optimal,hyun2016exact}. 
 
A key to the methodology in \citet{tibshirani2016exact}
(and much of the work in selective pivotal inference) is to the assumption of
normality of the errors. To fix notation, consider the regression of a response 
$Y \in \R^n$ on predictor
variables $X_1,\ldots,X_d \in \R^n$, stacked together as
columns of a matrix $X \in \R^{n\times d}$.  We will treat the
predictors $X$ are fixed (nonrandom), and assume the model
\begin{equation}
\label{eq:model}
Y_i = \theta_i + \epsilon_i, \;\;\; i=1,\ldots,n,
\end{equation}
where $\theta \in \R^n$ is an unknown mean parameter of
interest. \citet{tibshirani2016exact} assume that the errors
$\epsilon_1,\ldots,\epsilon_n$ are i.i.d.\ $N(0,\sigma^2)$, where the
error variance $\sigma^2>0$ is known.  An
advantage of their approach is that it does not require $\theta$ to be
an exact linear combination of the predictors $X_1,\ldots,X_d$, and
makes no assumptions about the correlations among these predictors.
But as far as the finite-sample guarantees are concerned, normality of
the errors is crucial.  In this work, we examine the properties of
the test statistic proposed in
\citet{tibshirani2016exact}---hereafter, the {\it truncated
  Gaussian} (TG) statistic---without using an assumption 
about normal errors.  We only assume that
$\epsilon_1,\ldots,\epsilon_n$ are i.i.d.\ from a distribution with
mean zero and essentially no other restrictions.  

A high-level description of the selective pivotal inference framework
for sequential regression is as follows (details
are provided in Section \ref{sec:inference}).  FS, LAR, or the
lasso is run for some number of steps $k$, and a model is selected,
call it $M$. For FS and LAR, this model will always have $k$ active 
variables, and for the lasso, it will have at most $k$, as variables
can be added to or deleted from the active set at each step.  We
specify a linear contrast of the mean $v^T \theta$ of interest, e.g.,
one giving the coefficient of a variable of interest in the model
$M$ at step $k$, in the regression of $\theta$ onto the active
variables. By assuming normal errors in \eqref{eq:model}, and
examining the distribution of $v^T Y$ conditional on having selected
model $M$, which we denote by \smash{$\hat{M}(Y)=M$}, we can construct
a confidence interval $C_\alpha$ satisfying
\begin{equation*}
\P\Big( v^T \theta \in C_\alpha \, \Big | \, 
    \hat{M}(Y)=M \Big) = 1-\alpha,
\end{equation*}
for a given $\alpha \in [0,1]$.  The interpretation: if we were to
repeatedly draw $Y$ from \eqref{eq:model} and run FS, LAR, or the
lasso for $k$ steps, and only pay attention to cases in which we
selected model $M$, then among these cases, the constructed intervals 
$C_\alpha=C_\alpha(Y;M)$ contain $v^T \theta$ with frequency
tending to $1-\alpha$.    

The above is a {\it conditional} perspective of the selective pivotal
inference framework for FS, LAR, and lasso.  An {\it unconditional} or marginal  
point of view is also possible, which we now describe. For each possible
selected model $M$, a constrast vector $v_M$ is 
specified, and the contrast $v_M^T \theta$ is considered when
model $M$ is selected, \smash{$\hat{M}(Y)=M$}.  To be concrete, we can
again think of a setup such that $v_M^T \theta$ gives the
coefficient of a variable in the model $M$ at step $k$, in
the projection of $\theta$ onto the active set.  Confidence intervals
are then constructed in exactly the same manner as above (without 
change), and conditional coverage over all models $M$ implies the 
following unconditional property for $C_\alpha$,
\begin{equation*}
\P\Big( v_{\hat{M}(Y)}^T \theta \in C_\alpha\Big) = 1-\alpha.
\end{equation*}
The interpretation is different: if we were to repeatedly draw $Y$
from \eqref{eq:model} and run FS, LAR, or lasso for $k$ steps, and 
construct confidence intervals
\smash{$C_\alpha=C_\alpha(Y;\hat{M}(Y))$}, then these intervals 
contain their respective targets \smash{$v_{\hat{M}(Y)}^T \theta$}
with frequency approaching $1-\alpha$.  Notice that, by construction, 
the target itself may change each time we draw $Y$, though it is the  
same for all $Y$ that give rise to the same selected model.  In terms
of the setting for regression contrasts described above, each time we
draw $Y$ and carry out the inferential procedure, the interval
$C_\alpha$ covers the coefficient of a possibly different variable
in the active model, in the projection of $\theta$ onto the
active variables. Figure \ref{fig:coverage} demonstrates this point.   

\begin{figure}[tb]
\centering
\includegraphics[width=\textwidth]{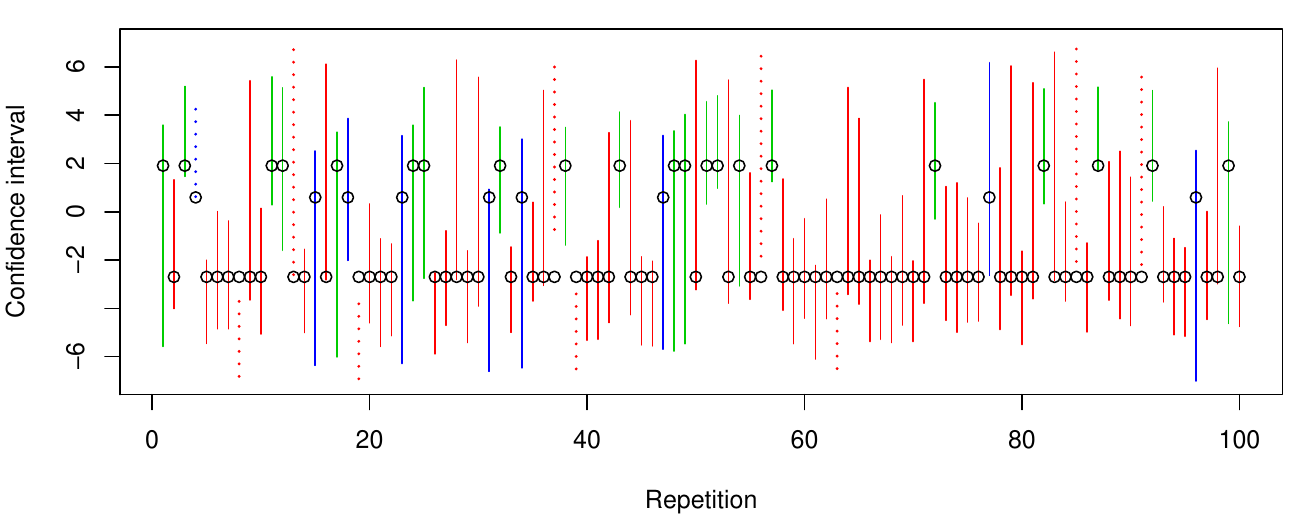}
\caption{\it\small An example of conditional and unconditional  
  coverage for one step of FS (the variables are normalized, and this
  is equivalent to one step of LAR, or lasso).  Here $n=20$ and $d=3$,
  and a response $Y$ was drawn 100 times from a model as in
  \eqref{eq:model} with i.i.d.\ $N(0,\sigma^2)$ errors.  The different
  colors denote different active models that were selected after one
  step, where an active model is a variable-sign pair, namely, the
  variable achieving the largest absolute inner product with $Y$, and
  the sign of this inner product.  Across the 100 repetitions, the
  circles denote a target to be covered, and the segments are 90\%    
  confidence intervals.  E.g, the color green corresponds to
  the model $+X_2$, so in repetitions 1, 3, 11, 12, etc.,
  $X_2^T Y$ was largest among all absolute inner products of
  variables with $Y$, and the green segments denote 90\% 
  confidence intervals designed to cover the contrast $X_2^T 
  \theta$.  Similarly, red corresponds to the model $-X_1$, and
  blue to $+X_3$.  Dotted segments indicate that the given interval
  does not cover its target.  The empirical coverage among green
  intervals: 21/21, among red intervals: 61/70, and among blue
  intervals: 8/9.  Hence in each case, the empirical
  coverage is close to the nominal 90\% level. Further, in total, 
  i.e., unconditionally, the empirical coverage is 90/100, right at
  the nominal 90\% level. }  
\label{fig:coverage}
\end{figure}

\subsection{Uniform convergence}
\label{sec:uniform}

When making asymptotic inferential guarantees, as we do in this paper,
it is important to be clear about the type of guarantee.  Here we
review the concepts of uniform convergence and validity. Let      
$\xi_1,\ldots,\xi_n \in \R^s$ be random vectors with joint
distribution $(\xi_1,\ldots,\xi_n) \sim F_n$, where $F_n \in \cP_n$, 
and $\cP_n$ is a class of distributions.  For
example, we could have $\xi_1,\ldots,\xi_n \in \R^s$ 
i.i.d.\ from $F$, and the class $\cP_n$ could contain product
distributions of the form $F_n=F \times \ldots \times F$ ($n$ times);
our notation allows for a more general setup than this one.   Let  
$W_n=T_n(\xi_1,\ldots,\xi_n)$ for a statistic $T_n$, and
$W \sim G$, where $W_n,W \in \R^q$. We will say that $W_n$, converges
{\it uniformly in distribution} to $W$, over $\cP_n$, provided that      
\begin{equation}
\label{eq:uniform_converge}
\lim_{n \to \infty} \; \sup_{F_n \in \cP_n} \; 
\sup_{x \in \R^q} \; \big|\P_{F_n}(W_n \leq x) - \P(W \leq x)\big| =
0. 
\end{equation}
(The above inequalities, as in $W_n \leq x$ and $W \leq x$, are meant 
to be interpreted componentwise; we are also implicitly assuming 
that the limiting distribution $G$ is continuous, otherwise the above
inner supremum should be restricted to continuity points $x$ of $G$.)  
This is much stronger than the notion of {\it pointwise}
convergence in distribution, which only requires that
\begin{equation}
\label{eq:pointwise_converge}
\lim_{n \to \infty} \; 
\sup_{x \in \R^q} \; \big|\P_{F_n}(W_n \leq x) - \P(W \leq x)\big| =
0,
\end{equation}
for a particular sequence of distributions $F_n$, $n=1,2,3,\ldots$.  

A recent article by \citet{kasy2015uniformity} emphasizes the
importance of uniformity in asymptotic approximations.  This authors
points out that a uniform version of the continuous mapping theorem
follows directly from a standard proof of the continuous mapping
theorem (e.g., see Theorem 2.3 in \citet{vandervaart1998asymptotic}). 

\begin{lemma}
\label{lem:cmt}
Suppose that $W_n$ converges uniformly in distribution to $W$, with
respect to the class $\cP_n$.   
Let $\psi : \R^q \to \R$ be a map that is continuous on a set
$D$, such that $\P(W \in D)=1$.
Then $\psi(W_n)$ converges uniformly in distribution to $\psi(W)$ with
respect to $\cP_n$. 
\end{lemma}

\citet{kasy2015uniformity} also remarks that the central limit theorem
for triangular arrays, specifically the Lindeberg-Feller central limit 
theorem (e.g., Proposition 2.27 in \citet{vandervaart1998asymptotic}) 
naturally extends to the uniform case.  The logic is, roughly
speaking: uniform convergence in \eqref{eq:uniform_converge} is
equivalent to pointwise convergence over {\it all} sequences of
distributions $F_n$,  
$n=1,2,3,\ldots$, and triangular arrays, by design, can have a
different distribution assigned to each row.  Therefore if the
Lindeberg condition holds for any possible sequence, then so does the
convergence to normality.

\begin{lemma}
\label{lem:clt}
Let $\xi_1,\ldots,\xi_n \in \R^q$ be a triangular array of independent
random vectors, with joint distribution $F_n$.  Assume 
$\xi_1,\ldots,\xi_n$ have mean zero and finite variance.  Also assume
that for any sequence $F_n \in \cP_n$, $n=1,2,3,\ldots$, we have  
\begin{equation*}
\lim_{n\to\infty} \; \sum_{i=1}^n \E_{F_n} \Big( \|\xi_i\|_2^2 \cdot
1\{\|\xi_i\|_2 \geq \epsilon \} \Big) = 0, 
\;\;\; \text{for all $\epsilon>0$},
\end{equation*}
and
\begin{equation*}
\lim_{n\to\infty} \; \sum_{i=1}^n \Cov_{F_n}(\xi_i) = \Sigma,
\end{equation*}
where $\Sigma$ does not depend on the sequence $F_n$,
$n=1,2,3,\ldots$.  Then  
\smash{$W_n = \sum_{i=1}^n \xi_i$} converges in distribution to 
$W \sim N(0,\Sigma)$, uniformly with respect to $\cP_n$.  
\end{lemma}

In our work, a motivating reason for the study of uniform convergence
is the associated property of {\it uniform validity} of 
asymptotic confidence intervals.  That is, if $W_n=W_n(\mu)$ depends 
on a parameter $\mu=\mu(F_n)$ of the distribution $F_n$, but $W$ does 
not, then we can consider any $(1-\alpha)$ confidence set
$C_{n,\alpha}$ built from a $(1-\alpha)$ probability rectangle 
$R_\alpha$ of $W$,   
\begin{equation*}
C_{n,\alpha} = \{ \mu : W_n(\mu) \in R_{\alpha} \},
\end{equation*}
and the uniform convergence of $W_n$ to $W$, really just by
rearranging its definition in \eqref{eq:uniform_converge}, implies 
\begin{equation}
\label{eq:uniform_interval}
\lim_{n\to\infty} \; \sup_{F_n \in \cP_n} \; \sup_{\alpha \in [0,1]}
\;  \Big|\P_{F_n} \Big( \mu(F_n) \in C_{n,\alpha} \Big) -(1-\alpha)  
\Big| = 0.
\end{equation}
Meanwhile, pointwise convergence as in
\eqref{eq:pointwise_converge} only implies 
\begin{equation}
\label{eq:pointwise_interval}
\lim_{n\to\infty} \; \sup_{\alpha \in [0,1]}\; 
\Big| \P_{F_n} \Big( \mu(F_n) \in C_{n,\alpha} \Big) -(1-\alpha)
\Big| = 0,
\end{equation}
for a particular sequence $F_n$, $n=1,2,3,\ldots$. For a confidence
set satisfying \eqref{eq:uniform_interval}, and a given tolerance
$\epsilon>0$,  there exists a sample size $n(\epsilon)$ such that 
the coverage is  
guaranteed to be at least $1-\alpha-\epsilon$, for $n \geq
n(\epsilon)$, no matter the underlying distribution (over the
class of distributions in question). 
Note that this is not necessarily true for a pointwise confidence set 
as in \eqref{eq:pointwise_interval}, as the required sample size 
here could depend on the particular distribution under consideration.    

\subsection{Summary of main results}

An overview of our main contributions is as follows.

\begin{enumerate}
\item We establish that TG statistics for typical inferences along the FS,  
  LAR, and lasso paths only depend on the data $(X,Y)$ through
  \smash{$\frac{1}{n} X^T X$} and \smash{$\frac{1}{\sqrt{n}} X^T Y$} (Lemmas
  \ref{lem:master1}, \ref{lem:master2}, and \ref{lem:master3} in Section
  \ref{sec:master}), which is important since these two quantities have
  asymptotic limits in a standard low-dimensional asymptotic setup.

\item Placing mild constraints on the mean and error
  distribution in \eqref{eq:model}, and treating the dimension $d$ as 
  fixed, we prove that the TG test statistic is
  asymptotically pivotal, converging to $U(0,1)$ (the standard uniform
  distribution), when evaluated at the
  true population value for its pivot argument. We show that this
  holds uniformly over a wide class of distributions for the errors,
  without any real restrictions on the predictors $X$ (first part of
  Theorem \ref{thm:pivot} in Section \ref{sec:asymptotics}).      

\item The resulting confidence intervals are therefore asymptotically
  uniformly valid, over the same class of distributions (second part
  of Theorem \ref{thm:pivot} in Section \ref{sec:asymptotics}).   

\item The above asymptotic results assume that the error variance
  $\sigma^2$ is known, so for $\sigma^2$ unknown, we propose 
  a plug-in approach that replaces $\sigma^2$ in the TG statistic with 
  a simple estimate, and alternatively, an efficient bootstrap
  approach.  Both allow for conservative asymptotic inference (Theorem  
  \ref{thm:unknown} in Section \ref{sec:unknown}). 

\item We present detailed numerical experiments that support the
  asymptotic validity of the TG p-values and confidence intervals for
  inference in low-dimensional regression problems that have nonnormal 
  errors (Section \ref{sec:examples}). Our experiments reveal that the 
  plug-in and bootstrap versions also show good performance,
  and the bootstrap method can often deliver substantially shorter
  intervals than those based directly on the TG statistic.

\item Our experiments also also suggest that the TG test statistic
  (and plug-in, bootstrap variants) may be asymptotically valid in
  even broader settings not covered by our theory, e.g., problems
  with heteroskedastic errors and (some) high-dimensional problems.    

\item We prove that TG statistic does not exhibit a general uniform
  convergence to $U(0,1)$ when the dimension $d$ is allowed to
  increase (Theorem \ref{thm:highdim} in Section \ref{sec:highdim}).      
\end{enumerate}

\subsection{Related work}

A recent paper by \citet{tian2017asymptotics} is very related to our
work here. These authors examine the asymptotic distribution of the  
TG statistic under nonnormal errors.  Their main result proves
that the TG statistic is asymptotically pivotal, under some
restrictions on the model selection events in
question.  We view their work as providing a complementary perspective 
to our own: they consider a setting where the dimension $d$ grows, 
but place strong regularity conditions on the selected models; 
we adopt a more basic setting
with $d$ fixed, and prove more broad uniformly valid convergence   
results for the TG pivot, free of regularity conditions. 

In a sequence of papers,
\citet{leeb2003finite,leeb2006canone,leeb2008canone} prove that in a  
classical regression setting, it is impossible to find the
distribution of a post-selection estimator of the underlying
coefficients, even asymptotically.  Specifically, they prove for an
estimate \smash{$\hbeta$} of some underlying coefficient  
vector \smash{$\beta_0$}, any quantity of the form
\smash{$Q_n = \sqrt{n}A(\hbeta-\beta_0)$}, for a linear transform  
$A$, cannot be used for inference after model selection.  Though $Q_n$ 
can be made to be pivotal or at least asymptotically pivotal (once $A$ is chosen 
once appropriately), this is no longer true in the presence of selection, even
if the dimension $d$ is fixed and the sample size $n$ approaches $\infty$. 
Furthermore, they show that there is no uniformly consistent estimate
of the distribution of $Q_n$ 
(either conditionally or unconditionally), which makes $Q_n$
unsuitable for inference. This fact is essentially a manifestation of
the well-known Hodges phenomenon.  The selective pivotal inference framework, 
and hence our paper, circumvents this problem as we do not claim (nor attempt)
to estimate the distribution of $Q_n$, and instead make inferences
using an entirely different pivot that is constructed via a careful
conditioning scheme.    

\subsection{Notation}

As our paper considers an asymptotic regime, with the number of 
samples $n$ growing, we will often use a subscript $n$ to mark the
dependence of various quantities on the sample size. An exception 
is our notation for the predictors, response, and mean, which we
will always denote by $X,Y,\theta$, respectively. Though these
quantities will (of course) vary with $n$, our notation hides
this dependence for simplicity.  

When it comes to probability statements involving $Y$, drawn from
\eqref{eq:model}, we will write $\P_{f(\theta)=\mu}(\,\cdot\,)$ to 
denote the probability operator under a mean vector
$\theta$ such that $f(\theta)=\mu$.  With a subscript omitted, as in 
$\P(\,\cdot\,)$, it is implicit that the probability is taken under
$\theta$. Also, we will generally write $y$
(lowercase) for an arbitrary response vector, and $Y$ (uppercase) for
a random response vector drawn from \eqref{eq:model}.  This is
intended to distinguish statements that 
hold for an arbitrary $y$, and statements that hold for a 
random $Y$ with a certain distribution. 
Lastly, we will denote \smash{$\hat{M}$} the model selection
procedure associated with the regression algorithm under consideration
(FS, LAR, or lasso), and we  
will treat this as a mapping from $\R^n$ to the space of models, so
that \smash{$\hat{M}(y)$} is a fixed quantity, representing 
the model selected when the response is the fixed vector $y$, and
\smash{$\hat{M}(Y)$} is a random variable, representing the model
selected when the response is the random vector $Y$. Similar
notation will be used for related quantities.

\section{Selective inference}
\label{sec:inference}

In this section, we review the selective pivotal inference framework
for sequential regression procedures.  We present interpretations for the
inferences from both conditional and unconditional perpsectives, in Sections  
\ref{sec:conditional} and \ref{sec:unconditional}, respectively.  The other
subsections provide the necessary details for understanding the framework, 
beginning with the selection events encountered along the FS, LAR, and lasso
paths. 

\subsection{Model selection}
\label{sec:selection}

Consider forward stepwise regression (FS), least angle regression
(LAR), or the lasso, run for a number of steps  
$k$, where $k$ is arbitrary (but treated as fixed throughout this 
paper). Such a procedure defines a {\it partition} of the sample
space, \smash{$\R^n=\bigcup_{M \in \cM} \Pi_M$}, with elements 
\begin{equation}
\label{eq:partition}
\Pi_M = \{y : \hat{M}(y) = M\}, \;\;\; M \in \cM.
\end{equation}
Here \smash{$\hat{M}(y)$} denotes the {\it selected model} from 
the given $k$-step procedure, run on $y$, and $\cM$
is the space of possible models.  Calling \smash{$\hat{M}(Y)$} a
selected model may be bit of an abuse of common nomenclature, 
because, as we will see, \smash{$\hat{M}(y)$} will describe more than
just a set of selected variables at the point $y$.  In fact, one can
think of \smash{$\hat{M}(y)$} as a representation of the decisions
made by the algorithm across its $k$ steps. For FS, we define 
\smash{$\hat{M}(y)=\{(\hat{A}_\ell(y),\hat{s}_\ell(y)): \ell=1,\ldots,
  k\}$}, comprised of two things:   
\begin{enumerate}
\item a sequence of active sets \smash{$\hat{A}_\ell(y)$},
  $\ell=1,\ldots,k$, denoting the variables that are given nonzero
  coefficients, at each of the $k$ steps;
\item a sequence of sign vectors \smash{$\hat{s}_\ell(y)$},
  $\ell=1,\ldots,k$, denoting the signs of nonzero coefficients, at
  each of the $k$ steps. 
\end{enumerate}
The active sets are nested across steps, \smash{$\hat{A}_1(y)
  \subseteq \hat{A}_2(y) \subseteq \hat{A}_3(y) \subseteq \ldots$}, 
as FS selects one variable to add to the
active set at each step.   However, the sign vectors
\smash{$\hat{s}_1(y), \hat{s}_2(y), \hat{s}_3(y), \ldots$} are not,
since these are determined by least squares on the active variables at
each step.  Hence, as defined, the number of possible models
\smash{$\hat{M}(y)$} after $k$ steps of FS is 
\begin{equation*}
|\cM| =
d\cdot(d-1)\cdots(d-k+1)\cdot 2 \cdot 2^2 \cdots 2^k= 
O(d^k 2^{k^2}).
\end{equation*}
Moreover, the corresponding  
partition elements $\Pi_M$, $M \in \cM$ in \eqref{eq:partition}  
are all convex cones.  The proof of this fact is not difficult, and
requires only a slight modification of the arguments in
\citet{tibshirani2016exact},
given in Appendix \ref{app:partition} for completeness. The 
result is easily seen for $k=1$: after one step of FS,
assuming without a loss of generality that $X_1,\ldots,X_d$ have unit
norm, we can express, e.g.,  
\begin{align*}
\big\{y : \big(\hat{A}_1(y),\hat{s}_1(y)\big)=(1,1) \big\} &= 
\big\{ y : X_1^T y \geq \pm X_j^T y, \; 
j=2,\ldots,d\big\} \\ 
&= \bigcap_{j=2}^d \big\{ y : (X_1-X_j)^T y \geq 0 \big\}  
\cap \big\{y : (X_1+X_j)^T y \geq 0\big\},
\end{align*}
the right-hand side above being an intersection of half-spaces passing
through zero, and therefore a convex cone.  As we enumerate the
possible choices for \smash{$(\hat{A}_1(y),\hat{s}_1(y))$}, these
cones form a partition of $\R^n$.  Figure \ref{fig:partition} shows
an illustration.  

\begin{figure}[h]
\centering
\includegraphics[width=0.6\textwidth]{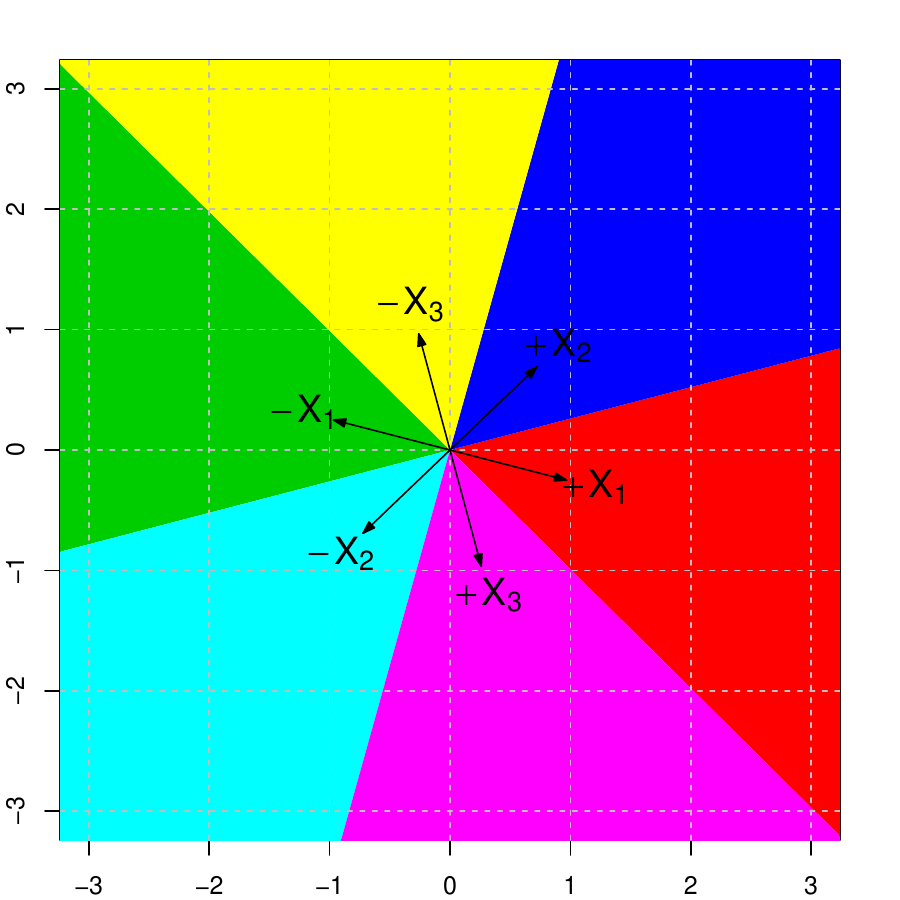}
\caption{\it\small An example of the model selection partition from
  one step of FS (the variables are normalized, and this is equivalent
  to one step of LAR, or lasso).  Here $n=2$ and $d=3$.  The colors 
  indicate the regions of the sample space $\R^2$ for which different
  models---pairs of active variables and signs---are selected, so
  that, e.g., the red region contains points in $\R^2$ that are
  maximally aligned with $X_1$.} 
\label{fig:partition}
\end{figure}

For LAR and the lasso, we need to modify the definition of the
selected model \smash{$\hat{M}(y)$} in order for the resulting   
partition elements in \eqref{eq:partition} to be convex cones.
We add an ``extra'' bit of model information and
define \smash{$\hat{M}(y)=\{(\hat{A}(y),\hat{s}(y),\hat{I}_\ell(y)) :
\ell=1,\ldots,k\}$}, where
\smash{$\hat{I}_\ell(y)$} is a list of variables that play a special role 
in the construction of the LAR or lasso active set at the $\ell$th
step, but that a user would not typically pay attention to.  
In truth, the latter quantity is only a detail that is included so
that $\Pi_M$, $M \in \cM$ are convex cones (without it, the partition 
elements would each be a union of cones), and so we do not describe it
here.  Furthermore, it does not affect our treatment of 
inference in what follows, and for this reason, we will largely 
ignore the minor differences in model selection events between FS,
LAR, and lasso hereafter.  

The description of
\smash{$\hat{I}_\ell(y)$, $\ell=1,\ldots,k$}, and the proof that the
partition elements $\Pi_M$, $M \in \cM$ are cones for LAR and
lasso, mirrors that in \citet{tibshirani2016exact}, and is again given in 
Appendix \ref{app:partition}.   Like FS, the active sets 
from LAR are nested, \smash{$\hat{A}_1(y)
  \subseteq \hat{A}_2(y) \subseteq \hat{A}_3(y) \subseteq \ldots$},
since one variable is added to the active set at each
step.  But for the lasso, this is not necessarily true, as in this
case variables can be either added or deleted at each step.

\subsection{Inference after selection}
\label{sec:conditional}

We review the selective pivotal inference approach for 
hypothesis testing after model selection
with FS, LAR, or the lasso. The technical details of the TG
statistic are deferred to the next two subsections, as they
are not needed to 
understand how the method is used.  The null hypotheses we consider
are of the form $H_0: v^T \theta = 0$.  An important special case
occurs when the  linear contrast $v^T \theta$ gives a normalized 
coefficient in the regression of $\theta$ onto a subset of the
variables in $X$.  To be specific, in this case 
\smash{$v=X_A(X_A^T X_A)^{-1} e_j/(e_j^T(X_A^T X_A)^{-1}
  e_j)^{1/2}$}, for a subset $A \subseteq \{1,\ldots,d\}$, 
where we let $X_A \in \R^{n \times |A|}$ denote the submatrix of $X$
whose columns correspond to elements of $A$ (with
\smash{$X_A^T X_A$} assumed to be invertible for the chosen subset),
and we write $e_j$ for the $j$th standard basis vector.  This gives 
\begin{equation}
\label{eq:proj_coef}
v^T \theta = \frac{e_j^T (X_A^T X_A)^{-1} X_A^T \theta}
{\sqrt{e_j^T(X_A^T X_A)^{-1} e_j}} := \beta_j(A),
\end{equation}
and therefore $H_0 : v^T \theta = 0$ is a test for the significance 
of the $j$th normalized coefficient in the linear projection of
$\theta$ onto $X_A$, written as $\beta_j(A)$ for short. 
(Though the normalization in the denominator is 
irrelevant for this significance test, it acts as a key scaling
factor for the asymptotics in Section \ref{sec:asymptotics}.) The idea
of using a projection parameter for inference, $\beta_j(A)$,
has also appeared in, e.g., \citet{berk2013valid,
wasserman2014discussion,lee2016exact}.
Here is now a summary of the testing framework.

\begin{itemize}
\item For each possible model $M \in \cM$, and any $v \in \R^n$
  and $\mu \in \R$, a TG statistic $T(\,\cdot\,;M,v,\mu)$ is defined
  (see \eqref{eq:cond_pivot_defn}, in the next subsection), whose
  domain is the partition element $\Pi_M$.
  This can be used as follows: if $Y$ is drawn from \eqref{eq:model},
  and lands in the partition element $\Pi_M$ for model $M$, then the
  statistic $T(Y;M,v,\mu)$ provides us with a test for the hypothesis  
  $H_0 : v^T \theta = \mu$.

\item A concrete case to keep in mind, denoting $M=\{(A_\ell,s_\ell)
  : \ell=1,\ldots,k\}$, is a choice of $v$ such that
  \smash{$v^T \theta = \beta_j(A_\ell)$}, in the notation of
  \eqref{eq:proj_coef}.  This is the $j$th normalized coefficient in
  the regression of $\theta$ onto the active variables
  \smash{$X_{A_\ell}$}, for an active set $A_\ell$ at some step
  $\ell=1,\ldots,k$.    

\item Assume i.i.d.\ $N(0,\sigma^2)$ errors in \eqref{eq:model}.
  Under the null hypothesis,
  the TG statistic has a standard uniform distribution,
  over draws of $Y$ that land in $\Pi_M$. Mathematically, this
  is the property    
  \begin{equation}
    \label{eq:cond_pivot}
    \P_{v^T \theta=\mu} \Big( T(Y; M,v,\mu) \leq t \, \Big| \, 
    \hat{M}(Y)=M \Big) = t,
  \end{equation}
  for all $t \in [0,1]$. The probability above is taken over an
  arbitrary mean parameter $\theta$ for which $v^T \theta=\mu$ (in   
  fact, the TG statistic is constructed so that the law of
  \smash{$T(Y; M,v,\mu) \, | \, \hat{M}(Y)=M$} only depends on 
  $\theta$ through $v^T \theta$, so this is unambiguous).  
  In order for \eqref{eq:cond_pivot} to hold, of course, $v$ and $\mu$
  cannot be random, i.e., they cannot depend on $Y$, though they can
  be functions of $M$.  

\item Thus $T(Y;M,v,\mu)$ serves as a valid p-value
  (with exact finite sample size) for testing the null hypothesis
  $H_0 : v^T \theta=\mu$, conditional on \smash{$\hat{M}(Y)=M$}.  

\item A confidence interval is obtained by inverting the test in
  \eqref{eq:cond_pivot}.  Given a desired confidence level $1-\alpha$,
  we define $C_\alpha$ to be the set of all values $\mu$ such 
  that $\alpha/2 \leq T(Y; M,v,\mu) \leq 1-\alpha/2$.
  Then, by construction, the property in \eqref{eq:cond_pivot}
  (which we reiterate, assumes i.i.d.\ $N(0,\sigma^2)$ errors) 
  translates into 
  \begin{equation}
    \label{eq:cond_interval}
    \P\Big( v^T \theta \in C_\alpha \, \Big | \, 
    \hat{M}(Y)=M \Big) = 1-\alpha.
  \end{equation}
  The interpretation of the above statement is straightforward: the
  random interval $C_\alpha$ contains the fixed parameter $v^T \theta$ 
  with probability $1-\alpha$, conditional on \smash{$\hat{M}(Y)=M$}.  
\end{itemize}

\subsection{The truncated Gaussian pivot}
\label{sec:pivot}

We now describe the truncated Gaussian (TG) pivotal quantity in
detail.  As defined in Section \ref{sec:selection}, if
we write \smash{$\hat{M}(y)$} for the  
selected model from the given algorithm (FS, LAR, or lasso), run for   
$k$ steps on $y$, 
then \smash{$\Pi_M = \{y : \hat{M}(y)=M\}$} is a convex cone, for any
fixed achieveable model $M$. Hence
\begin{equation*}
\Pi_M = \{y : \hat{M}(y)=M\} = \{y : Q_M \, y \geq 0\},
\end{equation*}
for a fixed matrix $Q_M$ (here the inequality is meant to be
interpreted componentwise).  Now to define the pivot
$T(\,\cdot\,;M,v,\mu)$  for testing $H_0: v^T \theta =\mu$,
several preliminary quantities must be introduced:
\begin{equation*}
w = \frac{Q_M \, v}{\|v\|_2^2}, 
\;\;\;
a(y;M,v) = v^T y - \min_{i : w_i > 0} \; 
\frac{(Q_M \, y)_i}{w_i}, 
\;\;\; \text{and} \;\;\;
b(y;M,v) = v^T y - \max_{i : w_i < 0} \; 
\frac{(Q_M \, y)_i}{w_i}.
\end{equation*}
The TG pivot is then defined by
\begin{equation}
\label{eq:cond_pivot_defn}
T(y; M,v,\mu) = \frac
{\displaystyle
\Phi\Bigg(\frac{b(y;M,v)-\mu}{\sigma\|v\|_2}\Bigg)- 
\Phi\Bigg(\frac{v^T y-\mu}{\sigma\|v\|_2}\Bigg)}
{\displaystyle
\Phi\Bigg(\frac{b(y;M,v)-\mu}{\sigma\|v\|_2}\Bigg) - 
\Phi\Bigg(\frac{a(y;M,v)-\mu}{\sigma\|v\|_2}\Bigg)}.
\end{equation}
This has the following property, as stated in \eqref{eq:cond_pivot}:
when $Y$ is drawn from \eqref{eq:model} with i.i.d.\ $N(0,\sigma^2)$ 
errors, and $v^T \theta=\mu$, the pivot $T(Y;M,v,\mu)$ is
uniformly distributed conditional on \smash{$\hat{M}(Y)=M$}.
See Lemmas~1 and 2 in \citet{tibshirani2016exact} for a proof of this result. 

\subsection{P-values and confidence intervals}
\label{sec:pvals}

For the null hypothesis $H_0: v^T \theta = 0$, we have seen from
\eqref{eq:cond_pivot} that $T(Y;M,v,0)$ acts as a proper
(conditional) p-value.  But as defined in \eqref{eq:cond_pivot_defn},
the statistic $T(Y;M,v,0)$ is implicitly aligned to have
power against the one-sided alternative hypothesis $H_1 : v^T \theta >
0$. 
Therefore, when seeking to test
the significance of, say, the $j$th coefficient in the projected
linear model of $\theta$ on  \smash{$X_{A_\ell}$}, we will actually
choose $v$ so that 
\begin{equation}
\label{eq:sign_ctrst}
v^T \theta = 
(s_\ell)_j \, \beta_j(A_\ell),
\end{equation}
where recall
\smash{$(s_\ell)_j =\sign(e_j^T 
  (X_{A_\ell}^T X_{A_\ell})^{-1} X_{A_\ell}^T  y)$} is the sign
of the $j$th coefficient in the regression of $y$ onto the set 
$A_\ell$ of active variables, for $y \in \Pi_M$. 
This orients the test in a meaningful direction: $v^T \theta > 0$ 
is now the hypothesis that the $j$th coefficient in the projection of 
$\theta$ onto \smash{$X_{A_\ell}$} is nonzero, and {\it shares the
  same sign} as the $j$th coefficient in the projection of $y$ onto 
\smash{$X_{A_\ell}$}, over $y \in \Pi_M$; that is, with the
above choice of $v$, the p-value $T(Y;M,v,0)$ is designed to be small
when the $j$th coefficient in the projection of $Y$ on 
\smash{$X_{A_\ell}$} corresponds to a projected population effect that
is both large and of the same sign as this computed coefficient.
Beyond the current subsection, we will not be explicit about the  
sign factor
in \eqref{eq:sign_ctrst} when discussing such contrasts (i.e.,
those giving regression coefficients in a projected linear model
for $\theta$), but it is implicitly understood when computing
one-sided p-values.      

A statistic aligned to have power against the two-sided
alternative $H_1 : v^T \theta \not= 0$ is simply given by 
$2\min\{T(Y;M,v,0),\,1-T(Y;M,v,0)\}$.  
For purely testing purposes, we find the one-sided p-values discussed 
above to be more natural, and hence these will serve as our
default.  On the other hand, for constructing confidence
intervals, we prefer to invert the two-sided statistics, since these 
lead to two-sided intervals.  As
\begin{equation*}
2\min\{T(Y;M,v,\mu),\;1-T(Y;M,v,\mu)\} \geq \alpha \iff 
\alpha/2 \leq T(Y;M,v,\mu) \leq 1-\alpha/2,
\end{equation*}
the previously described confidence interval in 
\eqref{eq:cond_interval} is just given by inverting the
two-sided pivot. 

To summarize: the default in this work, as with \citet{tibshirani2016exact}, is to
consider one-sided hypothesis tests, but two-sided intervals.  These
are just two slightly different uses of the same pivot.

\subsection{Inference after selection, revisited}
\label{sec:unconditional}

We have portrayed selective pivotal inference, in sequential
regression procedures, as a method for producing conditional p-values
and intervals. An unconditional interpretation of this framework 
is also possible, which we describe here.

\begin{itemize}
\item For each model $M \in \cM$, a contrast
  vector $v_M \in \R^n$ and pivot value $\mu_M \in \R$ are identified, 
  so that the hypothesis \smash{$H_{0,M} : v_M^T \theta = \mu_M$} is 
  to be tested whenever $y \in \Pi_M$, i.e., whenever
  \smash{$\hat{M}(y)=M$}. A  
  TG statistic $\cT(\,\cdot\,;V,U)$ is then defined, whose domain
  is the entire sample space $\R^n$.  Here we write $V=\{v_M :
    M \in  \cM\}$ and $U = \{ \mu_M : M \in \cM\}$ to denote the
    collection of contrast vectors and pivot values, respectively,
    across partition elements---we will also refer to these as {\it
      catalogs}.  This unconditional TG statistic is defined by
\begin{equation*}
\cT(\,\cdot\,;V,U) 
= \sum_{M \in \cM} T (\,\cdot\, ; M, v_M, \mu_M) \,
1_{\Pi_M}(\,\cdot\,),
\end{equation*}
where \smash{$1_{\Pi_M}(\,\cdot\,)$} denotes
the indicator function for the partition element $\Pi_M$ (and
\smash{$T(\,\cdot\,; M, v_M, \mu_M)$} is as before, defined in
\eqref{eq:cond_pivot_defn}). The unconditional statistic can be used
as follows: if a response $Y$ is drawn from \eqref{eq:model}, then  
we can form \smash{$\cT(Y;V,U) =
  T(Y;\hat{M}(Y),v_{\hat{M}(y)},\mu_{\hat{M}(y)})$} to test the hypothesis 
\smash{$H_0 : v_{\hat{M}(Y)}^T \theta = \mu_{\hat{M}(Y)}$}. 

\item A concrete case to keep in mind is when $V$ assigns a
  contrast vector $v_M$ to each model $M$, such that \smash{$v_M^T 
    \theta=\beta_{j_M}(A_{\ell_M})$}, in the notation of 
  \eqref{eq:proj_coef}, where $M = \{(A_\ell,s_\ell) : 
  \ell=1,\ldots,k\}$ as usual.  This is  
  the $j_M$th normalized coefficient from projecting $\theta$ onto
  \smash{$X_{A_{\ell_M}}$}, the active variables at step $\ell_M$.

\item Assume that the errors in \eqref{eq:model} are i.i.d.\
  $N(0,\sigma^2)$.  Then under the proper hypothesis, by summing up  
  the conditional property in \eqref{eq:cond_pivot} across partition 
  elements, we have 
\begin{equation}
\label{eq:pivot}
\P_{V^T \theta=U} \Big( \cT(Y;V,U) \leq t \Big) = t,
\end{equation}
for all $t \in [0,1]$. The assertion above holds for a parameter 
$\theta$ such that $V^T \theta = U$,
which we use as shorthand for $v_M^T \theta = \mu_M$ for all $M \in 
\cM$. 
Note that this full specification, across all $M \in \cM$, is critical
in order to apply the relevant null probability within each partition
element (giving rise to the equality in \eqref{eq:pivot}).

\item Therefore $\cT(Y;V,U)$ serves as a valid p-value (with exact 
  finite sample size)---but for testing what null hypothesis?
  Formally, it is attached to $H_0 : V^T \theta = U$, an exhaustive
  specification of $v_M^T \theta = \mu_M$, over all $M \in \cM$, but
  in truth, $\cT(Y;V,U)$ carries no information about models other
  than the selected one, \smash{$\hat{M}(Y)$}. For this reason, we  
  actually consider $\cT(Y;V,U)$ to be a p-value for the {\it random}  
  null hypothesis \smash{$H_0 : v_{\hat{M}(Y)}^T \theta =
    \mu_{\hat{M}(Y)}$}. This is made more precise through confidence
  intervals.  

\item A confidence interval is obtained by inverting the test in
  \eqref{eq:pivot}. But the TG statistic at $Y$,
\begin{equation*}
\cT(Y;V,U) =  
\sum_{M \in \cM} T(Y;M,v_M,\mu_M) 1_{\Pi_M}(Y) =  T\big(Y;\hat{M}(Y),
v_{\hat{M}(Y)}, \mu_{\hat{M}(Y)}\big), 
\end{equation*}
only depends on $U$ through \smash{$\mu_{\hat{M}(Y)}$}.  Thus, given a
desired confidence level $1-\alpha$, let us define
$D_\alpha$ to be the set of $U$ such that 
$\alpha/2 \leq \cT(Y;V,U) \leq 1-\alpha/2$, and $C_\alpha$ to be the set
of \smash{$\mu_{\hat{M}(Y)}$} such that
\smash{$\alpha/2 \leq T(Y;\hat{M}(Y),v_{\hat{M}(Y)},\mu_{\hat{M}(Y)}) \leq
1-\alpha/2$}.  Then we can see that
\begin{equation*}
U \in D_\alpha \iff \mu_{\hat{M}(Y)} \in C_\alpha,
\end{equation*}
so the confidence interval is effectively infinite with respect to 
the values $\mu_M$, \smash{$M \not= \hat{M}(Y)$}, and
inverting the test in \eqref{eq:pivot} yields
\begin{equation}
\label{eq:interval}
\P\Big( v_{\hat{M}(Y)}^T \theta \in C_\alpha\Big) = 1-\alpha.
\end{equation}
The above expression says that the random interval
$C_\alpha$ traps the random parameter \smash{$v_{\hat{M}(Y)}^T
  \theta$} with probability $1-\alpha$, and thus, this supports the  
interpretation of \smash{$H_0 : v_{\hat{M}(Y)}^T
  \theta = \mu_{\hat{M}(Y)}$} as the null hypothesis underlying
the unconditional TG statistic.  
\end{itemize}

\begin{remark}
The pivotal property in \eqref{eq:pivot} is derived under the
distributional assumption that $V^T \theta = U$, i.e., $v_M^T
\theta=\mu_M$ for all $M \in \cM$, which may seem unnatural, 
as the catalog $U$ of pivot value can be large (e.g., on the order of 
$d^k$ after $k$ steps of FS), and so this is condition on possibly many  
contrasts of $\theta$.  However, it is worth emphasizing that the 
unconditional testing property in \eqref{eq:pivot} is really only useful in that
it allows us to formulate the unconditional confidence interval property in  
\eqref{eq:interval}, which is a more natural statement about coverage of a
single (random) parameter. When viewing selective inference
from an unconditional perpsective, we find it more natural to place the focus
on confidence intervals rather than hypothesis testing; in many ways, we find
the former the more natural of the two perspectives, unconditionally.   
\citet{tibshirani2016exact} in fact suggest separate nomenclature for the
unconditional case, referring to the property in \eqref{eq:interval}
as that of a {\it selection interval} (rather than confidence interval), to
emphasize that this interval covers a moving target.    

\end{remark}

\section{The master statistic}
\label{sec:master}

Given a response $y$ and predictors $X$, our description thus far of
the selected model $\hat{M}(y)$, statistics $T(y;M,v,\mu)$ and 
$\cT(y;V,U)$, etc., has ignored the role of $X$.    
This was done for simplicity. The theory to come in Section 
\ref{sec:asymptotics} will consider $X$ to be nonrandom, but
asymptotically $X$ must (of course) grow with $n$, and so it will help 
to be precise about the dependence of the selected model and
statistics on $X$. We will denote these quantities by  
\smash{$\hat{M}(X,y)$}, \smash{$T(X,y;M,v,\mu)$}, and \smash{$\cT(X,y;
  V,U)$} to emphasize this dependence.  We define
\begin{equation*}
\Omega_n = \Big(\frac{1}{n}X^T X, \frac{1}{\sqrt{n}}X^T y\Big),
\end{equation*}
a $d(d+3)/2$-dimensional quantity that we will call
the {\it master statistic}.  As its name might suggest, this plays an 
important role: all normalized coefficients from regressing
$y$ onto subsets of the variables $X$ can be written in terms of
$\Omega_n$.  That is, 
for an arbitrary set $A \subseteq \{1,\ldots,p\}$, the
$j$th normalized coefficient from the regression of $y$ onto $X_A$ is 
\begin{equation*}
\frac{(e_j^T X_A^T X_A)^{-1} X_A^T y}{\sqrt{e_j^T (X_A^T X_A)^{-1} e_j}} =
\frac{e_j^T n(X_A^T X_A)^{-1} \frac{1}{\sqrt{n}} X_A^T y}{\sqrt{e_j^T
    n(X_A^T X_A)^{-1} e_j}}, 
\end{equation*}
which only depends on $(X,y)$ through $\Omega_n$. The same dependence 
is true, it turns out, for the selected models from FS, LAR, and the
lasso. We defer the proof of the next lemma, as with all proofs in
this paper, until the appendix. 

\begin{lemma}
\label{lem:master1}
For each of the FS, LAR, and lasso procedures, run for $k$ steps on 
data $(X,y)$, the selected model \smash{$\hat{M}(X,y)$} only
depends on $(X,y)$ through 
\smash{$\Omega_n=(\frac{1}{n}X^T X, \frac{1}{\sqrt{n}}X^T y)$}, 
the master statistic.

In more detail, for any fixed $M \in \cM$, the matrix $Q_M(X)$ 
such that
\smash{$\hat{M}(X,y)=M \iff Q_M(X) \, y \geq 0$} can be written
as \smash{$Q_M(X) = P_M(\frac{1}{n} X^T X) \, \frac{1}{\sqrt{n}}
   X^T$},  where $P_M$ depends only on 
\smash{$\frac{1}{n} X^T  X$}. Hence     
\begin{equation*} 
\hat{M}(X,y)=M \iff P_M\Big(\frac{1}{n}X^T X\Big) \, 
\frac{1}{\sqrt{n}} X^T y \geq 0.
\end{equation*}
\end{lemma}

This lemma asserts that the master statistic governs model selection, 
as performed by FS, LAR, and the lasso. It is also central to TG pivot
for these procedures. Denoting  
\smash{$M=\hat{M}(X,y)$}, the statistic $T(X,y;M,v,\mu)$ in
\eqref{eq:cond_pivot_defn} only depends on $(X,y)$ through three
quantities: 
\begin{equation*}
\frac{v^T y}{\|v\|_2}, 
\;\;\; \frac{Q_M(X) \, v}{\|v\|_2}, 
\;\;\;\text{and}\;\;\;
Q_M(X) \, y.
\end{equation*}
The third quantity is always a function of $\Omega_n$, by Lemma
\ref{lem:master1}.  When $v$ is chosen so that $v^T y$  
is a normalized coefficient in the regression of $y$ onto a
subset of the variables in $X$, the first two quantities are also 
functions of $\Omega_n$.  Thus, in this case, the TG pivot only 
depends on $(X,y)$ through the master statistic $\Omega_n$; in fact, it is
continuous at any point such that \smash{$\frac{1}{n} X^T X$} is nonsingular and
$y$ does not lie on the boundary of a model selection event. 

\begin{lemma}
\label{lem:master2}
Fix any model $M \in \cM$, and suppose that $v$ is chosen so  
that $v^T y$ is a normalized coefficient from projecting 
$y$ onto a subset of the variables in $X$.  Then the TG statistic 
only depends on $(X,y)$ by means of $\Omega_n$, so that we may   
write 
\begin{equation*}
T(X,y;M,v,\mu) = \psi_M \Big(\frac{1}{n}X^T X, 
\frac{1} {\sqrt{n}}X^T y\Big).   
\end{equation*}
Further, the function $\psi_M$ is continous at any point $(S,z)$ such that $S$ 
is nonsingular and $P_M(S) \, z > 0$.
\end{lemma}

Finally, we show that the conditional pivotal property of the TG 
statistic in \eqref{eq:cond_pivot} can be phrased entirely in terms of the
master statistic.

\begin{lemma}
\label{lem:master3}
Assume the conditions of Lemma \ref{lem:master2}, and additionally that $Y$ 
is drawn from \eqref{eq:model}. Construct the master statistic
\smash{$\Omega_n=(\frac{1}{n}X^T X,\frac{1}{\sqrt{n}} X^T Y)$}.  Then there is a
function $g$ such that
\begin{equation*}
v^T \theta 
= g(\E(\Omega_n)). 
\end{equation*}
Thus if the errors in \eqref{eq:model} are i.i.d.\
$N(0,\sigma^2)$, then the conditional pivotal property 
\eqref{eq:cond_pivot} of the TG statistic can be reexpressed as
\begin{equation*}
\P_{g(\E(\Omega_n))=\mu} \Big( \psi_M(\Omega_n) \leq t  
  \, \Big| \, \hat{M}(X,Y)=M \Big) = t,
\end{equation*}
for all $t \in [0,1]$.  
\end{lemma}

Equipped with the last two lemmas, asymptotic 
theory for the TG test, when $d$ is fixed, is not far off.   
Under weak conditions on the data model in \eqref{eq:model},  
the central limit theorem tells us that
\smash{$\frac{1}{\sqrt{n}} X^T Y$} converges weakly to a
normal random variable. With \smash{$\frac{1}{n}X^T X$} converging to
a deterministic matrix, the continuous mapping theorem 
will then provide the appropriate asymptotic limit for the 
statistic \smash{$T(X,y;M,v,\mu)=\psi_M(\frac{1}{n}X^T X, 
\frac{1}{\sqrt{n}}X^T Y)$}.   
This is made more precise next.

\section{Asymptotic theory}
\label{sec:asymptotics}

Here we treat the dimension $d$ as fixed, and consider the limiting
distribution of the TG statistic as $n\to \infty$.  (See Section
\ref{sec:highdim} for the case when $d$ grows.)
Throughout, the matrix $X \in \R^{n\times d}$ will be
treated as nonrandom, and we consider a sequence of predictor matrices
satisfying two conditions:  
\begin{equation}
\label{eq:xcov}
\lim_{n\to\infty} \; \frac{1}{n} X^T X = \Sigma,
\end{equation}
for a nonsingular matrix $\Sigma \in \R^{d\times d}$, and
\begin{equation}
\label{eq:xnorm}
\lim_{n\to\infty} \; \max_{i=1,\ldots,n} \; \frac{\|x_i\|_2}{\sqrt{n}}
= 0,
\end{equation}
where $x_i \in \R^d$, $i=1,\ldots,n$ denote the rows of $X$.
These are not strong conditions.
  
\subsection{A nonparametric family of distributions}
\label{sec:dists}

We specify the class of distributions that we will be working with for
$Y$ in \eqref{eq:model}. Let $\sigma^2>0$ be a fixed, known constant.
First we define a set of error distributions    
\begin{equation*}
\cE = \Bigg\{F \;:\;
\int x dF(x)=0, \; 
\int x^2 dF(x)=\sigma^2 \Bigg\}.
\end{equation*}
The first moment condition in the above definition is needed to make
the model identifiable, and the second condition is used for simplicity.
Aside from these moment conditions, the class $\cE$ contains 
a small neighborhood (say, as measured in the total variation metric)
around essentially every element. Thus, modulo the moment
assumptions, $\cE$ is strongly nonparametric in the sense of 
\citet{donoho1988onesided}. Given $\mu \in \R$, let $F_\mu$
denote the distribution of $\mu+ \delta$, where $\delta \sim F$, and
given $\theta=(\theta_1,\ldots,\theta_n) \in \R^n$, let 
\smash{$F_n(\theta)=F_{\theta_1} \times \ldots,\times 
  F_{\theta_n}$}.  Now we define a class of distributions  
\begin{equation}
\label{eq:pn}
\cP_n(\theta)= \Bigg\{ F_n(\theta) =
F_{\theta_1} \times \ldots \times F_{\theta_n} \;:\;  
F \in \cE \Bigg\}.
\end{equation}
In words, assigning a distribution $Y \sim F_n(\theta)$ means that
$Y$ in drawn from the model \eqref{eq:model}, with mean $\theta
\in \R^n$, and errors $\epsilon_1,\ldots,\epsilon_n$ i.i.d.\ from an 
arbitrary centered distribution $F$ with variance $\sigma^2$.  

As $n$ grows, we allow the underlying mean $\theta$ to change, but we 
place a restriction on this parameter so that it has an appropriate
asymptotic limit.  Specifically, we consider a class $\Theta$ of
sequences of mean parameters such that \smash{$\frac{1}{\sqrt{n}} X^T \theta$}
has an asymptotic limit lying in some compact set, with uniform
convergence to this limit.  Formally, write (in a slight abuse of notation)
$\theta \in \Theta$ to denote a sequence of mean parameters in $\Theta$, and
let  $E(\Theta)$ denote the set of limit points of 
\smash{$\{\frac{1}{\sqrt{n}} X^T \theta : \theta \in \Theta\}$}. 
Then, for some constant $B>0$, we require of the class $\Theta$,
\begin{equation}
\label{eq:theta}
E(\Theta) \subseteq [-B,B]^d, \;\;\;\text{and}\;\;\;
\lim_{n\to\infty} \; \sup_{\eta \in E(\Theta)} \; 
\sup_{\frac{1}{\sqrt{n}} X^T \theta \to \eta} \; 
\Bigg|\frac{1}{\sqrt{n}} X^T \theta - \eta \Bigg| = 0.
\end{equation}
We emphasize once again that $\theta\in\R^n$ and $X \in \R^{n\times p}$ will
both vary with $n$, i.e., we can think of $\theta$ and the columns of $X$ as
triangular arrays, but our notation suppresses this dependence for simplicity. 

\subsection{Uniform convergence results}

We begin with a result on the uniform convergence of (the random part
of) the master statistic to a normal distribution, both marginally and
conditionally.   

\begin{lemma}
\label{lem:master}
Assume that $X$ has asymptotic covariance matrix $\Sigma$, as in
\eqref{eq:xcov}, and satisfies the normalization condition
in \eqref{eq:xnorm}. Let $Y \sim F_n(\theta) \in \cP_n(\theta)$, this
class as defined in \eqref{eq:pn}, for a sequence of mean parameters 
$\theta \in \Theta$, as defined in \eqref{eq:theta}.  Denote
\smash{$\frac{1}{\sqrt{n}} X^T \theta \to \eta$} as $n \to \infty$.  Then   
\smash{$Z_n=\frac{1}{\sqrt{n}} X^T Y$}
converges in distribution to $Z \sim N(\eta,\sigma^2\Sigma)$, 
uniformly over $\cP_n(\theta)$, and uniformly over all $\theta \in \Theta$.
That is,   
\begin{equation*}
\lim_{n\to\infty} \; \sup_{\theta \in \Theta} \;
\sup_{F_n(\theta) \in \cP_n(\theta)} \; \sup_{x \in \R^d} \;
\big|\P(Z_n \leq x) - \P(Z \leq x)\big| = 0.
\end{equation*}
Further, given a sequence of matrices $A_n \in \R^{q \times d}$,
$n=1,2,3,\ldots$ with $A_n \to A$ as $n \to \infty$, such that the set
\smash{$\{z : Az \geq 0\}$} has nonempty interior,  
\smash{$Z_n \,|\, A_nZ_n
  \geq 0$} converges in distribution to $Z \,|\, AZ \geq 0$, uniformly over 
$\cP_n(\theta)$, and uniformly over all $\theta \in \Theta$.
\end{lemma}

This lemma, combined with Lemmas \ref{lem:master2} and \ref{lem:master3} of the
last section, leads us to uniform asymptotic theory for the TG
test. We remind the reader that $k$, the number of steps, is to be
considered fixed in the next result (as it is throughout the paper).   

\begin{theorem}
\label{thm:pivot}
Assume the conditions of Lemma \ref{lem:master}.  Suppose FS, LAR, or    
the lasso is run for $k$ steps on $(X,Y)$. Below we describe the conditional 
and unconditional asymptotic results separately.

\smallskip\smallskip
\noindent
(a, Markovic) Fix any model $M \in \cM$. Let $v$ be a vector such that $v^T     
\theta$ gives a normalized coefficient in the projection of
$\theta$ onto some subset of the variables in $X$, and let $\mu$ be an 
arbitrary pivot value.  Then  
under $v^T \theta=\mu$, the conditional TG statistic 
\smash{$T(X,Y;M,v,\mu) \,|\, \hat{M}(X,Y)=M$} converges in distribution to 
$W \sim U(0,1)$, uniformly over $\cP_n(\theta)$, and over $\theta \in
\Theta$. That is,             
\begin{equation*}
\lim_{n \to \infty} \;
\sup_{\theta \in \Theta} \;
\sup_{F_n(\theta) \in \cP_n(\theta)} \;
\sup_{t \in [0,1]} \;
\Bigg|\P_{v^T \theta=\mu} \Big(T(X,Y;M,v,\mu) \leq t 
\,\Big|\, \hat{M}(X,Y)=M \Big) - t\Bigg| = 0.   
\end{equation*}
Moreover, if we define $C_{n,\alpha}$ to be the set
of $\mu$ such that \smash{$\alpha/2 \leq
  T(X,Y;M,v,\mu) \leq 1-\alpha/2$},
then $C_{n,\alpha}$ is an asymptotically uniformly valid confidence 
interval for \smash{$v^T \theta$}.  That is,
\begin{equation*}
\lim_{n \to \infty} \;
\sup_{\theta \in \Theta} \;
\sup_{F_n(\theta) \in \cP_n(\theta)} \;
 \sup_{\alpha \in [0,1]} \;
\Bigg|\P_{v^T \theta=\mu} \Big(v^T \theta \in C_{n,\alpha} 
\,\Big|\, \hat{M}(X,Y)=M \Big) - (1-\alpha) \Bigg| = 0. 
\end{equation*}

\smallskip\smallskip
\noindent
(b) Let $V=\{v_M : M \in \cM\}$ 
be a catalog of vectors such that each $v_M^T \theta$ yields a
normalized coefficient in the projection of $\theta$ onto a subset 
of the variables in $X$, for $M \in \cM$, and 
$U=\{\mu_M : M \in \cM\}$ be a catalog of pivot
values.  Then under $V^T \theta = U$, the same results as in part (a) hold
marginally.  That is,  
\begin{equation*}
\lim_{n \to \infty} \;
\sup_{\theta \in \Theta} \;
\sup_{F_n(\theta) \in \cP_n(\theta)} \;
\sup_{t \in [0,1]} \;
\Big|\P_{V^T \theta=U} \Big(\cT(X,Y;V,U) \leq t\Big) - t\Big| = 0.  
\end{equation*}
and for $C_{n,\alpha}$ defined to be the set
of $\mu$ such that \smash{$\alpha/2 \leq
  T(X,Y;\hat{M}(X,Y),v_{\hat{M}(X,Y)},\mu) \leq 1-\alpha/2$},
\begin{equation*}
\lim_{n \to \infty} \;
\sup_{\theta \in \Theta} \;
\sup_{F_n(\theta) \in \cP_n(\theta)} \;
 \sup_{\alpha \in [0,1]} \;
\Big|\P\Big(v_{\hat{M}(X,Y)}^T \theta \in C_{n,\alpha}\Big) - (1-\alpha)  
\Big| = 0.
\end{equation*}
\end{theorem}

\begin{remark}
An initial version of this work contained only the unconditional result in part
(b) of the theorem.  Jelena Markovic pointed out that the conditional
result in part (a) should also be possible, and thus this conditional
result should also be attributed to her. Between the initial and the
current version of this paper, in addition to revising Theorem \ref{thm:pivot},
we have also revised Theorems \ref{thm:unknown} and \ref{thm:highdim} to include
the appropriate conditional results.  
\end{remark}

\section{Unknown $\sigma^2$ and the bootstrap}
\label{sec:unknown}

The results of the previous section assumed that the error variance 
$\sigma^2$ in the model \eqref{eq:model} was known.  Here we consider
two strategies when $\sigma^2$ is unknown.  The first plugs a
(rather naive) estimate of $\sigma^2$ into the usual TG
statistic. The second is a computationally efficient bootstrap
method. Both, as we will show, yield asymptotically conservative
p-values.   (In practice, the bootstrap often gives shorter
confidence intervals than those based on the TG pivot; see
Section \ref{sec:examples}.)

\subsection{A simple plug-in approach}
\label{sec:plugin}

Given a model $M \in \cM$, contrast vector $v$, and pivot value $\mu$,
consider the TG statistic $T(X,Y;M,v,\mu)$. Let us abbreviate 
\begin{equation*}
\hat{a}_M=a(X,Y;M,v),
\;\;\;\text{and} \;\;\;
\hat{b}_M=b(X,Y;M,v),
\end{equation*}
where the latter two functions are as defined in Section
\ref{sec:pivot}.  In this notation, we can succintly
write the TG statistic as
\begin{equation}
\label{eq:pivot_succint}
T(X,Y;M,v,\mu) = \frac
{\displaystyle
\Phi\Bigg(\frac{\hat{b}_M-\mu}{\sigma\|v\|_2}\Bigg)-  
\Phi\Bigg(\frac{v^T Y-\mu}{\sigma\|v\|_2}\Bigg)} 
{\displaystyle
\Phi\Bigg(\frac{\hat{b}_M-\mu}{\sigma\|v\|_2}\Bigg)-
\Phi\Bigg(\frac{\hat{a}_M-\mu}{\sigma\|v\|_2}\Bigg)}.
\end{equation}
When $\sigma^2$ is unknown, we propose a simple plug-in approach that 
replaces $\sigma$ with $cs_Y$, where
\begin{equation*}
s_Y^2 = \frac{1}{n}\sum_{i=1}^n |Y_i-\bar{Y}|^2, 
\end{equation*}
the sample variance of $Y$ (here \smash{$\bar{Y}=\sum_{i=1}^n
  Y_i/n$} denotes the sample mean), and $c>1$ is a fixed constant.
To be explicit, we consider the modified TG statistic
\begin{equation}
\label{eq:pivot_plugin}
\tilde{T}(X,Y;M,v,\mu) = \frac
{\displaystyle
\Phi\Bigg(\frac{\hat{b}_M-\mu}{cs_Y\|v\|_2}\Bigg)-  
\Phi\Bigg(\frac{v^T Y-\mu}{cs_Y\|v\|_2}\Bigg)} 
{\displaystyle
\Phi\Bigg(\frac{\hat{b}_M-\mu}{cs_Y\|v\|_2}\Bigg)-
\Phi\Bigg(\frac{\hat{a}_M-\mu}{cs_Y\|v\|_2}\Bigg)}.  
\end{equation}
The scaling factor $c$ facilitates our theoretical study of the above 
plug-in statistic, and practically, we have found that ignoring it
(i.e., setting $c=1$) works perfectly well, though a choice of, say,
$c=1.0001$ seems to have a minor effect anyway.  

When the mean $\theta$ of $Y$ is nonzero, the sample variance $s_Y^2$
is generally too large as an estimate of $\sigma^2$. As we will show, 
the modified statistic in \eqref{eq:pivot_plugin} thus yields
asymptotically conservative p-values.  Residual based estimates of
$\sigma^2$ are not as useful in our setting because they depend
more heavily on the linearity of the underlying regression model,
and they suffer practically when $d$ is close to $n$ (see also the 
discussion at the start of Section \ref{sec:examples}).

\subsection{An efficient bootstrap approach}
\label{sec:bootstrap}

As an alternative to the plug-in method of the last subsection, we
investigate a highly efficient bootstrap scheme that does not rely on
knowledge of $\sigma^2$. Our general framework so far treats $X$ as
fixed, and for our bootstrap strategy to respect this assumption, we
cannot use, say, the pairs bootstrap, and must perform sampling with
respect to $Y$ only.  The residual bootstrap is ruled out since we
do not assume that the mean $\theta$ follows a linear model in
$X$.  This leaves us to consider simple bootstrap sampling of the
components of $Y$.  This is somewhat nonstandard, as the
components of $Y$ in \eqref{eq:model} are not i.i.d., but it provides
a mechanism for provably conservative asymptotic inference, and it is 
what makes our approach so computationally efficient.


Given $Y=(Y_1,\ldots,Y_n)$ drawn from the model in \eqref{eq:model},
let $Y^*=(Y_1^*,\ldots,Y_n^*)$ denote a bootstrap sample of $Y$.  We
will denote by $\P_*$ the conditional distribution of $Y^*$ on $Y$,
and $\E_*$ the associated expectation operator.  That is, $\P_*(Y^*
\in A)$ is shorthand for $\P(Y^* \in A|Y)$, and similarly for $\E_*$.   
Using the notation of the last subsection (notation for
\smash{$\hat{a}_M,\hat{b}_M$}), and assuming without a loss
of generality that \smash{$\|v\|_2=1$}, let us motivate our 
bootstrap proposal by expressing the TG statistic as
\begin{equation*}
T(X,Y;M,v,\mu) = \P\Big(Z_{\mu,\sigma^2} \geq v^T Y \, \Big|  
\, \hat{a}_M \leq Z_{\mu,\sigma^2} \leq \hat{b}_M, \, Y \Big), 
\end{equation*}
where the probability on the right-hand side is taken
with $Y$ (and thus \smash{$\hat{a}_M,\hat{b}_M$}) 
treated as fixed, and with \smash{$Z_{\mu,\sigma^2}$} denoting a 
\smash{$N(\mu,\sigma^2)$} random variable.  The main idea is now
to approximate the truncated normal distribution underlying 
the TG statistic with an appropriate one from bootstrap samples,
\begin{equation*}
\P\Big(Z_{\mu,\sigma^2} \geq v^T Y \, \Big| 
\, \hat{a}_M \leq Z_{\mu,\sigma^2}  \leq \hat{b}_M, \, Y \Big) \approx 
\P_*\Big( v^T (Y^* - \bar{Y} \mathbb{1}) + \mu \geq 
v^T Y  \, \Big| \,  \hat{a}_M \leq v^T (Y^* - \bar{Y}
\mathbb{1}) + \mu \leq \hat{b}_M \Big).     
\end{equation*}
Recall \smash{$\bar{Y}=\sum_{i=1}^n Y_i/n$} is the sample mean of $Y$,
so $\E_*(v^T Y^*)= v^T (\bar{Y} \mathbb{1})$ (with
$\mathbb{1} \in \R^n$ denoting the vector of all 1s), and we 
have shifted \smash{$v^T Y^*$} so that the resulting quantity
\smash{$v^T (Y^* - \bar{Y} \mathbb{1}) + \mu$} mimics
a normal variable with mean $\mu$.
The right-hand side above very nearly defines our bootstrap version of
the TG statistic, except that for technical reasons, we must make
two small modifications.  In particular, we define the bootstrap TG  
statistic as      
\begin{equation}
\label{eq:pivot_boot}
T^*(X,Y;M,v,\mu) = 
\frac{\P_*\big(v^T Y \leq c v^T 
(Y^* - \bar{Y} \mathbb{1}) + \mu  
  \leq \hat{b}_M\big) + \delta_n}   
{\P_*\big(\hat{a}_M \leq c v^T 
(Y^* - \bar{Y} \mathbb{1}) + \mu 
  \leq \hat{b}_M\big) + \delta_n},
\end{equation}
where $c>1$ is a constant as before, and $\delta_n=\gamma
n^{-1/4}$ for a small constant $\gamma>0$.  Again, we have
found that ignoring the scaling factor $c$ (i.e., setting $c=1$) works
just fine in practice, though a choice like $c=1.0001$ does not cause
major differences anyway. On the contrary, a
nonzero choice of the padding factor like 
$\delta_n=10^{-4} n^{-1/4}$ 
does play an important practical role, since the bootstrap 
probabilities in the numerator and denominator in \eqref{eq:pivot_boot}   
can sometimes be zero.

Lastly, it is worth emphasizing that practical estimation of
the bootstrap probabilities appearing in \eqref{eq:pivot_boot} is quite 
an easy computational task, because the regression procedure in
question, be it FS, LAR, or the lasso, need not be rerun beyond its
initial run on the observed $Y$.  After this initial run, we
can just save the realized quantities \smash{$\hat{a}_M,\hat{b}_M$}, and then 
draw, say, $B=1000$ bootstrap samples $Y^*$ in order to estimate the  
probabilities in \eqref{eq:pivot_boot}.  This is not at all
computationally expensive.  Moreover, to estimate
\eqref{eq:pivot_boot} over multiple trial values of $\mu$ (so 
that we can invert these bootstrap p-values for a
bootstrap confidence interval), only a single common set of bootstrap 
samples is needed, since we can just shift \smash{$v^T Y^*$}
appropriately for each bootstrap sample $Y^*$. 

\subsection{Asymptotic theory for unknown $\sigma^2$}
\label{sec:asymp_sigma}

Treating the dimension $d$ as fixed, we will assume the
previous limiting conditions \eqref{eq:xcov},
\eqref{eq:xnorm} on the matrix $X$, and additionally, that 
\begin{equation}
\label{eq:xthree}
\frac{1}{n} \sum_{i=1}^n \|x_i\|_2^3 = O(1).
\end{equation}
Note that \eqref{eq:xcov} already implies that
\smash{$\frac{1}{n}\sum_{i=1}^n \|x_i\|_2^2 \to \tr(\Sigma)$}, and the
above is a little stronger, though it is still not a strong
condition by any means. For example, it is satisfied when
\smash{$\max_{i=1,\ldots,n} \|x_i\|_2 = O(1)$}.  These conditions on
$X$ imply important scaling properties for our usual choices of 
contrast vectors.

\begin{lemma}
\label{lem:vthree}
Assume that $X$ satisfies \eqref{eq:xcov}, \eqref{eq:xnorm},
\eqref{eq:xthree}.  If $v$ is any vector such that $v^T \theta$ gives a
normalized regression coefficient from projecting $\theta$ onto some subset of
the variables in $X$, then 
\begin{equation*}
\|v\|_3^3 = O\Big(\frac{1}{\sqrt{n}}\Big). 
\end{equation*}
\end{lemma}

We specify assumptions on the distribution of $Y$ in 
\eqref{eq:model} that are similar to (but slightly stronger than) 
those in Section \ref{sec:dists}. For
constants $\sigma^2,\tau,\kappa > 0$, we define a set of error
distributions        
\begin{equation*}
\cE' = \Bigg\{F \;:\;
\int x dF(x)=0, \; 
\int x^2 dF(x)=\sigma^2, \;
\int x^3 dF(x) \leq \tau, \;
\int x^4 dF(x) \leq \kappa \Bigg\}.
\end{equation*}
We also define a class of distributions  
\begin{equation}
\label{eq:pn2}
\cP'_n(\theta)= \Bigg\{ F_n(\theta) =
F_{\theta_1} \times \ldots \times F_{\theta_n} \;:\;  
F \in \cE' \Bigg\}.
\end{equation}
where as before, $F_\mu$ denotes the
distribution of $\mu+\delta$, for $\delta \sim F$.  We define a class
$\Theta'$ of sequences of mean parameters that satisfies, as before, 
\begin{equation}
\label{eq:theta2a}
E(\Theta') \subseteq [-B,B]^d, \;\;\;\text{and}\;\;\;
\lim_{n\to\infty} \; \sup_{\eta \in E(\Theta')} \; 
\sup_{\frac{1}{\sqrt{n}} X^T \theta \to \eta} \; 
\Bigg|\frac{1}{\sqrt{n}} X^T \theta - \eta \Bigg| = 0,
\end{equation}
for a constant $B>0$, where recall $E(\Theta')$ denotes the set of limit points
in $\Theta'$; also, for each $\theta \in \Theta'$, at each $n$, we require 
\begin{equation}
\label{eq:theta2b}
s_\theta^2 = \frac{1}{n} \sum_{i=1}^n
|\theta_i-\bar\theta|^2 \leq S, \;\;\;\text{and}\;\;\;
r_\theta^3 = \frac{1}{n} \sum_{i=1}^n 
|\theta_i-\bar\theta|^3,
\end{equation}
for constants $S,R>0$, where \smash{$\bar\theta=\sum_{i=1}^n \theta_i/n$}. 
Note that the assumptions $Y \sim F_n(\theta)$, with $F_n(\theta) \in 
\cP'_n(\theta)$ and $\theta \in \Theta'$, are not much stronger than our
assumptions in Section \ref{sec:dists}: we require the existence of two
more moments for the error distribution, and place an additional weak
condition on the growth of (components of) $\theta$.  These conditions are
sufficient to prove the following helpful lemma. 

\begin{lemma}
\label{lem:rysy}
Assume that $X$ satisfies \eqref{eq:xcov}, \eqref{eq:xnorm}.
Let $Y \sim F_n(\theta) \in \cP'_n(\theta)$, where this class 
is as defined in \eqref{eq:pn2}, and let $\theta \in \Theta'$,  
where this class is as in \eqref{eq:theta2a}, \eqref{eq:theta2b}. Then for any
fixed $M \in \cM$, and $c>1$, 
\begin{equation*}
\lim_{n \to \infty} \;
\sup_{\theta \in \Theta'} \; \sup_{F_n(\theta) \in \cP'_n(\theta)} \; 
\P \Big(c s_Y \geq \sigma \,\Big|\, \hat{M}(X,Y)=M \Big) = 1.  
\end{equation*}
In words, the event $\{c s_Y \geq \sigma\}$ has probability 
tending to 1 conditional on \smash{$\hat{M}(X,Y)=M$}, uniformly over 
$\cP'_n(\theta)$, and over $\theta \in \Theta'$. Furthermore, denoting the
sample third moment of $Y$ as  
\begin{equation*}
r_Y^3 = \frac{1}{n}\sum_{i=1}^n |Y_i-\bar{Y}|^3,
\end{equation*}
we have that for any $\delta>0$, there exists $C>0$ such that for sufficiently 
large $n$, 
\begin{equation*}
\sup_{\theta \in \Theta'} \; \sup_{F_n(\theta) \in \cP'_n(\theta)} \; 
\P \Bigg(\frac{r_Y^3}{s_Y^3} \geq C \,\Bigg|\, \hat{M}(X,Y)=M
\Bigg) \leq \delta, 
\end{equation*}
In words, $r_Y^3/s_Y^3 = O_\P(1)$ 
conditional on \smash{$\hat{M}(X,Y)=M$}, 
uniformly over $\cP'_n(\theta)$, and over $\theta \in \Theta'$. 
\end{lemma}

The last two lemmas allow us to tie the distribution function of our 
bootstrap contrast to that of a normal random variable.

\begin{lemma}
\label{lem:basic}
Assume that $X$ satisfies \eqref{eq:xcov}, \eqref{eq:xnorm},
\eqref{eq:xthree}. Let $Y \sim F_n(\theta) \in \cP'_n(\theta)$, as defined in
\eqref{eq:pn2}, and let $\theta \in \Theta'$, as defined in \eqref{eq:theta2a},
\eqref{eq:theta2b}. Let $M \in \cM$, and let $v$ be such that $v^T \theta$ gives
a normalized regression coefficient from projecting $\theta$ onto a subset of 
the variables in $X$. Then for any $\delta>0$, there exists $C>0$ such that  
sufficiently large $n$,
\begin{equation*}
\sup_{\theta \in \Theta'} \; \sup_{F_n(\theta) \in \cP'_n(\theta)} \;
\P \Bigg( \sup_{t \in \R} \; \big| \P_*\big(
v^T (Y^* - \bar{Y}\mathbb{1}) \leq t\big)     
- \P\big( s_Y  Z \leq t \, \big| \, Y \big)\big| \geq
\frac{C}{\sqrt{n}} \,\Bigg|\, \hat{M}(X,Y)=M \Bigg) \leq \delta, 
\end{equation*}
where we use $Z \sim N(0,1)$ for a standard normal random variate.   
In words, \smash{$\sup_{t \in \R} | 
\P_*(v^T (Y^* - \bar{Y}\mathbb{1}) \leq t)    
-{}$} \smash{$\P( s_Y  Z \leq t \, | \, Y )|=O_\P(1/\sqrt{n})$}   
conditional on \smash{$\hat{M}(X,Y)=M$}, uniformly
over $\cP'_n(\theta)$, and over $\theta \in \Theta'$. 
\end{lemma}

We are now ready to present uniform asymptotic results for the plug-in and 
bootstrap TG statistics. We remind the reader the number of steps $k$ is treated
as fixed below (as it is throughout).

\begin{theorem}
\label{thm:unknown}
Assume the conditions of Lemma \ref{lem:basic}.  Suppose FS, LAR, or    
the lasso is run for $k$ steps on $(X,Y)$.  Then under 
$v^T \theta = 0$, the conditional plug-in TG statistic 
\smash{$\tilde{T}(X,Y;M,v,0) \,|\, \hat{M}(X,Y)=M$} and conditional bootstrap
TG statistic \smash{$T^*(X,Y;M,v,0) \,|\, \hat{M}(X,Y)=M$} are each
asymptotically larger than $U(0,1)$ in distribution, uniformly over  
$\cP'_n(\theta)$, and over $\theta \in \Theta'$.  That is,  

\begin{equation*}
\lim_{n\to\infty} \; 
\sup_{\theta \in \Theta'} \; 
\sup_{F_n(\theta) \in \cP'_n(\theta)} \; \sup_{t \in [0,1]} \;  
\Big[\P_{v^T \theta = 0} \Big( \tilde{T}(X,Y; M,v, 0) \leq t
 \,\Big|\, \hat{M}(X,Y)=M \Big) - t \Big]_+ = 0, 
\end{equation*}
and  
\begin{equation*}
\lim_{n\to\infty} \; 
\sup_{\theta \in \Theta'} \; 
\sup_{F_n(\theta) \in \cP'_n(\theta)} \; \sup_{t \in [0,1]} \;  
\Big[\P_{v^T \theta = 0} \Big( T^*(X,Y; M,v, 0) \leq t
 \,\Big|\, \hat{M}(X,Y)=M \Big) - t \Big]_+ = 0, 
\end{equation*}
where $x_+=\max\{x,0\}$ denotes the positive part of $x$.  
Further, given any catalog $V=\{\mu_M : M \in \cM\}$ of vectors such that each 
$v_M^T \theta$ yields a normalized coefficient in the projection of $\theta$
onto a subset of the variables in $X$, for $M \in \cM$, the same results hold 
marginally under $V^T \theta = 0$.
\end{theorem}

\begin{remark} 
For simplicity, we analyzed the plug-in and 
bootstrap statistics simultaneously.  Consequently, the conditions   
assumed to prove asymptotic properties of the plug-in 
approach are stronger than what we would need if we were to study 
this method on its own, but there are not major differences in these 
conditions.
\end{remark}

Theorem \ref{thm:unknown} establishes that the plug-in and bootstrap 
versions of the TG statistic are asymptotically conservative when
viewed as p-values under $v^T \theta = 0$. If we look more broadly at  
the distribution of these test statistics under $v^T \theta =\mu$, for an
arbitrary value of $\mu$, then a technical barrier arises. 
For each statistic, our proof of its asymptotic conservativeness
leverages the fact that the truncated Gaussian survival function
decreases (in a pointwise sense), as its 
underlying variance parameter decreases.    
To extend these results to the case
of an arbitrary pivot value $\mu$, we would need the analogous 
fact to hold when we replace the survival function of the 
Gaussian variate $cs_Y Z + \mu$ truncated to
\smash{$[\hat{a}_M,\hat{b}_M]$}, with that of $\sigma Z + \mu$
tuncated to \smash{$[\hat{a}_M,\hat{b}_M]$}, on the event 
$\{c s_Y \geq \sigma\}$. 
Yet, without the guarantee that $\hat{a}_M \geq \mu$ (which 
clearly cannot always be true, for an arbitrary value of $\mu$),
it is no longer the case that decreasing the variance from 
\smash{$c^2 s_Y^2$} 
to $\sigma^2$ always decreases the survival functions of these
two truncated Gaussians; see Appendix \ref{app:trun_sigma}.  This  
means that confidence intervals given by directly inverting
either the plug-in or bootstrap TG statistic do not have provably correct 
asymptotic coverage properties, under the current analysis.

From the arguments in the proof of Theorem
\ref{thm:unknown}, we can construct one-sided
confidence intervals with conversative asymptotic coverage, by forcing
them to include \smash{$\hat{a}_M$}.  We do not pursue the details here, 
as we have found that these one-sided intervals are
practically too wide to be of interest.  

Importantly, the plug-in and bootstrap TG statistics often display
excellent empirical properties, as we will show in the next section. A
more refined analysis is needed to establish asymptotic uniformity for
the distribution of these statistics under $v^T \theta = \mu$.  Such asymptotic
uniformity, for arbitrary $\mu$, would lead to asymptotic coverage guarantees
for confidence intervals produced by inverting these statistics, and we leave
this extension to future work.

\section{Examples}
\label{sec:examples}

We present empirical examples that support the theory
developed in the previous sections, and also suggest that there is
much room to refine and expand our current set of results.
The first two subsections examine a
low-dimensional problem setting that is covered by our theory.
The last two look at substantial departures 
from this theoretical framework, the heteroskedastic and 
high-dimensional settings, respectively. In all 
examples, the LAR algorithm was used for variable selection and 
associated inferences; results with the FS and lasso paths were roughly
similar.  Also, in all examples, where not explicitly stated otherwise, the
computed p-values are a test of whether the target population value is 0. 

It may be worth discussing two potentially common reactions to
our experimental setups, especially for the low-dimensional problems 
described in the next subsections.  First, our plug-in
statistic uses $s_Y^2$ as an estimate for $\sigma^2$; why not use an 
estimate from the full least squares model of $Y$ on $X$, since this
would be less conservative?  While experiments (not shown) confirm
that this works in low-dimensional regression problems, such an
estimate becomes anti-conservative as the number of variables  
grows (particularly, irrelevant ones), and is obviously not applicable
in high-dimensional problems.  Therefore, we stick with the simple
estimate $s_Y^2$, as this is always applicable and always
conservative.  

Second, to determine variable significance in a
low-dimensional problem, one could of course fit a full regression
model and inspect the resulting p-values and confidence intervals.
These p-values and intervals could even be Bonferonni-adjusted to
account for selection.  Of course, this strategy would not be
possible for a high-dimensional problem, but if the number of
predictors is small enough, then it may work perfectly fine. 
So when should one use more complex tools for post-selection
inference? This is an important question, deserving of study, but it
is not the topic of this paper.  
The examples that follow are intended to portray
the robustness of the selective pivotal inference method against  
nonnormal error distributions;
they are not meant to represent the ideal statistical   
practice in any given scenario.

\subsection{P-value examples}
\label{sec:pval_ex}

We begin by studying a low-dimensional setting with $n=50$ and $d=10$.    
We defined predictors $X \in \R^{50 \times 10}$, by
drawing the columns independently according to the following mixture 
distribution: with equal probability, a column was filled with
i.i.d.\ entries from $N(0,1)$, $\mathrm{Bern}(0.5)$, or
$SN(0,1,5)$, where $SN(0,1,5)$ denotes
the skew normal distribution \citep{ohagan1976bayes} with  
shape parameter equal to 5. We then scaled the columns of $X$ to have
unit norm.   
The underlying mean was defined as $\theta = X\beta_0$, where
$\beta_0 \in \R^{10}$ has its first 2 components
equal to $-4$ and $4$, and the rest set to 0. Over 500 
repetitions, we drew a response $Y \in \R^{50}$ from
\eqref{eq:model}, with i.i.d.\ errors, and 4 different choices for the
error distribution: normal, Laplace, uniform, 
and skew normal. In each case, we centered the error distribution,
and we scaled it to have variance $\sigma^2=1$ (for the skew normal  
distribution, we used a shape parameter 5).  Every 10 repetitions, the
predictor matrix $X$ was regenerated according to the prescription
described above. 

Figure \ref{fig:lo_pval} displays QQ plots of p-values for testing the 
significance of the variable entered into the active model, across 3 
steps of LAR.  (The QQ plots compare the p-values to a standard
uniform distribution.) 
The p-values were computed using the TG statistic with
$\sigma^2=1$, the plug-in 
TG statistic with $s_Y^2$ as its estimate for $\sigma^2$, and the
bootstrap 
TG statistic with 50,000 bootstrap samples used to approximate the  
probabilities in the numerator and denominator of
\eqref{eq:pivot_boot}, and padding factor $\delta_n = 10^{-4}  
n^{-1/4}$.  (The scaling factor was ignored, i.e., set to $c=1$, 
for the plug-in and bootstrap statistics.) In steps 1 and 2, the
p-values are restricted to repetitions in which a correct variable 
selection was made---i.e., variable 1 or 2 was entered into the active
LAR model.  In step 3, the p-values are from repetitions in
which an incorrect variable selection was made---i.e., one of
variables 3 through 10 was entered into the active model. 
Since the underlying signal was fairly strong and the predictors
uncorrelated, such selections happened the majority of the time; 
specifically, the p-values displayed for steps 1, 2, and 3 comprise 
approximately 95\%, 85\%, and 87\% of the 500 repetitions,
respectively.  
The p-values in steps 1 and 2 show reasonable power, 
for all 3 statistics (TG, plug-in, and bootstrap types), and all 4
error distributions.  Also, the p-values in step 3 are uniform, as
desired, again for all statistics and all error distributions. Though
the guarantees (for uniform null p-values) are only asymptotic for the 
Laplace, uniform, and skew normal error distributions, such
asymptotic behavior appears to kick in quite early for these
distributions, as the sample size here is only $n=50$.  
Further, the QQ plots reveal that the p-values for the nonnormal
error  distributions are not really any farther from uniform than they
are in the normal case. This is somewhat remarkable, recalling that
the p-values are, by construction, {\it exactly} uniform under normal
errors.  

\begin{figure}[p]
\begin{subfigure}[b]{\textwidth}
\centering
Step 1, p-values \smallskip \\
\includegraphics[width=0.24\textwidth]{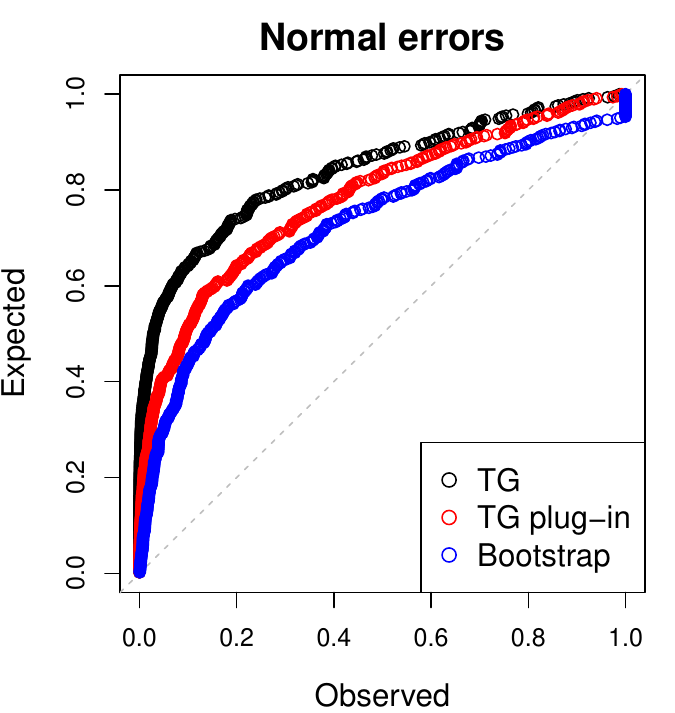}
\includegraphics[width=0.24\textwidth]{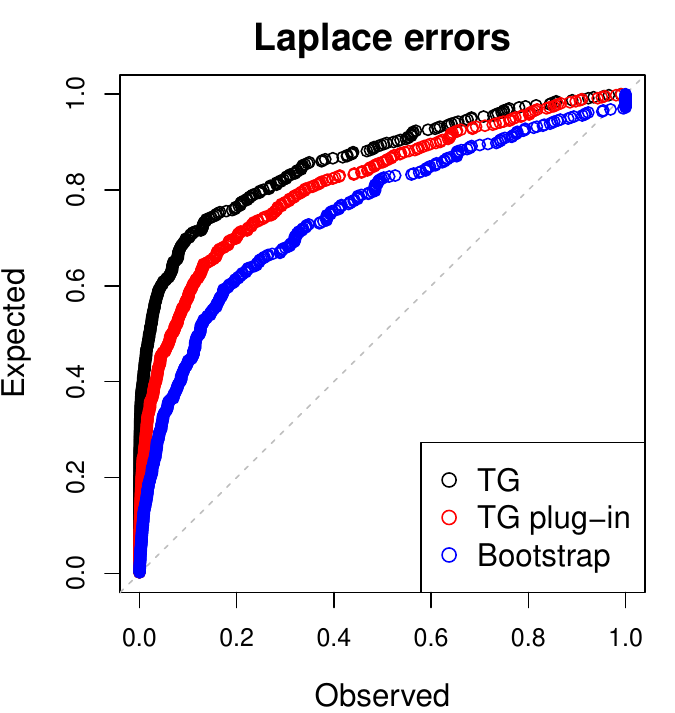}
\includegraphics[width=0.24\textwidth]{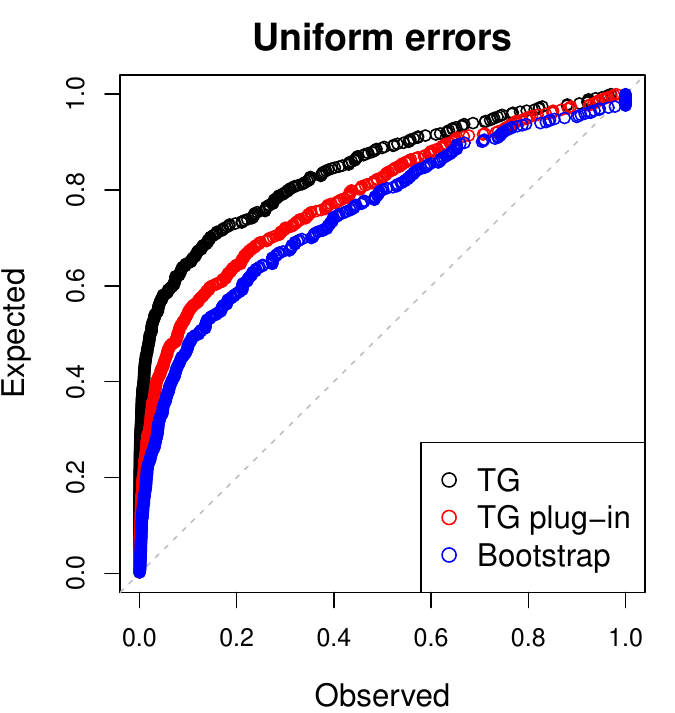}
\includegraphics[width=0.24\textwidth]{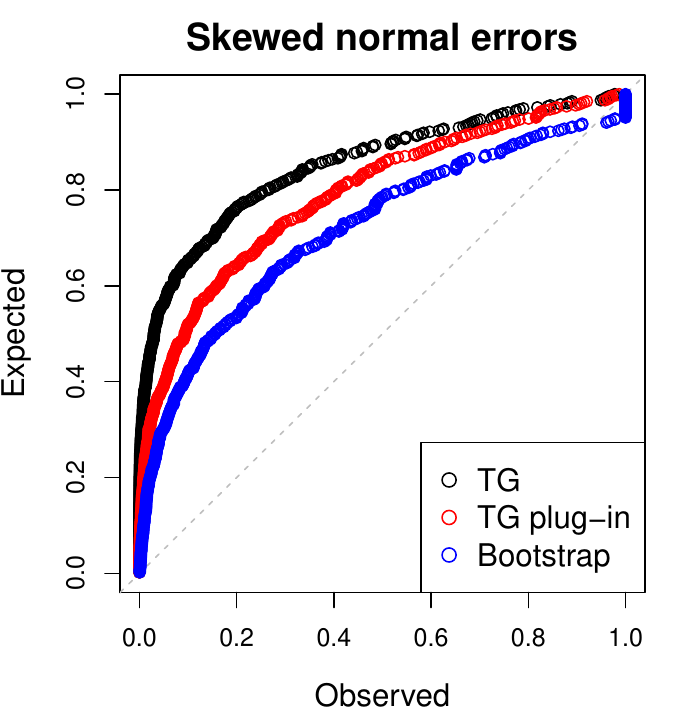} 
\smallskip \\
Step 2, p-values \smallskip \\
\includegraphics[width=0.24\textwidth]{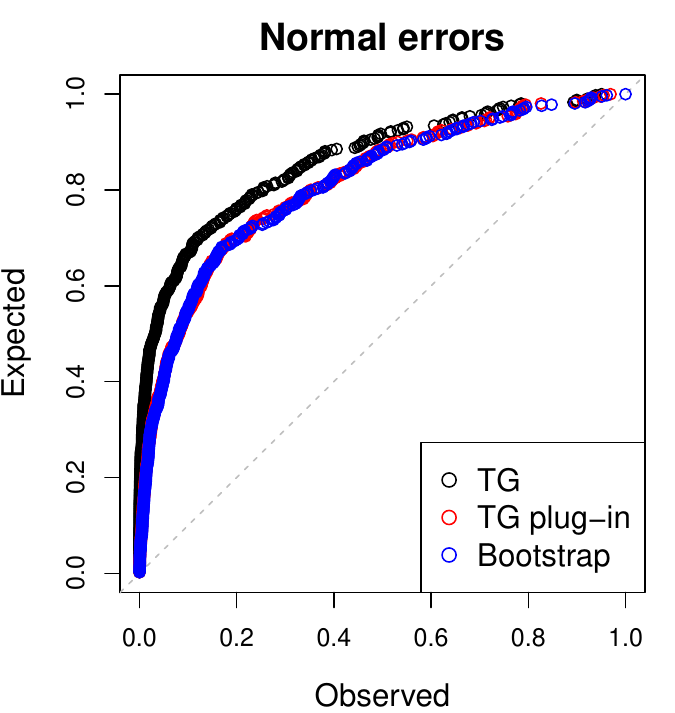}
\includegraphics[width=0.24\textwidth]{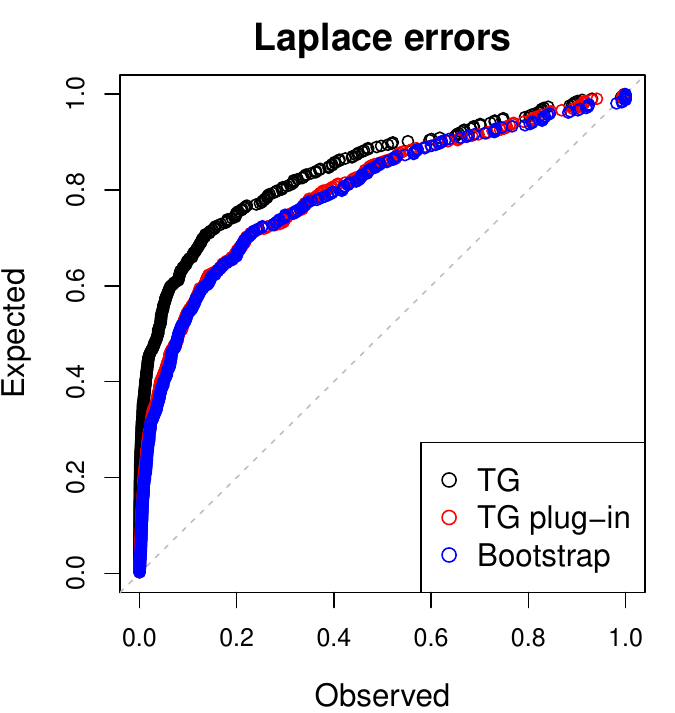} 
\includegraphics[width=0.24\textwidth]{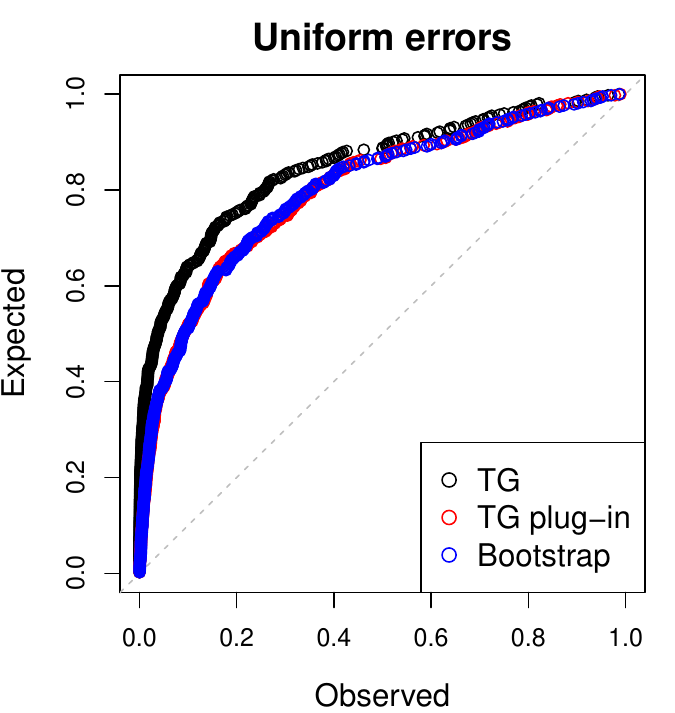} 
\includegraphics[width=0.24\textwidth]{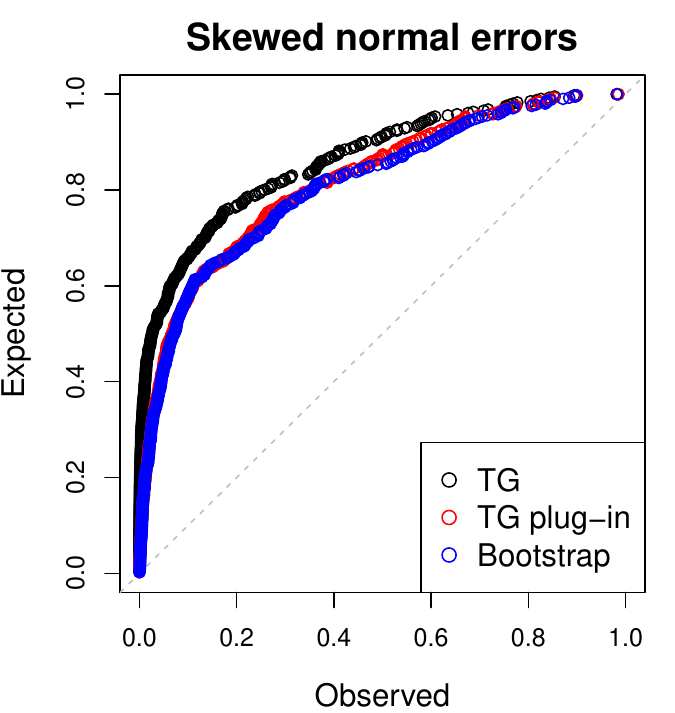} 
\smallskip \\
Step 3, p-values \smallskip \\ 
\includegraphics[width=0.24\textwidth]{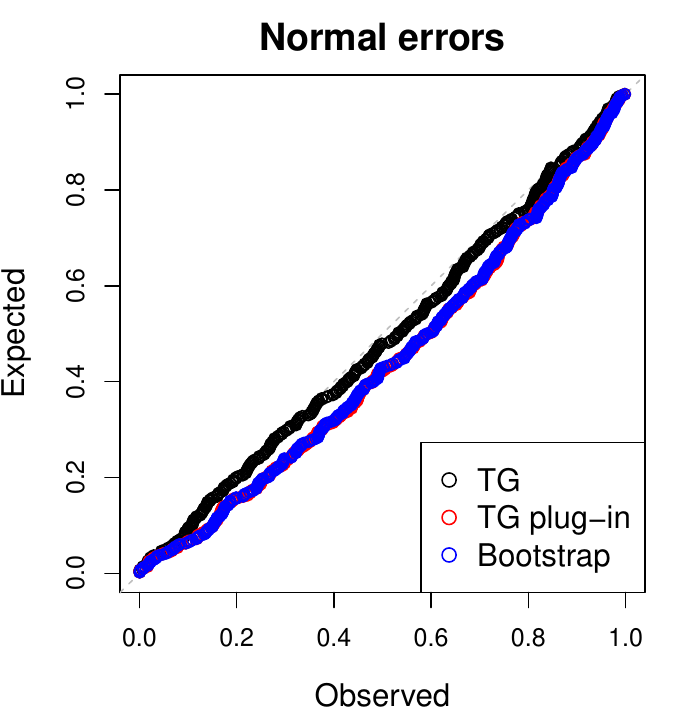}
\includegraphics[width=0.24\textwidth]{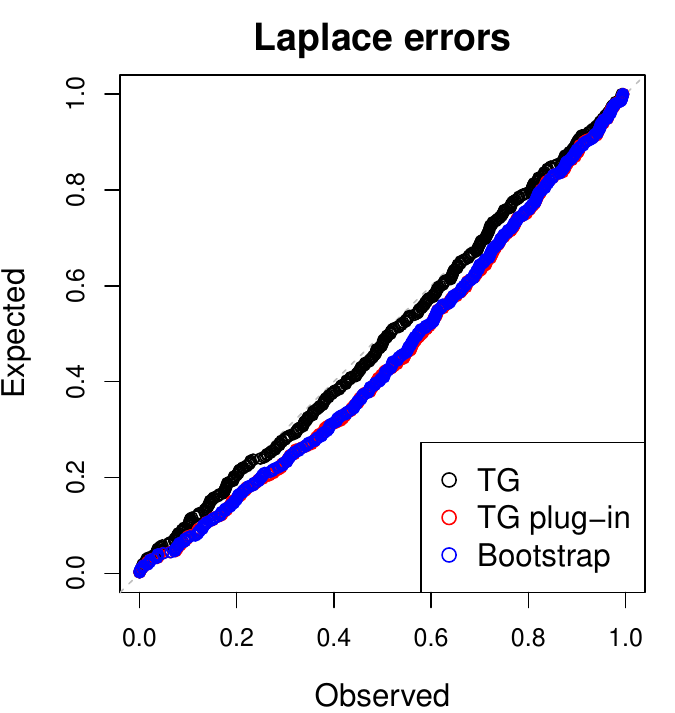}
\includegraphics[width=0.24\textwidth]{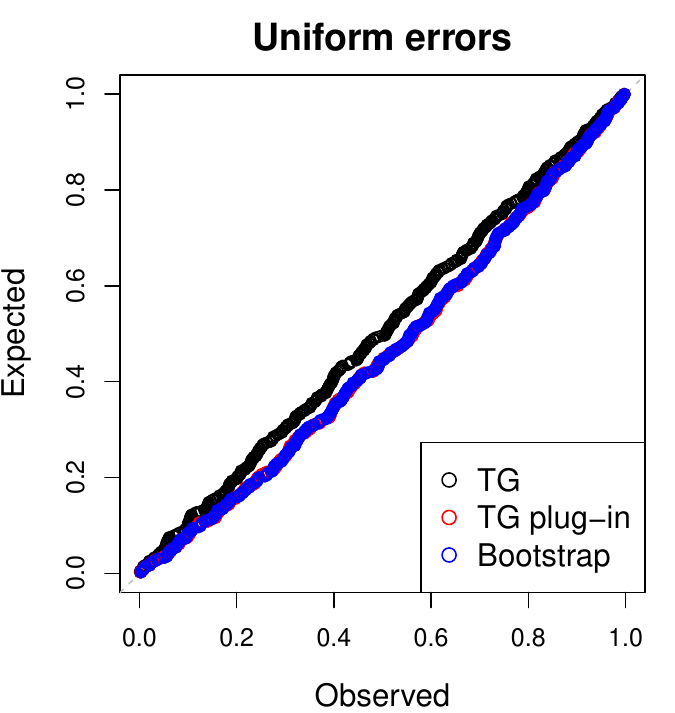}
\includegraphics[width=0.24\textwidth]{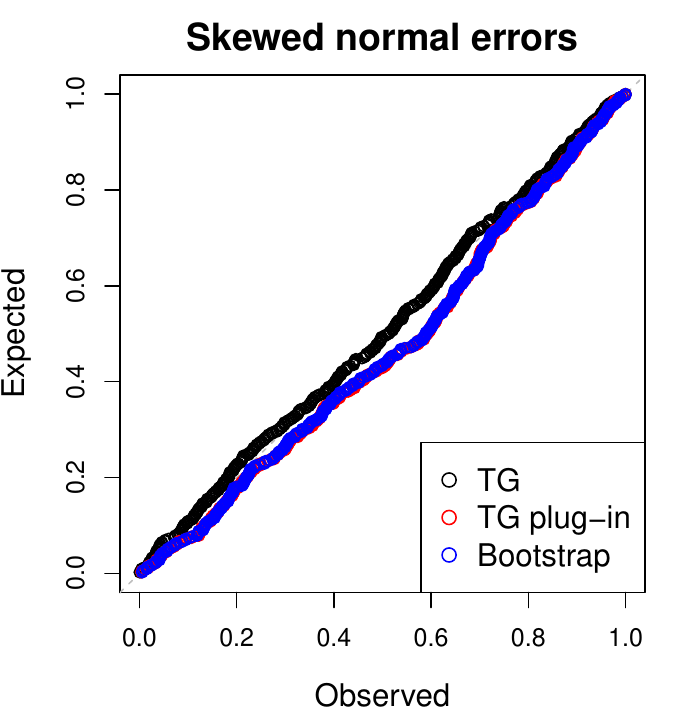}
\caption{\it\small P-values are shown, after each of 3 steps of LAR.}
\label{fig:lo_pval}
\end{subfigure}

\bigskip\smallskip\smallskip
\begin{subfigure}[b]{\textwidth}
\centering
All steps, pivotal statistics \smallskip \\
\includegraphics[width=0.24\textwidth]{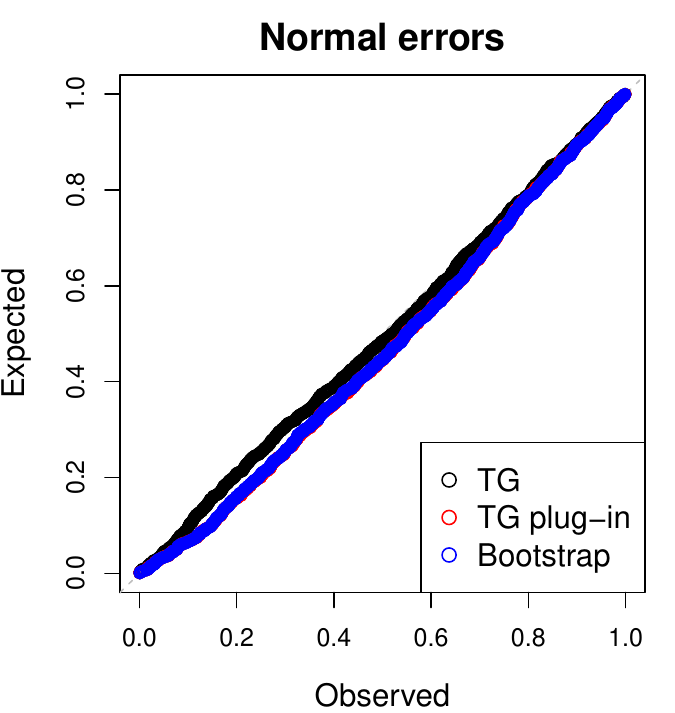}
\includegraphics[width=0.24\textwidth]{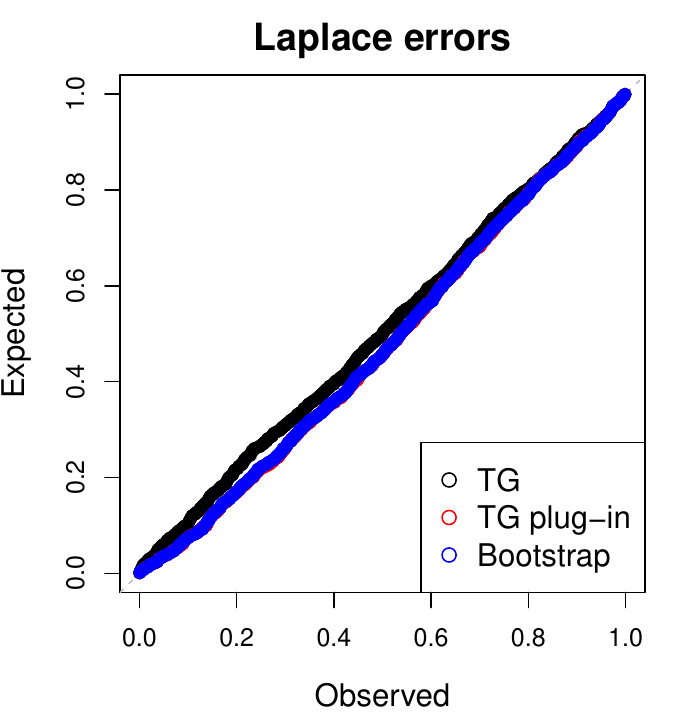}
\includegraphics[width=0.24\textwidth]{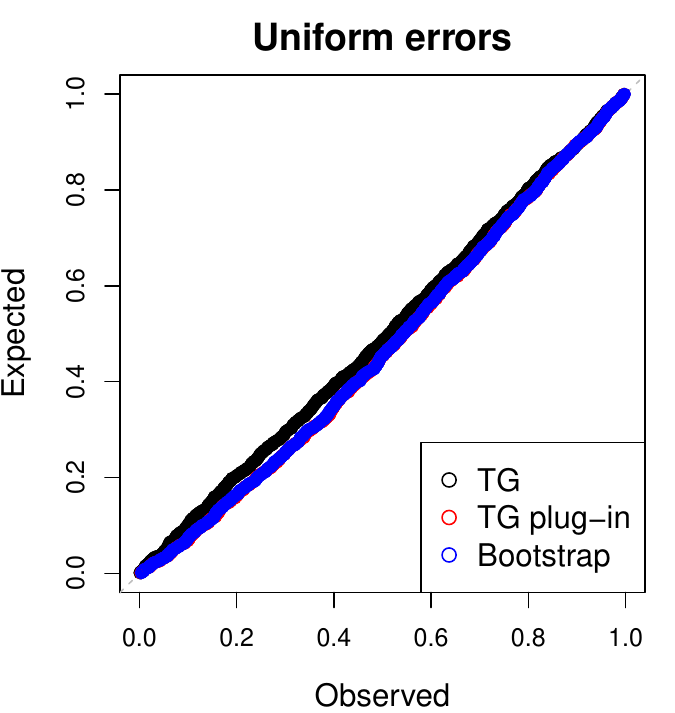}
\includegraphics[width=0.24\textwidth]{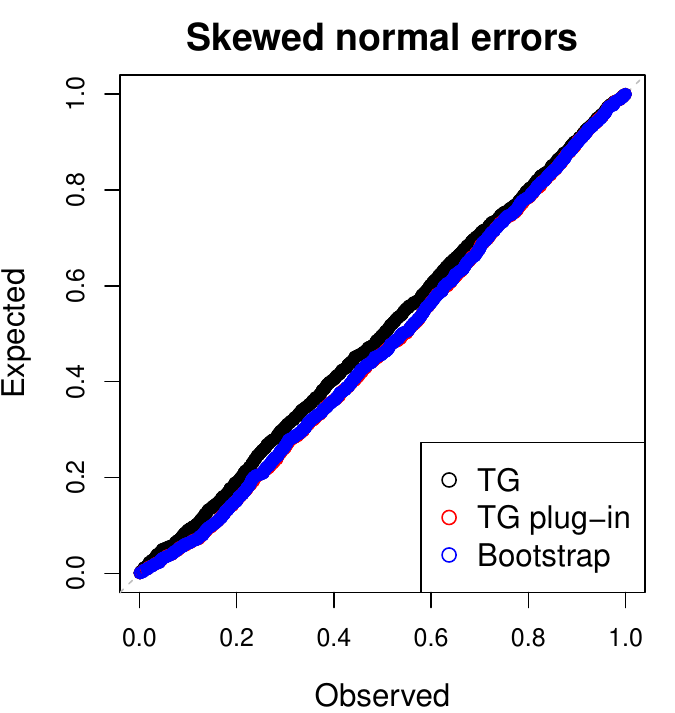}
\caption{\it\small Pivotal statistics are shown, aggregated over all 
  3 steps of LAR.}
\label{fig:lo_pivot}
\end{subfigure}

\caption{\it\small A simulation setup with $n=50$ and $d=10$, and
  a mean $\theta=X\beta_0$, where $\beta_0$ has 2 nonzero components.} 
\label{fig:lo}
\end{figure}

Figure \ref{fig:lo_pivot} inspects the TG statistic and plug-in
and boostrap variants, when the pivot value $\mu$ is set to the true 
population value.  That is, we set $\mu=v^T \theta$ in computing the
statistics in \eqref{eq:pivot_succint},  
\eqref{eq:pivot_plugin}, and \eqref{eq:pivot_boot}, in each data
instance and each step of LAR. The figure collects the p-values across
all 3 steps of LAR, for each of the 4 error distribution
types. According to our theory, the distribution of the TG
pivotal statistics here should be asymptotically uniform.  This is
clearly supported by the QQ plots.  Interestingly, both plug-in and
bootstrap pivotal statistics also appear uniform in the QQ plots, and  
yet, this is not a case handled by our asymptotic theory: recall, 
Theorem \ref{thm:unknown} fixes the pivot value $\mu$ to be 0
(as, otherwise, technical difficulties are encountered in its proof). 
This gives empirical evidence to the idea that a more refined analysis
could extend Theorem \ref{thm:unknown} to the broader setting (of
arbitrary pivot values) handled by Theorem \ref{thm:pivot}.
Moreover, it suggests that inverting the plug-in and bootstrap TG
statistics should yield intervals with proper coverage, which is
verified in the next subsection. 

Lastly, we repeated all experiments in this subsection
with the predictors $X \in \R^{50 \times 10}$
generated in such a way to induce a (population) correlation of 0.5
between all pairs of predictor variables. The results are quite
similar to those shown in Figure \ref{fig:lo}, and are hence deferred
to Appendix \ref{app:pval_ex_cor}.

\subsection{Confidence interval examples}
\label{sec:interval_ex}

We stay in same setting as the last subsection, so that $n=50$,
$d=10$, and $\theta=X\beta_0$ for a coefficient vector $\beta_0$ with
its first 2 components equal to $-4$ and $4$, and the rest equal to 0.
We invert the TG, plug-in 
TG, and bootstrap TG statistics to obtain 90\% confidence intervals at
each LAR step. See 
Table \ref{tab:int} for a numerical summary.
``Coverage'' refers to the average fraction of intervals that
contained their respective targets over the 500 repetitions, ``power''
is the average fraction of intervals that excluded zero, and ``width''
is the median interval width. These are all recorded in
an unconditional sense, i.e., no screening of
repetitions was performed based on the variables that were selected  
across the 3 steps of LAR (the conditional coverages however, were quite
similar).  From the table, we can see that all 3
methods lead to accurate coverage (around 90\%) in all cases. We can
further see that the intervals from the bootstrap TG statistic are 
shorter than those from the plug-in TG statistic in all cases, and 
considerably shorter than both the plug-in and original TG statistics
in steps 2 and 3.  The power from the bootstrap TG intervals is
generally better than that from the plug-in TG intervals; also, it is 
on par with the power from the original TG statistic in step 1, but
somewhat worse in step 2. Recall that the original TG statistic
uses knowledge of the error variance ($\sigma^2=1$) but the
bootstrap and plug-in variants do not. 

\begin{table}[htbp]
\small
\centering
\begin{tabular}{|r|r|}
\multicolumn{2}{r}{} \\ 
\multicolumn{2}{r}{} \\ 
\hline
\multirow{3}{*}{N} 
& \multicolumn{1}{r|}{TG} \\
& \multicolumn{1}{r|}{Plug-in} \\
& \multicolumn{1}{r|}{Boot} \\
\hline
\multirow{3}{*}{L} 
& \multicolumn{1}{r|}{TG} \\
& \multicolumn{1}{r|}{Plug-in} \\
& \multicolumn{1}{r|}{Boot} \\
\hline
\multirow{3}{*}{U}
& \multicolumn{1}{r|}{TG} \\
& \multicolumn{1}{r|}{Plug-in} \\
& \multicolumn{1}{r|}{Boot} \\
\hline
\multirow{3}{*}{S}
& \multicolumn{1}{r|}{TG} \\
& \multicolumn{1}{r|}{Plug-in} \\
& \multicolumn{1}{r|}{Boot} \\
\hline
\end{tabular}
\begin{tabular}{|rrr|}
\hline
\multicolumn{3}{|c|}{Step 1} \\
\hline
Coverage & Power & Width \\
\hline
0.914 & 0.508 & 5.622 \\
0.928 & 0.378 & 7.561 \\
0.932 & 0.528 & 5.477 \\
\hline
0.904 & 0.568 & 5.193 \\
0.944 & 0.410 & 7.271 \\
0.944 & 0.566 & 5.429 \\
\hline
0.912 & 0.538 & 5.153 \\
0.928 & 0.396 & 7.284 \\
0.924 & 0.540 & 5.453 \\
\hline
0.892 & 0.540 & 5.346 \\
0.940 & 0.402 & 7.210 \\
0.936 & 0.520 & 5.477 \\
\hline
\end{tabular}
\begin{tabular}{|rrr|}
\hline
\multicolumn{3}{|c|}{Step 2} \\
\hline
Coverage & Power & Width \\
\hline
0.890 & 0.520 & 10.309 \\
0.914 & 0.404 & 15.774 \\
0.916 & 0.424 & 7.856 \\
\hline
0.926 & 0.536 & 11.153 \\
0.930 & 0.440 & 14.859 \\
0.944 & 0.454 & 7.892 \\
\hline
0.902 & 0.504 & 12.347 \\
0.910 & 0.390 & 17.497 \\
0.910 & 0.422 & 7.808 \\
\hline
0.878 & 0.504 & 10.876 \\
0.896 & 0.380 & 15.687 \\
0.912 & 0.394 & 8.060 \\
\hline
\end{tabular}
\begin{tabular}{|rrr|}
\hline
\multicolumn{3}{|c|}{Step 3} \\
\hline
Coverage & Power & Width \\
\hline
0.910 & 0.114 & 25.155 \\
0.918 & 0.100 & 34.642 \\
0.930 & 0.090 & 9.141 \\
\hline
0.912 & 0.118 & 26.393 \\
0.904 & 0.120 & 36.206 \\
0.924 & 0.108 & 9.273 \\
\hline
0.894 & 0.128 & 26.451 \\
0.886 & 0.126 & 39.299 \\
0.892 & 0.118 & 8.913 \\
\hline
0.906 & 0.116 & 26.592 \\
0.910 & 0.106 & 38.965 \\
0.918 & 0.102 & 9.057 \\
\hline
\end{tabular}
\caption{\small\it Summary statistics for 90\% confidence 
  intervals constructed in the problem setting of
  Figure \ref{fig:lo}. The 4 blocks of rows correspond to the 
  4 types of noise: normal, Laplace, uniform, and skew normal, 
  respectively. The standard errors are about 0.01, 0.02, and 0.42 for 
  the coverage, power, and width statistics, respectively.}    
\label{tab:int}
\end{table}

It is a bit surprising that the bootstrap intervals
can be shorter but still have worse power than the original
TG intervals.  This is easier to understand
once the intervals are visualized, as done in Figure \ref{fig:int_e1}.
The figure shows 100 
sample intervals from the first LAR step, under normally distributed 
errors. Sample intervals from the other error models are shown
in Appendix \ref{app:more_intervals}.  We see that the  
bootstrap TG intervals are indeed shorter, but compared to the
original TG intervals, they are more symmetric around the
target population values.  The original TG intervals, being more
asymmetric, are often shorter on the side (of the target value) 
facing 0, and this results in better power.   

\begin{figure}[p]
\centering
\includegraphics[width=\textwidth]{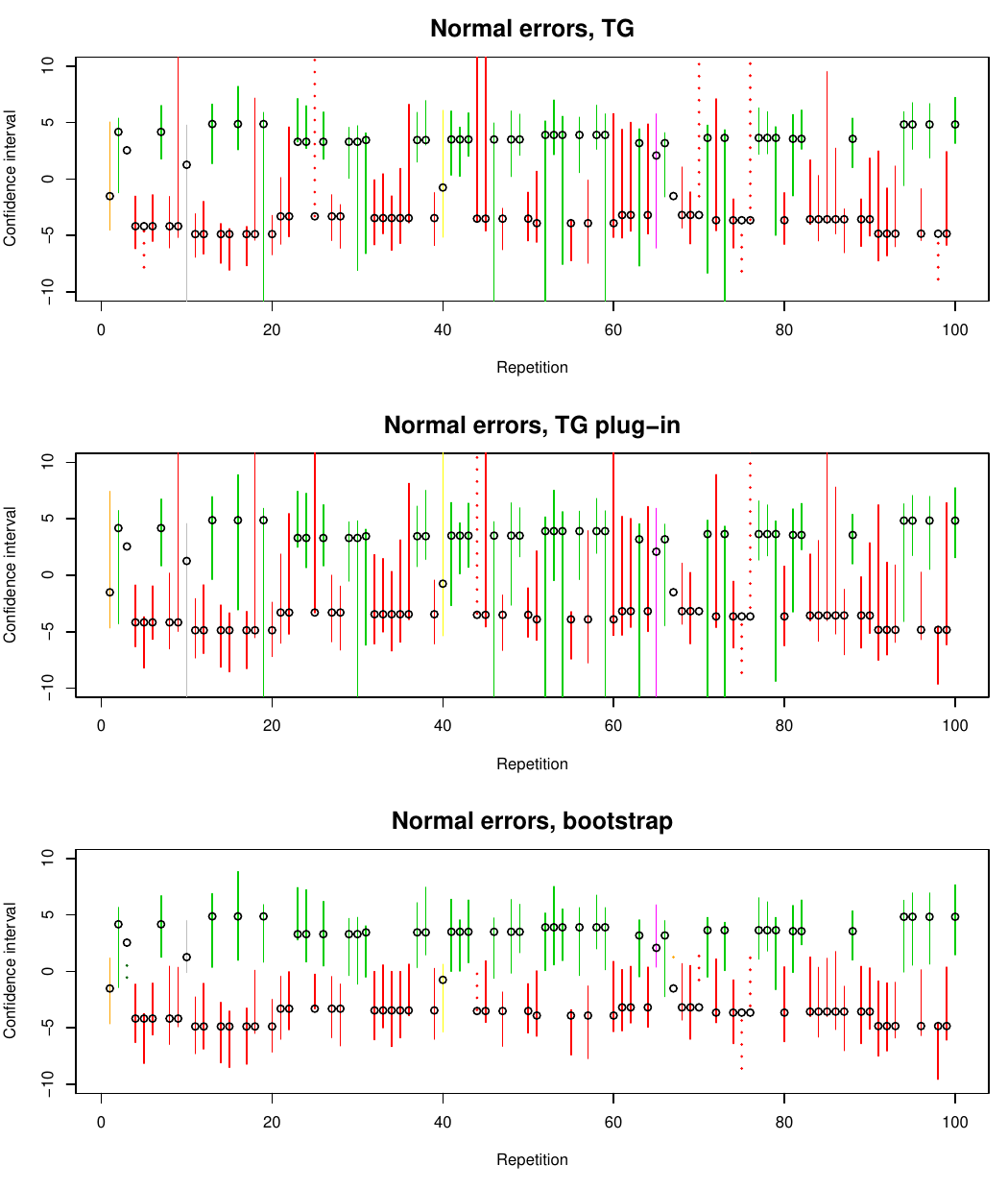}
\caption{\small\it Confidence intervals from 100 draws of $Y$ from the  
  same model as that in Figure \ref{fig:lo}.  These intervals are
  constructed from the first step of LAR, under a uniform 
  distribution for noise.  The colors are simply a visual aid to mark
  the selection of different variables at step 1.  The open circles
  denote the true population quantity to be covered (here, the 
  coefficient from projecting $\theta$ onto the first selected
  variable). Intervals that do not contain their targets are drawn
  as dotted segments.}
\label{fig:int_e1}
\end{figure}

Again, we repeated the experiments here
with the predictors $X \in \R^{50 \times 10}$
generated to have pairwise correlation 0.5. Comparisons can be
drawn between the results in a manner
that roughly parallels the discussions following Table \ref{tab:int};
however, on an absolute scale, all methods display a decrease in power 
across the board (as correlated predictors clearly make the problem  
more difficult). Details are provided in Appendix
\ref{app:interval_ex_cor}.  


\subsection{Heteroskedastic errors}
\label{sec:hetero_er}

In the same setup as in Sections \ref{sec:pval_ex} and
\ref{sec:interval_ex}, with $n=50$, $d=10$,
and the predictors $X$ and mean $\theta$ generated in the same manner,
we consider a heteroskedastic model for $Y$ by drawing $\epsilon'_i$,  
$i=1,\ldots,n$ i.i.d.\ from the given distribution---normal, Laplace,
uniform, or skew normal---and then taking the errors to be 
$\epsilon_i=\sigma_i \epsilon_i'$, $i=1,\ldots,n$, where
$\sigma^2_i=10\|x_i\|_2^2$, $i=1,\ldots,n$ (and where $x_i \in \R^d$, 
$i=1,\ldots,n$ denote the rows of $X$.)  The spread of error variances
ended up being fairly substantial, from about 0.3 to 5.5.  The
original TG statistic was computed with \smash{$\sigma^2 =
  \frac{1}{n} \sum_{i=1}^n \sigma_i^2$} as a surrogate for the common 
error variance; the plug-in and bootstrap variants were computed as
usual.  For brevity, we only plot the pivotal statistics, aggregated
over 3 steps of LAR, in Figure \ref{fig:hetero}.   (This is analogous
to what is shown in Figure \ref{fig:lo_pivot} for the homoskedastic 
case. P-values at steps 1, 2, 
and 3, not shown, end up being similar to those in Figure 
\ref{fig:lo_pval}, but the power from all methods is generally lower,
due to the heteroskedastic errors.) As we can see,
the pivotal statistics in the figure look very close
to uniformly distributed, as desired.  
This is especially encouraging because the current problem setup lies
outside of the scope of 
our asymptotic theory (which assumes a constant error variance), and
it suggests that our theory could possibly be extended to accomodate
errors with an (unknown) nonconstant variance structure. 

\begin{figure}[tb]
\centering
\includegraphics[width=0.24\textwidth]{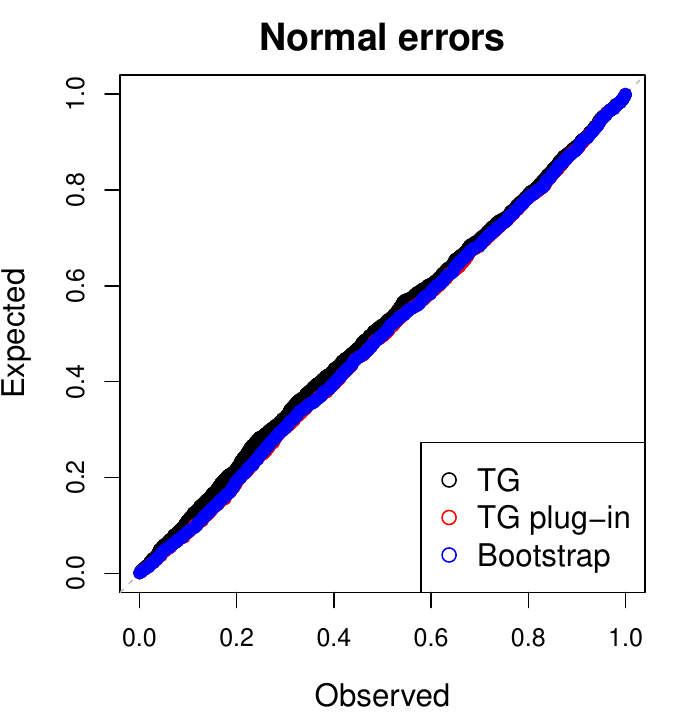}
\includegraphics[width=0.24\textwidth]{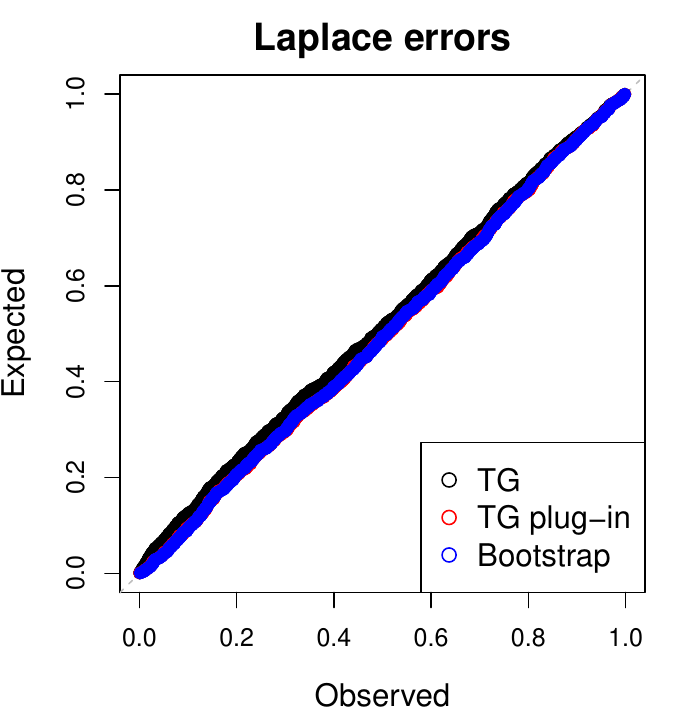}
\includegraphics[width=0.24\textwidth]{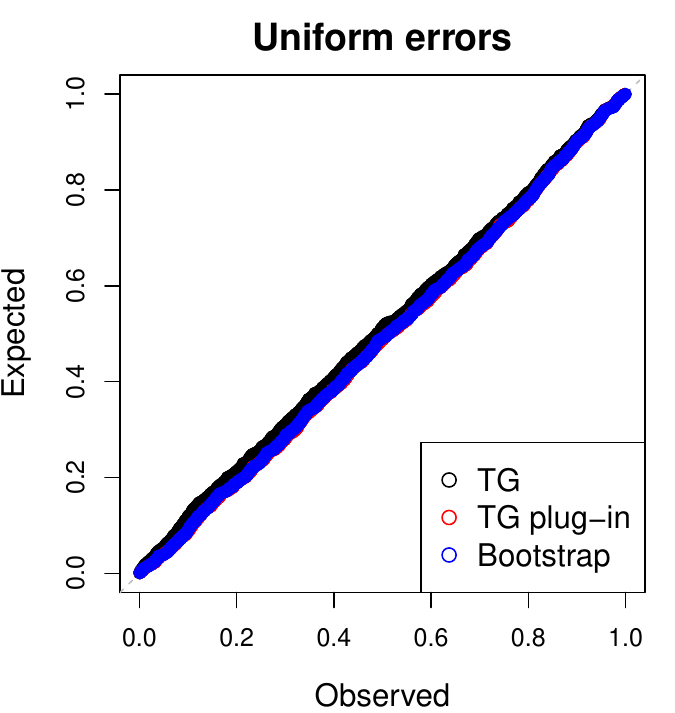}
\includegraphics[width=0.24\textwidth]{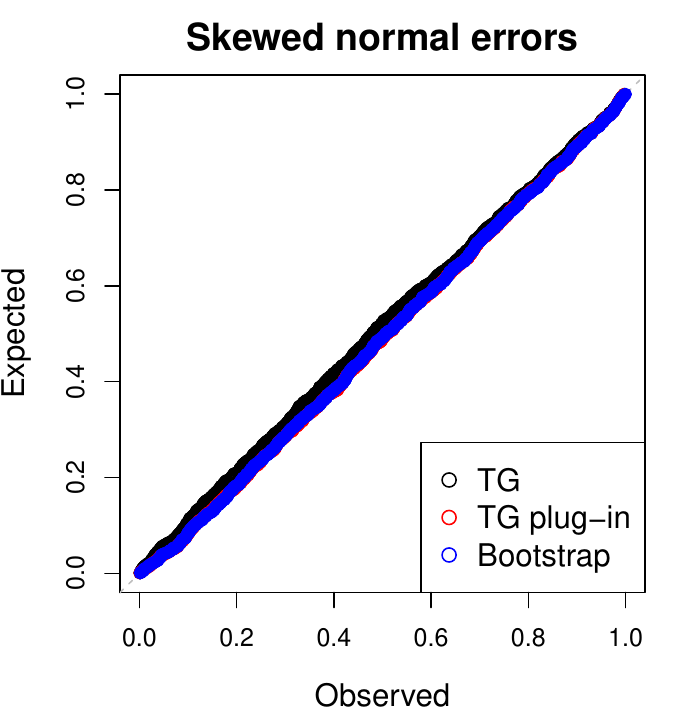} 
\caption{\it\small A simulation setup with $n=50$ and $d=10$, but with
  heteroskedastic errors. Shown are the pivotal statistics aggregated
  over 3 LAR steps.} 
\label{fig:hetero}
\end{figure}

\subsection{High-dimensional examples}
\label{sec:highdim_ex}

Finally, we consider a high-dimensional regime with $n=50$ and 
$d=1000$ predictors.  The matrix $X \in \R^{50 \times 1000}$ was
generated according to the same recipe as before: each column, with
equal probability, was assigned i.i.d.\ entries from $N(0,1)$, 
$\mathrm{Bern}(0.5)$, or\ $SN(0,1,5)$, and
then scaled to have unit norm. The mean was defined as $\theta=
X\beta_0$, where $\beta_0 \in \R^{1000}$ has its first 2 components
equal to -4 and 4, and the rest 0.  Over 500 repetitions, a
response $Y \in \R^{50}$ was generated by adding normal, Laplace,  
uniform, or skew normal noise to $\theta$, with an error variance of 
$\sigma^2=1$ (and every 10 repetitions, the predictor matrix $X$ was
regenerated). Figure \ref{fig:hi} plots the pivotal
statistics aggregated over the first 3 steps of LAR.  (This is as in
Figure \ref{fig:lo_pivot} for the low-dimensional case. P-values from 
the first 3 LAR steps are omitted for brevity, and are roughly
similar to those in Figure \ref{fig:lo_pval}, except that they display
less power, due to the high-dimensionality.) The pivotal
statistics here look quite close to uniform, as desired, and this is
again encouraging, especially given that the current high-dimensional
case lies outside of the scope of our theory (which assumes that $d$
is fixed).  Further work on high-dimensional asymptotic theory should
be pursued (see also \citet{tian2017asymptotics}), though, as we show
in the next section, there is no hope for a uniform convergence 
result in high dimensions that holds as generally as the one we
established in Theorem \ref{thm:pivot} for low dimensions. 

\begin{figure}[tb]
\centering
\includegraphics[width=0.24\textwidth]{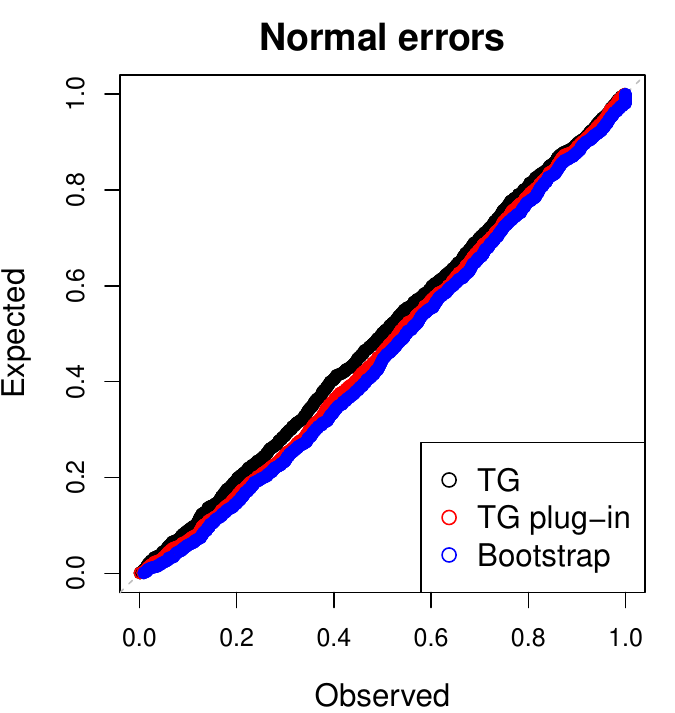}
\includegraphics[width=0.24\textwidth]{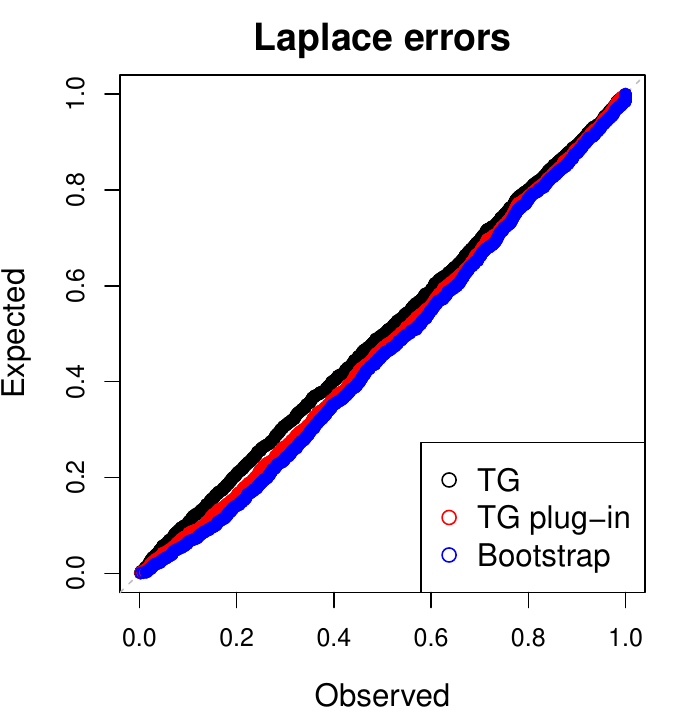}
\includegraphics[width=0.24\textwidth]{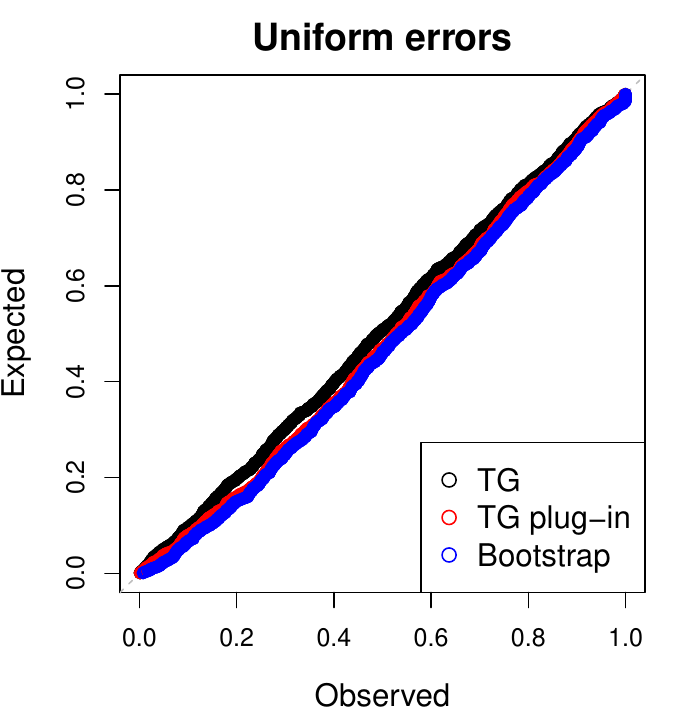}
\includegraphics[width=0.24\textwidth]{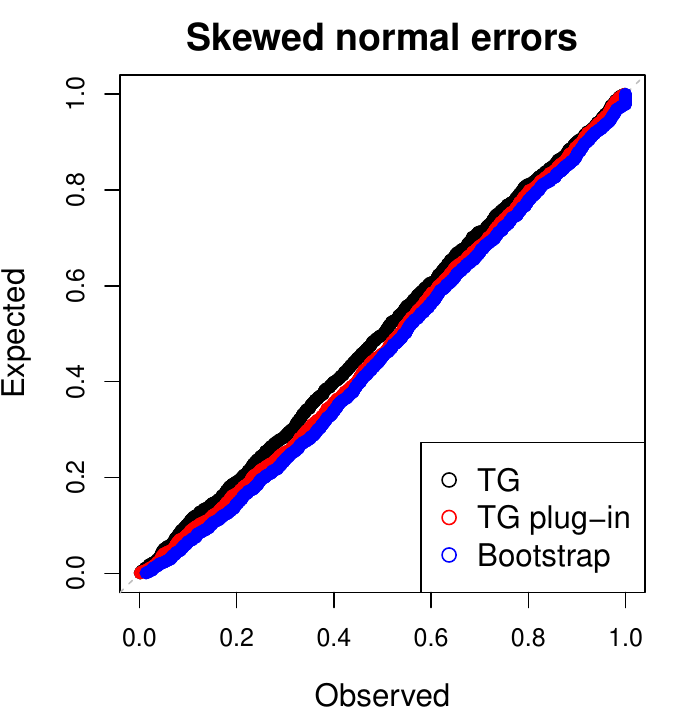} 
\caption{\it\small A simulation setup with $n=50$ and $d=1000$. Shown 
  are the pivotal statistics over 3 LAR steps.}   
\label{fig:hi}
\end{figure}

\section{A negative result in high dimensions} 
\label{sec:highdim}

We prove that the TG statistic fails to converge to a uniform
distribution, under the null hypothesis, in a data model that has
nonnormal errors and is high-dimensional, but otherwise represents a
fairly standard setting: the ``many means'' 
setting. We write the observation model as 
\begin{equation}
\label{eq:manymeans}
Y_{ij} = \mu_j + \epsilon_{ij}, \;\;\; 
i=1,\ldots,m, \; j=1,\ldots,d,
\end{equation}
where we interpret $i=1,\ldots,m$ as replications, and $j=1,\ldots,d$
as dimensions.  In total there are hence $n=md$ observations.  Denote 
\begin{equation*}
\bar{Y}_j = \frac{1}{m} \sum_{i=1}^m Y_{ij}, \;\;\; j=1,\ldots,d.
\end{equation*}
We will analyze the TG statistic, when selection is performed based on 
the largest of \smash{$|\bar{Y}_j|$}, $j=1,\ldots,d$, and inference is
then performed on the corresponding mean parameter. A
straightforward change of notation will translate 
the above into a regression problem, with an orthogonal design $X 
\in  \R^{n\times d}$, but we stick with the many means
formulation of the problem for simplicity.  

We assume that the errors $\epsilon_{ij}$, $i=1,\ldots,m$,
$j=1,\ldots,d$ in \eqref{eq:manymeans} are i.i.d.\ from the 
following mixture:
\begin{equation}
\label{eq:mixture}
\pi \cdot N(-B, 1) + (1-2\pi) \cdot N(0,1) + \pi
\cdot N(B, 1).
\end{equation}
The mixing proportion $\pi$ and mean shift $B$ will both scale
with $d$. Moreover, they will be chosen so that (for each $d$) the 
error variance is  
\begin{equation*}
\sigma^2 = 1+2\pi B^2 = 2.
\end{equation*}
As mentioned, we will consider model selection events of the form  
\begin{equation*}
\hat{M}(Y)=(j,s) \iff
s \bar{Y}_j \geq \max_{\ell \not= j} \; |\bar{Y}_\ell|.
\end{equation*}  
We note that this is exactly the same selection event as that 
from the first step of FS, LAR, or lasso paths, when run on
the regression version of this problem with orthogonal design
$X$. It is not hard to check that the TG statistic for
conditionally testing $\mu_j=0$, given that \smash{$\hat{M}(Y)=(j,s)$}, is    
\begin{equation}
\label{eq:mm_cond_pivot}
T(Y;j,s,0) = \frac{\displaystyle
1-\Phi\Bigg(\frac{\sqrt{m} s \bar{Y}_j}{\sqrt{2}}\Bigg)}
{\displaystyle
1-\Phi\Bigg(\frac{\max_{\ell\not=j} \;
  \sqrt{m} |\bar{Y}_\ell|}{\sqrt{2}}\Bigg)}.
\end{equation}
As per the spirit of our paper, we can also view this statistic
unconditionally; for this it is helpful to define
\smash{$W_1 = |\bar{Y}_1|, \ldots,W_d = |\bar{Y}_d|$}, and denote by 
$W_{(1)} \geq \ldots \geq W_{(d)}$ the order statistics.  Then from
\eqref{eq:mm_cond_pivot}, we can see that the
unconditional TG statistic for testing the selected mean being 0 is  
\begin{equation}
\label{eq:mm_pivot}
\cT(Y;0) = \frac{\displaystyle
1-\Phi\Bigg(\frac{\sqrt{m} W_{(1)}}{\sqrt{2}}\Bigg)}
{\displaystyle
1-\Phi\Bigg(\frac{\sqrt{m} W_{(2)}}{\sqrt{2}}\Bigg)}. 
\end{equation} 
The framework underlying the TG statistic tells us that if the errors in
\eqref{eq:manymeans} are i.i.d.\ $N(0,2)$, then for any fixed model $(j,s)$, the
pivot $T(Y;j,s,0)$ is uniformly distributed conditional on
\smash{$\hat{M}(Y)=(j,s)$}.  Further, if $W_{(1)}$ and $W_{(2)}$ are the 
largest and second largest absolute values of centered normal random
variables (each with variance $2/m$), then the unconditional pivot $\cT(Y;0)$ is  
again uniform. But when $W_{(1)},W_{(2)}$ are large, and are defined by the
order statistics of nonnormal random variates, the statistic $\cT(Y;0)$---which
in this case is defined by the extreme tail behavior of the normal 
distribution---could be nonuniform.  The next theorem asserts that such
nonuniformity does indeed happen asymptotically if we choose the mixture
distribution in \eqref{eq:mixture} appropriately.  

\begin{theorem}
\label{thm:highdim}
Assume the observation model \eqref{eq:manymeans}, where the errors
are all drawn i.i.d.\ from \eqref{eq:mixture}.  Let $d$ and $m$ scale
in such a manner that $(\log{d})/m \to \infty$.  Further, let
\begin{equation*}
\pi=\Bigg(\frac{1}{d}\Bigg)^{1/m}, \;\;\;
B=\sqrt{\frac{d^{1/m}}{2}},
\end{equation*}
so that the error variance is fixed at $\sigma^2=2$.
Then under the global null hypothesis, $\mu=0$, the unconditional 
TG statistic $\cT(Y;0)$ in \eqref{eq:mm_pivot} does not converge in
distribution to $U(0,1)$. In particular, on 
an event whose limiting probability is at least $1/e$, the
statistic $\cT(Y;0)$ converges to 0. 

Further, the same results hold conditionally on any selected model.  That 
is, for any fixed $(j,s)$, the conditional TG statistic \smash{$T(Y;j,s,0) \,|\,
  \hat{M}(Y)=(j,s)$} does not converge in distribution to $U(0,1)$, and on an
event with limiting probability (conditional on \smash{$\hat{M}(Y)=(j,s)$}) at
least $1/e$, it converges to 0.   
\end{theorem}

\begin{remark}  The assumed condition $(\log d)/m \to \infty$  
requires the dimension $d$ to diverge to $\infty$, but
not necessarily the number of replications $m$, though it clearly
allows $m$ to diverge at a sufficiently slow rate.  On the other hand,
if $d$ were fixed and $m$ diverged to $\infty$, then the result of the
theorem would no longer be true, and the limiting distribution of the 
TG p-value would revert to $U(0,1)$.  (To be careful, here we
would have cap the mixing probability $\pi$ at $1/2$ in order
for the mixture to make sense, since the current definition of $\pi$ 
diverges with $d$ fixed and $m$ tending to $\infty$.) In fact, this is 
ensured by our low-dimensional result in Theorem \ref{thm:pivot}:
after reformulating the many means problem in
appropriate regression notation, all of
the conditions of Theorem \ref{thm:pivot} are met by our current
setup when $d$ is fixed.  This is supported by the simulation
in Figure \ref{fig:larry}.
\end{remark} 

\begin{figure}[tb]
\centering
\includegraphics[width=0.4\textwidth]{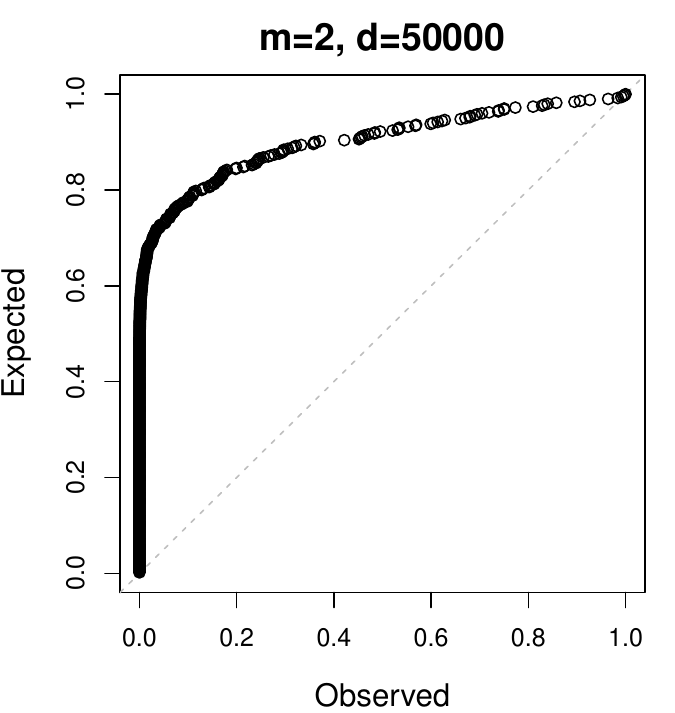}
\includegraphics[width=0.4\textwidth]{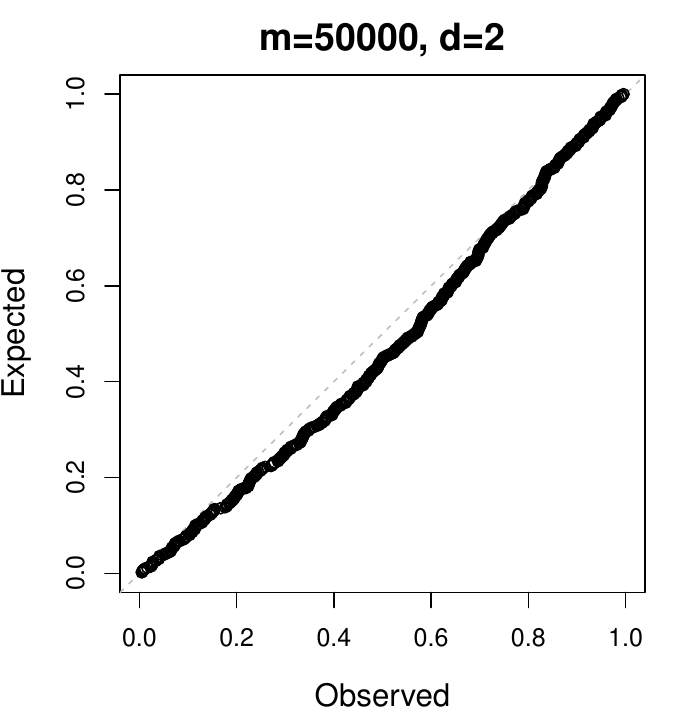}
\caption{\it\small The left plot shows a QQ plot of TG
  p-values, computed over 500 repetitions from the many means setup  
  exactly as described in Theorem \ref{thm:highdim}, with $d=50,000$
  and $m=2$. We can see that the p-values are clearly nonuniform, and
  $34\%$ of the p-values are 0 (up to computer precision), close to the  
  theoretically predicted proportion of $1/e$.  The right plot 
  shows p-values from the same model, but having reversed the roles of
  $d$ and $m$ (we also had to cap $\pi$ at 1/2); we can see that 
  the p-values are essentially uniform. 
}  
\label{fig:larry}
\end{figure}

\begin{remark}  The precise scaling $(\log d)/m \to \infty$ is 
chosen since this implies $\pi=(1/d)^{1/m} \to 0$,
i.e., the extreme mixture components $N(-B,1)$ and $N(B,1)$ each have 
probability tending to 0, an intuitively
reasonable property for the error distribution. But we note that this
scaling is not important for any other reason, and the proof would
still remain correct if $d/m \to \infty$.    
\end{remark} 

\begin{remark} In Theorem 3 of \citet{tian2017asymptotics}, the 
authors show that the TG statistic converges in distribution to a
standard uniform random variable, in a high-dimensional problem
setting, with some restrictions on the sequences of selection events
that are allowed.  One might ask what part of our high-dimensional
setup here violates their conditions, because both results obviously
cannot be true simultaneously. As far as we can tell, the issue
lies in the role of $\delta_n$ in Assumption 1 of
\citet{tian2017asymptotics}.  Namely, as we have defined the
error distribution in \eqref{eq:mixture}, the value of $\delta_n$
needed to certify the third condition Assumption 1 of
their work is too small for the main assumption in their
Theorem 3 to hold. Hence Theorem 3 of \citet{tian2017asymptotics}
does not apply to our current setup.
\end{remark}

\section{Discussion}
\label{sec:discussion}

We have studied the selective pivotal inference framework, with a
focus on forward stepwise regression (FS), least angle regression
(LAR), and the lasso, in regression problems with nonnormal  
errors.  We have shown that the truncated Gaussian (TG) pivot is
asymptotically robust in low-dimensional settings to departures from  
normality, in that it converges to a $U(0,1)$ distribution (its
pivotal distribution under normality), and does so uniformly over a
broad class of nonnormal error distributions.  When the
error variance $\sigma^2$ is unknown, we have proposed plug-in and
bootstrap versions of the TG statistic, both of which yield provably 
conservative asymptotic p-values.  

Our numerical experiments revealed that the statistics under
theoretical investigation generally display excellent finite-sample
performance, for highly nonnormal error distributions.  These
experiments also revealed findings not predicted by our theory: (i)
the bootstrap TG statistic often produces shorter confidence intervals 
than those based on the plug-in TG statistic, and even the TG
statistic that relies on the error variance $\sigma^2$; and (ii) all 
three TG statistics show strong empirical properties
well-outside of the classic homoskedastic, fixed $d$ regression
setting that we presumed theoretically. 

However, as we have clearly demonstrated, 
one should not hope for a convergence result in high dimensions that
is as general as the result obtained in low dimensions.  In a
relatively simple many means problem, we showed the nonconvergence
of the TG statistic to $U(0,1)$ as $d \to \infty$, whereas
in the same problem but with $d$ fixed, the TG statistic converges
to its usual $U(0,1)$ limit. 

There is still much left to do in terms of understanding the behavior 
of selective pivotal inference tools that are constructed to have
exact finite-sample guarantees under normality, like the TG statistic
of \citet{tibshirani2016exact}, when applied in high-dimensional
regression settings with nonnormal data. When the pivot, the central
cog of this framework, is constructed under the assumption of normality, 
this creates robustness issues that are especially worrisome in high
dimensions. Appendix \ref{app:instability} provides a high-level discussion 
of some of these issues; a more detailed study will be the subject of future
research.    

\subsection*{Acknolwedgements}

We thank Jelena Markovic and Jonathan Taylor for 
many helpful discussions, and for their overall generosity.  An initial
version of our work contained only unconditional (i.e., marginal) results in
the main theorems (Theorems \ref{thm:pivot}, \ref{thm:unknown}, and
\ref{thm:highdim}); Jelena Markovic pointed out that Theorem
\ref{thm:pivot} should also hold conditionally, and the current version of this
work has been revised accordingly.

\appendix

\section{Appendix}

\subsection{Convex cones for FS, LAR, lasso}
\label{app:partition}

We describe a modification of the conic conditioning set in
\citet{tibshirani2016exact} for FS.  Our version is different in that
we additionally condition on the sign of {\it every active
  coefficient} at every step, rather than just the coefficient of the
variable to enter the model at each step.  The modifications needed
for the LAR and lasso conditioning sets, made on top of the sets for
LAR and lasso given in \citet{tibshirani2016exact}, will follow
similarly to that described for FS, and hence we omit the details.

After $k$ FS steps, we can always represent a sequence of 
active sets \smash{$\hat{A}_\ell(y)$}, $\ell = 1,\ldots,k$ by a sorted
list of variables \smash{$[\hat{j}_1(y), \ldots,\hat{j}_k(y)]$} that were
chosen to enter the model at each step.  Unfortunately, the same
cannot be done for a sequence
of active signs \smash{$\hat{s}_\ell(y)$}, $\ell = 1,\ldots,k$,
because these do not obey such a nested structure.  We will write   
\smash{$\hat{s}_\ell(y)=[\hat{s}_{\ell,1}(y) ,\ldots
  \hat{s}_{\ell,\ell}(y)]$} for the signs of coefficients corresponding
to the variables \smash{$[\hat{j}_1(y), \ldots,\hat{j}_\ell(y)]$}, at
the $\ell$th step.

Now we characterize the event that \smash{$\hat{A}_\ell(y)=A_\ell$},
\smash{$\hat{s}_\ell(y)=s_\ell$}, for $\ell=1,\ldots,k$, using
induction.  At step $\ell=1$, we have that \smash{$\hat{j}_1(y)=j_1$}
and \smash{$\hat{s}_1(y)=s_1$} if and only if
\begin{equation*}
s_1 X_{j_1}^T y / \|X_{j_1}\|_2^2 \geq \pm X_j^T y/\|X_j\|_2^2
\;\;\; \text{for all $j\not=j_1$},
\end{equation*}
or, rearranged,
\begin{equation*}
\big(s_1 X_{j_1}/\|X_{j_1}\|_2^2 \pm X_j/\|X_j\|_2^2 \big)^T y \geq 0
\;\;\; \text{for all $j\not=j_1$},
\end{equation*}
a set of $2(d-1)$ linear inequalities in $y$.  Assume that we have
represented the event that \smash{$\hat{A}_\ell(y)=A_\ell$}
and \smash{$\hat{s}_\ell(y)=s_\ell$}, for $\ell=1,\ldots,k-1$, by a
collection of linear inequalities in $y$.  Then to represent 
\smash{$\hat{j}_k(y)=j_k$} and 
\smash{$\hat{s}_k(y) = [s_{k,1}, \ldots,s_{k,k}]$}, we must only
append to this collection of inequalities.  The former subevent
\smash{$\hat{j}_k(y)=j_k$} is characterized by  
\begin{equation*}
\big(s_{k,k} \tilde{X}_{j_k}/\|\tilde{X}_{j_k}\|_2^2 \pm 
\tilde{X}_j/\|\tilde{X}_j\|_2^2 \big)^T r \geq 0
\;\;\; \text{for all $j \not= j_1, \ldots,j_k$},
\end{equation*}
where \smash{$\tilde{X}_j$} is the residual from regressing $X_j$ onto 
\smash{$X_{A_{k-1}}$}, and $r$ is the residual from regression $y$
onto \smash{$X_{A_{k-1}}$}.  By expressing  
\smash{$\tilde{X}_j = P_{A_{k-1}}^\perp X_j$} 
and \smash{$r = P_{A_{k-1}}^\perp X_j$}, where
\smash{$P_{A_{k-1}}^\perp$} projects onto the orthocomplement
of the column space of \smash{$X_{A_{k-1}}$}, we can rewrite the above
constraints as 
\begin{equation*}
\big(s_{k,k} P_{A_{k-1}}^\perp X_{j_k}/
\|P_{A_{k-1}}^\perp X_{j_k}\|_2^2 \pm  
P_{A_{k-1}}^\perp X_j/\|P_{A_{k-1}}^\perp X_j\|_2^2 \big)^T y \geq 0 
\;\;\; \text{for all $j \not= j_1, \ldots,j_k$},
\end{equation*}
a set of $2(d-k)$ linear inequalities in $y$.
Meanwhile, the subevent 
\smash{$\hat{s}_k(y) = [s_{k,1}, \ldots,s_{k,k}]$} can be
characterized by $k$ inequalities expressed in block form, 
\begin{equation*}
\mathrm{diag}(s_{1,1}, \ldots,s_{k,k}) \, (X_{A_k}^T X_{A_k})^{-1}
X_{A_k}^T y \geq 0.
\end{equation*}
This completes the proof. 

\subsection{Proof of Lemma \ref{lem:master1}}
\label{app:master1}

We prove the result for FS; the results for the LAR and lasso paths 
follows similarly, by inpsecting the form of the linear inequalities
that determine their selection events.

Consider the first FS step as described in Appendix
\ref{app:partition}.  Multiplying through by $\sqrt{n}$, we see that
an equivalent set of inequalities that characterize the selection
event \smash{$\hat{j}_1(y)=j_1$}, \smash{$\hat{s}_1(y)=s_1$} is
\begin{equation*}
s_1 \frac{n}{X_{j_1}^T X_{j_1}} \frac{X_{j_1}^T y}{\sqrt{n}} 
\pm \frac{n}{X_j^T X_j} \frac{X_j^T y}{\sqrt{n}} \geq 0
\;\;\; \text{for all $j\not=j_1$}.
\end{equation*}
This is clearly of the desired form 
\smash{$P_1(\frac{1}{n} X^T X) \, \frac{1}{\sqrt{n}}
  X^T y \geq 0$}, for a matrix \smash{$P_1(\frac{1}{n} X^T X)$}
dependent only on \smash{$\frac{1}{n} X^T X$}.  At the $k$th step of
FS, there are two sets of inequalities to be examined: one that
describes the variable to enter \smash{$\hat{j}_k(y)=j_k$}, and the
second that describes the active signs 
\smash{$\hat{s}_k(y) = [s_{k,1}, \ldots,s_{k,k}]$}.  The first set,
multiplying through by $\sqrt{n}$, is
\begin{multline*}
s_{k,k} \frac{n X_{j_k}^T X_{A_{k-1}} (X_{A_{k-1}^T} X_{A_{k-1}})^{-1}}
{X_{j_k}^T X_{A_{k-1}} (X_{A_{k-1}^T} X_{A_{k-1}})^{-1} X_{A_{k-1}}^T
  X_{j_k}} 
\frac{X_A^T y}{\sqrt{n}} \pm 
\frac{n X_j^T X_{A_{k-1}} (X_{A_{k-1}^T} X_{A_{k-1}})^{-1}}
{X_j^T X_{A_{k-1}} (X_{A_{k-1}^T} X_{A_{k-1}})^{-1} X_{A_{k-1}}^T
  X_j} 
\frac{X_A^T y}{\sqrt{n}} \geq 0 \\
\text{for all $j \not= j_1, \ldots,j_k$},
\end{multline*}
while the second set, again multiplying through by $\sqrt{n}$, is 
\begin{equation*}
\mathrm{diag}(s_{1,1}, \ldots,s_{k,k}) \, n (X_{A_k}^T X_{A_k})^{-1}
\frac{X_{A_k}^T y}{\sqrt{n}} \geq 0.
\end{equation*}
These inequalities are clearly all summarized by
\smash{$P_k(\frac{1}{n} X^T X) \, \frac{1}{\sqrt{n}} 
  X^T y \geq 0$}, where \smash{$P_k(\frac{1}{n} X^T X)$} is a matrix
that depends only on \smash{$\frac{1}{n} X^T X$}.  This completes the 
proof. 

\subsection{Proof of Lemma \ref{lem:master2}}
\label{app:master2}

Under the conditions of the lemma, the TG pivot for fixed $M$ in
\eqref{eq:cond_pivot} depends only on $X,y$ through the master 
statistic, because, as explained above the lemma, the only dependence 
in the pivot on $X,y$ is through the quantities 
$v^T y/\|v\|_2$, $(Q_M(X) \, v)/\|v\|_2$, 
$Q_M(X) \, y$, and each of these is in turn a function of the master   
statistic $\Omega_n$.  Moreover, we may reexpress the TG statistic in 
\eqref{eq:cond_pivot_defn} as  
\begin{equation*}
T(X,y;M,v,\mu) = 
\frac{\Phi\big(f_1(\frac{1}{n}X^T X, \frac{1}{\sqrt{n}} X^T y)\big) -   
\Phi\big(f_2(\frac{1}{n}X^T X, \frac{1}{\sqrt{n}} X^T y)\big)}
{\Phi\big(f_1(\frac{1}{n}X^T X, \frac{1}{\sqrt{n}} X^T y)\big)  - 
\Phi\big(f_3(\frac{1}{n}X^T X, \frac{1}{\sqrt{n}} X^T y)\big)},
\end{equation*}
for some functions $f_1,f_2,f_3$, or more succinctly, as  
\smash{$T(X,y;M,v,\mu) = \psi_M (\frac{1}{n}X^T X, 
\frac{1} {\sqrt{n}}X^T y)$}, where
\begin{equation*}
\psi_M(S,z) = \frac{\Phi\big(f_1(S,z)\big) - \Phi\big(f_2(S,z)\big)} 
{\Phi\big(f_1(S,z)\big) - \Phi\big(f_3(S,z)\big)}.
\end{equation*}
Note that the quantities $v^T y/\|v\|_2$, $(Q_M(X) \, v)/\|v\|_2$, 
$Q_M(X) \, y$ depend smoothly on the master statistic
\smash{$\Omega_n=(\frac{1}{n}X^T X, \frac{1}{\sqrt{n}}X^T y)$} at any point such 
that \smash{$\frac{1}{n} X^T X$} is nonsingular.  This implies $f_1,f_2,f_3$
are smooth functions of $(S,z)$ at any point such that $S$ is nonsingular. 
Lastly, for all $S,z$ such that $P_M(S) \,z > 0$, we have $f_1(S,z) > f_3(S,z)$,
and thus the denominator of $\phi_M(S,z)$ is positive.  This proves the desired 
continuity result on $\psi_M$. 

\subsection{Proof of Lemma \ref{lem:master3}}
\label{app:master3}

For the master statistic \smash{$\Omega_n=(\frac{1}{n}X^T X,\frac{1}{\sqrt{n}}
  X^T Y)$}, note that \smash{$\E(\Omega_n)=(\frac{1}{n}X^T X,\frac{1}{\sqrt{n}}
  X^T \theta)$}.  As $v$ is assumed to be chosen such that $v^T \theta$ is a
normalized regression coefficient from the projection of $\theta$ onto some
subset of the columns in $X$, we may assume without a loss of generality that  
$v^T \theta$ is as in \eqref{eq:proj_coef} for some $A,j$.  Then, we see that
we must only define   
\begin{equation*}
g(S,z)=\frac{e_j^T (S_{A,A})^{-1} z_A} 
{\sqrt{e_j^T (S_{A,A})^{-1} e_j}}, 
\end{equation*}
where we use \smash{$S_{A,A}$} to denote the submatrix of $S$ with rows in $A$ 
and columns in $A$, and $z_A$ to denote the subvector of $z$ with entries in
$A$. 

\subsection{Proof of Lemma \ref{lem:master}}
\label{app:master}

Define \smash{$Z_{0,n}=\sum_{i=1}^n \xi_i$}, where
\smash{$\xi_i=\frac{1}{\sqrt{n}}x_i \epsilon_i$}, 
$x_i$ is the $i$th row of $X$, and $\epsilon_i=Y_i-\theta_i$, for 
$i=1,\ldots,n$. Note that $(\xi_1,\ldots,\xi_n) \sim F_n(0)$, with independent,
mean zero components. 
We compute
\begin{equation*}
\sum_{i=1}^n \Cov(\xi_i) = \frac{\sigma^2}{n}
\sum_{i=1}^n x_ix_i^T = \frac{\sigma^2}{n} X^T X,
\end{equation*}
which converges to $\sigma^2 \Sigma$ as $n\to\infty$, by assumption. 
Further, for any $\delta>0$, consider
\begin{equation*}
\sum_{i=1}^n \E \Big( \|\xi_i\|_2^2 \cdot 
1\{\|\xi_i\|_2 \geq \delta \} \Big) = 
\frac{1}{n} \sum_{i=1}^n \|x_i\|_2^2 \; \E \Bigg(  
\epsilon_i^2 \cdot 1\Bigg\{\frac{\|x_i\|_2}{\sqrt{n}}|\epsilon_i| 
  \geq \delta \Bigg\}  \Bigg).
\end{equation*}
We seek to show that this converges to 0 as $n\to\infty$.  As 
\smash{$\frac{1}{n}\sum_{i=1}^n \|x_i\|_2^2 \to \tr(\Sigma)$}, it
suffices to show that the maximum of the above expectations (in the 
summands) 
converges to 0, which is implied by the assumption that
\smash{$\max_{i=1,\ldots,n} \|x_i\|_2 / \sqrt{n} \to 0$.}  As the
above arguments did not depend on the sequence $F_n(0)$,
$n=1,2,3,\ldots$, we have verified the Lindeberg-Feller conditions 
uniformly, and hence the uniform Lindeberg-Feller central limit 
theorem, Lemma \ref{lem:clt}, implies that \smash{$Z_{0,n}$}
converges in distribution to $Z_0 \sim N(0,\sigma^2 \Sigma)$, uniformly over
$\cP_n(0)$. 

Now consider \smash{$Z_n =\frac{1}{\sqrt{n}} X^T Y = Z_{0,n}+\frac{1}{\sqrt{n}}
  X^T \theta$}. Writing $\Phi$ and $\phi$ for the standard normal CDF and
density, 
\begin{align*}
\sup_{\theta \in \Theta} \; \sup_{F_n(\theta) \in \cP_n(\theta)} \;
\sup_{x \in \R^d} \; &\big| \P(Z_n \leq x) - \P(Z \leq x)\big| \\
&= \sup_{\theta \in \Theta} \; \sup_{F_n(\theta) \in \cP_n(\theta)} \;
\sup_{x \in \R^d} \; \Bigg| \P\Bigg(Z_{0,n} \leq x-
\frac{1}{\sqrt{n}}X^T \theta \Bigg) - \P(Z \leq x) \Bigg| \\
&\leq \sup_{\theta \in \Theta} \; \sup_{F_n(\theta) \in \cP_n(\theta)} \;
\sup_{x \in \R^d} \; \big| \P(Z_{0,n} \leq x) - \P(Z_0 \leq x)\big| \;+\;  
\sup_{x \in \R^d} \;\Bigg|\Phi(x-\eta) - \Phi\Bigg(x-
\frac{1}{\sqrt{n}}X^T \theta\Bigg) \Bigg| \\
&\leq \underbrace{\sup_{\theta \in \Theta} \; 
  \sup_{F_n(\theta) \in \cP_n(\theta)} \; 
  \sup_{x \in \R^d} \; \big| \P(Z_{0,n} \leq x) - \P(Z_0 \leq
  x)\big|}_{a} \;+\;  
\underbrace{\vphantom{\sup_{x \in \R^d}}
\Bigg|\frac{1}{\sqrt{n}}X^T \theta-\eta\Bigg| \phi(0)}_{b}, 
\end{align*}
where the second line is due to the triangle inequality, and the third line is
due to the simple bound \smash{$|\Phi(x-t)-\Phi(x-s)|=|\int_{x-s}^{x-t}
  \phi(u)\,du| \leq |t-s| \phi(0)$}, for any $x,s,t$. Note that $a \to 0$ by the
argument at the start of this proof, and $b \to 0$ by assumption in
\eqref{eq:theta}.  This shows that $Z_n$ converges in distribution to $Z \sim
N(\eta,\sigma^2 \Sigma)$, uniformly over $\cP_n(\theta)$, and over $\theta \in
\Theta$.  

Lastly, we establish the conditional result.  By repeating the same arguments
as above, the uniform Lindeberg-Feller central limit theorem and condition 
\eqref{eq:theta} imply that \smash{$(Z_n, A_n Z_n)$} converges to $(Z,AZ)$,
uniformly over $\cP_n(\theta)$, and over $\theta \in \Theta$.  Thus, along 
sequence $F_n(\theta) \in \cP_n(\theta)$, $n=1,2,3,\ldots$ with $\theta \in
\Theta$, observe
\begin{equation*}
\P( Z_n \leq x \, | \, A_nZ_n \geq 0) = 
\frac{\P( Z_n \leq x, \, A_nZ_n \geq 0)}{\P(A_nZ_n \geq 0)} 
\to \frac{\P( Z \leq x, \, AZ \geq 0)}{\P(AZ \geq 0)},
\end{equation*}
at a rate that does not depend on the sequence in question.  This 
is true because the numerator and denominator each converge to their normal 
probability counterparts, and the denominator remains bounded away from zero 
since $\{z:Az \geq 0\}$ has nonempty interior, and the set of limits of
\smash{$\frac{1}{\sqrt{n}} X^T \theta$} was assumed compact, in
\eqref{eq:theta}.  Since $x$ was arbitrary, and the distribution of $Z \,|\, AZ 
\geq 0$ is continuous, we have (e.g., Lemma 2.11 in
\citet{vandervaart1998asymptotic}) 
\begin{equation*}
\sup_{x \in \R^d} \; \Big| \P( Z_n \leq x \, | \, A_nZ_n \geq 0) -
\P( Z \leq x \, | \, AZ \geq 0) \Big| \to 0.
\end{equation*}
And as the sequence $F_n(\theta) \in \cP_n(\theta)$, $n=1,2,3,\ldots$ with
$\theta \in \Theta$ was arbitrary, we have shown the desired uniform
convergence.     

\subsection{Proof of Theorem \ref{thm:pivot}}
\label{app:pivot}

We begin with the proof of part (a).   Let
\smash{$Z_n=\frac{1}{\sqrt{n}} X^T Y$} and $Z \sim N(\eta,\sigma^2 \Sigma)$.
Also, let \smash{$A_n=P_M(\frac{1}{n}X^T X)$} and \smash{$A=P_M(\Sigma)$}.
Recall that \smash{$A_nZ_n \geq 0 \iff \hat{M}(X,Y)=M$},
by Lemma \ref{lem:master1}. Also, 
\smash{$Z_n \,|\, A_nZ_n \geq 0$} converges weakly
to $Z \,|\, AZ \geq 0$, uniformly over $\cP_n(\theta)$ and over  
$\theta \in \Theta$, by Lemma \ref{lem:master}. As \smash{$\frac{1}{n}X^T X \to
  \Sigma$} deterministically, we also have that 
\smash{$\Omega_n=(\frac{1}{n}X^T X,Z_n)$} converges uniformly in distribution to
$\Omega=(\Sigma, Z)$.    

The choice of $v$ as specified in the 
theorem is now important for two reasons. First, by Lemma
\ref{lem:master2}, we can express 
\begin{equation*}
T(X,Y; M,v,\mu) = \psi_M (\Omega_n), 
\end{equation*}
for a function $\psi_M$. Second, by Lemma \ref{lem:master3}, 
we can express $v^T\theta = g(\E(\Omega_n))$ for a function $g$.    
Neither $\psi_M$ nor $g$ depend on $n$, and the distribution in 
question is that of $\psi_M (\Omega_n) \,|\, A_nZ_n \geq 0$ 
under $g(\E(\Omega_n))=\mu$.  The function $\psi_M$ is continuous at any point
$(S,z)$ such that $S$ is nonsingular and $Az>0$; recalling the assumed
nonsingularity of $\Sigma$, it is therefore continuous on a set of full
probability under the limiting distribution $\cL(\Omega \,|\, AZ \geq 0)$. By 
the uniform continuous mapping theorem, Lemma \ref{lem:cmt}, $\psi_M(\Omega_n) 
\,|\, A_nZ_n \geq 0$ converges uniformly to $\psi(\Omega) \,|\, AZ \geq 0$,
which is distributed as $U(0,1)$ when $g(\E(\Omega))=\mu$ by the pivotal
property of the TG statistic under normality, as in \eqref{eq:cond_pivot}.  
The proof of uniform validity of TG confidence intervals is
just a rearrangement of the uniform asymptotic pivotal statement. 

The proof of part (b) follows from the expansion 
\begin{equation*}
\cT(X,Y;V,U) = \sum_{M \in \cM} T(X,Y;M,v_M,\mu_M) \, 1\{ \hat{M}(X,Y)=M \}. 
\end{equation*}
As the number possible models $|\cM|$ is finite, we can simply 
apply the asymptotic pivotal result from part (a) to each $M \in \cM$ 
to establish the asymptotic pivotal property of $\cT(X,Y;V,U)$.  The confidence
interval result is again just a rearrangement of this pivotal property.


\subsection{Proof of Lemma \ref{lem:vthree}}
\label{app:vthree}

By assumption, the vector $v$ can be written as
\begin{equation*}
v = \frac{X_A(X_A^T X_A)^{-1} e_j}
{\sqrt{e_j^T (X_A^T X_A)^{-1} e_j}}, 
\end{equation*}
for some $A,j$.  We compute
\begin{align*}
\|v\|_3^3 &= \frac{\sum_{i=1}^n |X_{i,A}(X_A^T X_A)^{-1} e_j|^3}
{|e_j^T (X_A^T X_A)^{-1} e_j|^{3/2}} \\
&= \frac{\frac{1}{n^{3/2}} \sum_{i=1}^n |X_{i,A}n(X_A^T X_A)^{-1} e_j|^3}  
{|e_j^T n (X_A^T X_A)^{-1} e_j|^{3/2}}.
\end{align*}
The denominator converges to $|e_j^T(\Sigma_{A,A})^{-1}e_j|^{3/2}$ by 
\eqref{eq:xcov}.  The numerator satisfies
\begin{equation*}
\frac{1}{n^{3/2}} \sum_{i=1}^n |X_{i,A}n(X_A^T X_A)^{-1} e_j|^3 \leq 
\frac{1}{\sqrt{n}} \;\cdot\;
\underbrace{\frac{1}{n}\sum_{i=1}^n \|x_i\|_2^3}_a \;\cdot\;
\underbrace{\vphantom{\frac{1}{n}\sum_{i=1}^n}
\|n(X_A^T X_A)^{-1} e_j\|_2^3}_b,  
\end{equation*}
where $a$ is bounded by \eqref{eq:xthree} and $b$ converges to 
$\|(\Sigma_{A,A})^{-1} e_j\|_2^3$ by \eqref{eq:xcov}.  This completes
the proof. 

\subsection{Proof of Lemma \ref{lem:rysy}}
\label{app:rysy}

We start by proving the result about the event $\{cs_Y \geq \sigma\}$.  
First let us study its asymptotic probability marginally. Consider  
\begin{align*}
\E(s_Y^2) &= \frac{1}{n} \sum_{i=1}^n \E|\epsilon_i + \theta_i -
\bar\epsilon - \bar\theta|^2 \\
&= \frac{1}{n} \sum_{i=1}^n \E|\epsilon_i -\bar\epsilon|^2 + 
\frac{1}{n} \sum_{i=1}^n |\theta_i - \bar\theta|^2 + 
\frac{2}{n} \sum_{i=1}^n
\E(\epsilon_i-\bar\epsilon)(\theta_i-\bar\theta) \\
&= \frac{n-1}{n} \sigma^2 + s_\theta^2,
\end{align*}
where \smash{$\bar\epsilon=\sum_{i=1}^n \epsilon_i/n$}. 
Hence
\begin{align*}
\P\big( c s_Y \leq \sigma \big) 
&= \P\big(c^2 s_Y^2 -c^2\E(s_Y^2) \leq 
\sigma^2-c^2\E(s_Y^2)\big) \\
& \leq \frac{c^4\Var(s_Y^2)}
{(c^2\E(s_Y^2)-\sigma^2)^2},
\end{align*}
where in the last line we used Chebyshev's
inequality. Recalling that $c^2>1$, we have the lower bound 
$(c^2 \E(s_Y^2)-\sigma^2)^2 \geq 0.999(c^2-1)^2 \sigma^4$, for $n$
large enough.  Therefore, to show $\P(c s_Y \geq \sigma) \to 1$, it is enough to
show that $\Var(s_Y^2) \to 0$ as $n \to \infty$, uniformly.  For this, we will
use the simple inequality 
\begin{equation}
\label{eq:cov_ineq}
\Var(W_1+\ldots+W_m) \leq m \sum_{i=1}^m \Var(W_i),
\end{equation}
which follows from the fact that 
$2\Cov(W_i,W_j) \leq \Var(W_i)+\Var(W_j)$.  We will also invoke
Rosenthal's inequality \citep{rosenthal1970subspaces}, which for
independent $W_1,\ldots,W_m$, having mean zero and $\E|W_i|^t <
\infty$ for $i=1,\ldots,m$, states that  
\begin{equation}
\label{eq:rosenthal}
\E \Bigg|\sum_{i=1}^m W_i \Bigg|^t \leq C_t  
\max\Bigg\{ \sum_{i=1}^m \E|W_i|^t, \; \Bigg(\sum_{i=1}^m \E W_i^2
\Bigg)^{t/2} \Bigg\},
\end{equation}
for a constant $C_t>0$ only depending on $t$.  Hence, observe that
\begin{align*}
\Var(s_Y^2) &= \Var\Bigg(\frac{1}{n} \sum_{i=1}^n |\epsilon_i +
\theta_i - \bar\epsilon - \bar\theta|^2\Bigg) \\ 
&=\Var\Bigg(\frac{1}{n} \sum_{i=1}^n
|\epsilon_i + \theta_i - \bar\theta|^2 + 
\bar\epsilon^2 - \frac{2}{n} \sum_{i=1}^n (\epsilon_i + 
\theta_i - \bar\theta) \bar\epsilon \Bigg) \\ 
&= \Var\Bigg(\frac{1}{n} \sum_{i=1}^n
|\epsilon_i + \theta_i - \bar\theta|^2  
- \bar\epsilon^2 \Bigg) \\
&\leq \underbrace{2 \Var\Bigg(\frac{1}{n} \sum_{i=1}^n 
|\epsilon_i + \theta_i - \bar\theta|^2\Bigg)}_a + 
\underbrace{\vphantom{\Bigg(\frac{1}{n}\sum_{i=1}^n\Bigg|}  
2 \Var(\bar\epsilon^2)}_b,
\end{align*}
where in the last line we used \eqref{eq:cov_ineq}. We consider 
$a,b$ individually. We have
\begin{align*}
a &= \frac{2}{n^2} \Var\Bigg( \sum_{i=1}^n \epsilon_i^2 + 
\sum_{i=1}^n |\theta_i-\bar\theta|^2 + 2 \sum_{i=1}^n
\epsilon_i(\theta_i-\bar\theta) \Bigg) \\
&\leq \frac{6}{n^2} \Var\Bigg(\sum_{i=1}^n \epsilon_i^2\Bigg) + 
\frac{12}{n^2} \Var\Bigg( \sum_{i=1}^n 
\epsilon_i(\theta_i-\bar\theta)\Bigg) \\
&\leq \frac{6}{n} \kappa + \frac{12}{n} \sigma^2 s_\theta^2 
\to 0,
\end{align*}
where the second line again used \eqref{eq:cov_ineq}, and the third
used our assumptions on the error distribution in \eqref{eq:pn2}, and on
$\theta$ in \eqref{eq:theta2b}.  We also have
\begin{align*}
b &= \frac{2}{n^4} \E\Bigg|\sum_{i=1}^n \epsilon_i\Bigg|^4 \\
&\leq \frac{2}{n^4} C_4 
\max\{n \kappa, n^2 \sigma^4\} \to 0,
\end{align*}
where the second line used Rosenthal's inequality
\eqref{eq:rosenthal}. 
This implies $\Var(s_Y^2) \leq a+b \to 0$, uniformly over
$\cP'_n(\theta)$, and over $\theta \in \Theta'$. 

We have therefore shown $\P(c s_Y \geq \sigma) \to 1$, uniformly over
$\cP'_n(\theta)$, and over $\theta \in \Theta'$.  To see that the same result 
holds conditional on \smash{$\hat{M}(X,Y)=M$}, take any sequence
$F_n(\theta) \in \cP_n(\theta)$, $n=1,2,3,\ldots$ where $\theta \in \Theta'$,
and note that 
\begin{align*}
\P\Big(c s_Y \geq \sigma \,\Big|\, \hat{M}(X,Y)=M\Big) &=
\frac{ \P(c s_Y \geq \sigma \,, A_nZ_n \geq 0)}{\P(A_nZ_n \geq 0)} \\
&\geq \frac{ \P(A_nZ_n \geq 0) - \P(c s_Y < \sigma)}{\P(A_nZ_n \geq 0)} \\
&\to \frac{ \P(AZ \geq 0) - 0}{\P(AZ \geq 0)} = 1,
\end{align*}
where we have borrowed the notation and the normal convergence result 
$\P(A_nZ_n \geq 0) \to \P(AZ \geq 0)$ from the proof  
of Lemma \ref{lem:master3}. The rate of convergence in the last line does not
depend on the sequence in consideration, because of the uniform
convergence of $A_nZ_n$ to $AZ$, and the fact the denominator is bounded away 
from zero, since the set of limits of \smash{$\frac{1}{\sqrt{n}} X^T \theta$} is
assumed to be compact, in \eqref{eq:theta2a}. And as $F_n(\theta) 
\in \cP_n(\theta)$, $n=1,2,3,\ldots$ with $\theta \in \Theta'$ was arbitrary,
this completes the proof of the first part of the lemma.   

For the second part, on the boundedness of \smash{$r_Y^3/s_Y^3$}, consider that
for any $C>0$ we have
\begin{align*}
\P \Bigg( \frac{r_Y^3}{s_Y^3} < C \,\Bigg|\, \hat{M}(X,Y)=M \Bigg)  
&\geq \P \Big( r_Y^3 < \sigma^3C/c^3, \, s_Y^3 \geq
\sigma^3/c^3 \,\Big|\, \hat{M}(X,Y)=M \Big) \\ 
&\geq 1- \P \Big( r_Y^3 \geq \sigma^3C/c^3 \,\Big|\, \hat{M}(X,Y)=M \Big) -   
\P \Big(s_Y^3 < \sigma^3/c^3 \,\Big|\, \hat{M}(X,Y)=M \Big). 
\end{align*}
The last term here satisfies 
\smash{$\P(s_Y^3 < \sigma^3/c^3 \,|\, \hat{M}(X,Y)=M) \to 0$}, uniformly, by
what we showed above. It suffices to prove that, for any $\delta>0$,
there exists $C>0$ such that \smash{$\P ( r_Y^3 > c^3C/\sigma^3 \,|\, 
  \hat{M}(X,Y)=M) \leq \delta$} for large enough $n$, uniformly.  By Markov's 
inequality, this will be true as long as $\E (r_Y^3 \,|\, \hat{M}(X,Y)=M)$ is
uniformly bounded. To this end, we will use the simple inequality,
\begin{equation}
\label{eq:pow_ineq}
|a+b|^t \leq 2^t |a|^t + 2^t |b|^t,
\end{equation}
and compute
\begin{align}
\nonumber
\E\Big(r_Y^3 \Big|\, \hat{M}(X,Y)=M \Big) &= \frac{1}{n} \E 
\Bigg(\sum_{i=1}^n |\epsilon_i+\theta_i-\bar\epsilon-\bar\theta|^3
 \,\Bigg|\,\hat{M}(X,Y)=M \Bigg) \\ 
\nonumber
&\leq \frac{2^3}{n} \E \Bigg(\sum_{i=1}^n
|\epsilon_i-\bar\epsilon|^3 \,\Bigg|\,\hat{M}(X,Y)=M \Bigg) + 
2^3 r_\theta^3 \\
\nonumber
&\leq \frac{2^6}{n} \E \Bigg(\sum_{i=1}^n
|\epsilon_i|^3 \,\Bigg|\,\hat{M}(X,Y)=M \Bigg) +
  2^6 \E\Big(|\bar\epsilon|^3 \,\Big|\,\hat{M}(X,Y)=M \Big) + 
  2^3  r_\theta^3 \\
\label{eq:ry3}
&\leq 2^6 \tau_M + \frac{2^6}{n^3} C_3  \max\{
n\tau_M, n^{2/3} \sigma_M^3 \} + 2^3 r_\theta^3,
\end{align}
where the second and third lines used \eqref{eq:pow_ineq}, and the
last line used Rosenthal's inequality \eqref{eq:rosenthal}, along with the
abbreviations 
\begin{equation*}
\tau_M = \frac{1}{n} \E \Bigg(\sum_{i=1}^n
|\epsilon_i|^3 \,\Bigg|\,\hat{M}(X,Y)=M \Bigg), \;\;\;\text{and}\;\;\;
\sigma^2_M = \frac{1}{n} \E \Bigg(\sum_{i=1}^n
|\epsilon_i|^2 \,\Bigg|\,\hat{M}(X,Y)=M \Bigg).
\end{equation*}
Once again using \smash{$A_nZ_n \geq 0 \iff \hat{M}(X,Y)=M$} and  
the uniform convergence $\P(A_nZ_n \geq 0) \to \P(AZ \geq 0)$ from Lemma
\ref{lem:master3}, we have for large enough $n$, 
\begin{equation*}
\tau_M \leq \frac{\E |\epsilon_1|^3}{\P(A_nZ_n \geq 0)} \leq \frac{\tau}{\P(AZ
  \geq 0)/2} \leq \frac{\tau}{\rho/2},
\end{equation*}
where we have used the upper bound on the third moment of the error
distribution in \eqref{eq:pn2}, and we have used a lower bound $\P(AZ \geq 0)
\geq \rho > 0$ that holds uniformly over all $\theta \in \Theta'$, due to the
assumed compactness of the set of limits of \smash{$\frac{1}{\sqrt{n}} X^T
  \theta$}, in \eqref{eq:theta2a}.  Thus we have shown that $\tau_M$ is
uniformly upper bounded.  Similar arguments show that $\sigma_M$ is uniformly 
upper bounded.  As $r_\theta^3 \leq R$ by assumption in \eqref{eq:theta2b}, we
see from \eqref{eq:ry3} that \smash{$\E (r_Y^3 \,|\, \hat{M}(X,Y)=M)$} is
uniformly upper bounded. This completes the proof of the second part, and the
lemma.  

\subsection{Proof of Lemma \ref{lem:basic}}
\label{app:basic}

Let us write 
\begin{equation*}
\frac{v^T (Y^* - \bar{Y}\mathbb{1})} 
{s_Y} = \sum_{i=1}^n \xi_i,
\end{equation*}
where $\xi_1,\ldots,\xi_n$ are independent with mean zero and
$\sum_{i=1}^n \Var_*(\xi_i)=1$.  By Theorem 3.7 of
\citet{chen2011normal},  
\begin{equation*}
\sup_{t \in \R} \; \big|
\P_*\big( \sum_{i=1}^n \xi_i \leq t \big) 
- \P \big(Z \leq  t\, \big| \, Y \big) \big| 
\leq 10 \sum_{i=1}^n \E_* |\xi_i|^3.  
\end{equation*}
But the right-hand side is precisely
\begin{equation*}
10 \sum_{i=1}^n \E_* |\xi_i|^3 = 10
\frac{r_Y^3}{s_Y^3}  \|v\|_3^3. 
\end{equation*}
Lemmas \ref{lem:vthree} and \ref{lem:rysy} imply that this
is \smash{$O_\P(1/\sqrt{n})$} conditional on \smash{$\hat{M}(X,Y)=M$}, uniformly
over $\cP'_n(\theta)$, and over $\theta \in \Theta'$, giving the result.       

\subsection{Proof of Theorem \ref{thm:unknown}}
\label{app:unknown}

First, we prove the result for the plug-in statistic.  Denoting 
$Z \sim N(0,1)$, we have
\begin{equation*}
\tilde{T}(X,Y;M,v,0) = \P\Big(c s_Y  Z \geq v^T Y 
\,\Big|\, \hat{a}_M \leq c s_Y  Z \leq \hat{b}_M, \, Y \Big).
\end{equation*}
Consider the event $\{cs_Y \geq \sigma\}$, which has probability approaching
1 conditional on \smash{$\hat{M}(X,Y)=M$}, uniformly over $\cP'_n(\theta)$, and 
over $\theta \in \Theta'$, by Lemma \ref{lem:rysy}.
On this event, by the monotonicity of the truncated Gaussian survival function
in its variance parameter, shown in Appendix 
\ref{app:trun_sigma}, we can replace $cs_Y$ by $\sigma$, and this 
cannot increase the value of the statistic. (To verify that the result in
Appendix \ref{app:trun_sigma} can indeed be applied, notice that 
\smash{$\hat{a}_M \geq 0$}, i.e., the left endpoint of the interval is
at least the mean of the truncated Gaussian, which follows from the fact that
\smash{$v^T Y \geq 0$} by design.) Thus we can write    
\begin{equation*}
\tilde{T}(X,Y;M,v,0) = \P\Big(\sigma Z \geq v^T Y 
\,\Big|\, \hat{a}_M \leq \sigma  Z \leq \hat{b}_M, \, Y \Big) + E_n,
\end{equation*}
where $\P(E_n<0 \,|\, \hat{M}(X,Y)=M) \to 0$, uniformly over $\cP'_n(\theta)$,
and over $\theta \in \Theta'$.  Hence, for any $t \in [0,1]$, 
\begin{equation*}
\P_{v^T \theta = 0} \Big( \tilde{T}(X,Y;M,v,0) \leq t 
\,\Big|\, \hat{M}(X,Y)=M \Big) \leq 
\P_{v^T \theta = 0} \Big( T(X,Y; M,v, 0) \leq t 
\,\Big|\, \hat{M}(X,Y)=M \Big) + o(1),
\end{equation*}
where the $o(1)$ remainder term above is uniform over $t \in [0,1]$, 
over $\cP'_n(\theta)$, and over $\theta \in \Theta'$. Applying part (a) of
Theorem \ref{thm:pivot} proves the conditional result for the plug-in statistic.   

Next, we turn to the bootstrap result, whose proof is a little more    
involved.  Define a function 
\begin{equation*}
G^*(z) = \frac{\P_*\big(z \leq c v^T (Y^* - \bar{Y}\mathbb{1})     
  \leq \hat{b}_M\big) + \delta_n}   
{\P_*\big(\hat{a}_M \leq c  v^T (Y^* - \bar{Y}\mathbb{1})     
  \leq \hat{b}_M\big) + \delta_n} \cdot
1\big\{ \hat{a}_M \leq z \leq \hat{b}_M \big\}.
\end{equation*}
Lemma \ref{lem:basic} implies that we can write
\begin{equation*}
G^*(z) = \frac{\P\big(z \leq c s_Y  Z   
  \leq \hat{b}_M \,\big|\, Y \big) + E_n + \delta_n}   
{\P\big(\hat{a}_M \leq c s_Y  Z  
  \leq \hat{b}_M \,\big|\, Y \big) + E'_n + \delta_n} \cdot 
1\big\{ \hat{a}_M \leq z \leq \hat{b}_M \big\},
\end{equation*}
where $|E_n|,|E_n'|=O_\P(1/\sqrt{n})$ conditional on \smash{$\hat{M}(X,Y)=M$},
uniformly over $z \in \R$, over $\cP'_n(\theta)$, and over $\theta \in \Theta'$.
(Note that $c$ in the above can be absorbed into the
role of $t$ in the lemma.) Dividing through by the quantity 
\smash{$\P(\hat{a}_M \leq cs_Y  Z \leq \hat{b}_M \,|\, Y)+
  \delta_n$}, we have  
\begin{align*}
\nonumber
G^*(z) &= \frac{\displaystyle
\frac{\P\big(z \leq c s_Y  Z   
  \leq \hat{b}_M \,\big|\, Y \big) + \delta_n}
{\P\big(\hat{a}_M \leq c s_Y  Z  
  \leq \hat{b}_M \,\big|\, Y \big) + \delta_n} 
+ \frac{E_n}
{\P\big(\hat{a}_M \leq c s_Y  Z  
  \leq \hat{b}_M \,\big|\, Y \big) + \delta_n}}
{\displaystyle 1 + \frac{E_n'}
{\P\big(\hat{a}_M \leq c s_Y  Z  
  \leq \hat{b}_M \,\big|\, Y \big) + \delta_n}} \cdot
1\big\{ \hat{a}_M \leq z \leq \hat{b}_M \big\} \\
&= \frac{\P\big(z \leq c s_Y  Z   
  \leq \hat{b}_M \,\big|\, Y \big) + \delta_n}
{\P\big(\hat{a}_M \leq c s_Y  Z  
  \leq \hat{b}_M \,\big|\, Y \big) + \delta_n} \cdot
1\big\{ \hat{a}_M \leq z \leq \hat{b}_M \big\} + e_n \\
&\geq \P\Big( cs_Y  Z  \geq z \, \Big| \, 
\hat{a}_M \leq cs_Y  Z  \leq \hat{b}_M, \, Y \Big) \cdot   
1\big\{ \hat{a}_M \leq z \leq \hat{b}_M \big\} + e_n \\ 
&\geq \P \Big( \sigma  Z  \geq z \, \Big| \, 
\hat{a}_M \leq \sigma  Z  \leq \hat{b}_M, \, Y \Big) \cdot  
1\big\{ \hat{a}_M \leq z \leq \hat{b}_M \big\} + e_n, 
\end{align*}
where $|e_n|=o_\P(1)$ conditional on \smash{$\hat{M}(X,Y)=M$}, uniformly 
over $z \in \R$, over $\cP'_n(\theta)$, and over $\theta \in \Theta'$, but  
the precise value of $e_n$ may differ from line to line.  Above, in the second
line, we used $E_n/\delta_n=o_\P(1)$ conditional on \smash{$\hat{M}(X,Y)=M$}, 
uniformly, and similarly for $E_n'$; in the third line, we used the fact that 
$(p+\delta)/(q+\delta) \geq p/q$ for $0< p\leq q$ and $\delta \geq 0$; in the
last line, we have used, as before, the monotonicity of the truncated Gaussian
survival function in its underlying variance parameter, and the fact that
\smash{$\P(c s_Y \geq \sigma \,|\, \hat{M}(X,Y)=M) \to 1$}, uniformly.

Rewriting the result in the last display, we have
\begin{equation*}
\sup_{z \in \R} \;
\Bigg[G^*(z) -  \P \Big( \sigma  Z  \geq z \, \Big| \, 
\hat{a}_M \leq \sigma  Z  \leq \hat{b}_M, \, Y \Big) \cdot  
1\big\{ \hat{a}_M \leq z \leq \hat{b}_M \big\} \Bigg]_- \leq \; 
|e_n|,
\end{equation*}
where $x_-=\max\{0,-x\}$ denotes the negative part of $x$.
In particular, at \smash{$z=v^T Y$}, this implies 
\begin{equation*}
\Bigg[T^*(X,Y; M,v, 0) - T(X,Y; M,v, 0)\Bigg]_- \leq 
\; |e_n|.
\end{equation*}
Finally, this means that we can write, at an arbitrary level $t \in
[0,1]$, 
\begin{equation*}
\P_{v^T \theta = 0} \Big( T^*(X,Y; M,v,0) \leq t 
\,\Big|\, \hat{M}(X,Y)=M \Big) = 
\P_{v^T \theta = 0} \Big( T(X,Y; M,v,0) \leq t - E_n''
\,\Big|\, \hat{M}(X,Y)=M \Big),
\end{equation*}
where $(E''_n)_- =o_\P(1)$ conditional on \smash{$\hat{M}(X,Y)=M$}, uniformly
over $t \in [0,1]$, over $\cP'_n(\theta)$, and over $\theta \in \Theta'$. 
Therefore 
\begin{equation*}
\P_{v^T \theta = 0} \Big( T^*(X,Y; M,v,0) \leq t 
\,\Big|\, \hat{M}(X,Y)=M \Big) \leq 
\P_{v^T \theta = 0} \Big( T(X,Y; M,v,0) \leq t
\,\Big|\, \hat{M}(X,Y)=M \Big) + o(1),
\end{equation*}
where the $o(1)$ term above is uniform over $t \in [0,1]$, over 
$\cP'_n(\theta)$, and over $\theta \in \Theta'$.  Applying
part (a) of Theorem \ref{thm:pivot} proves the conditional result for bootstrap
statistic.  

The unconditional results for two modified TG statistics hold simply 
by marginalization.

\subsection{Monotonicity of the truncated Gaussian distribution 
in $\sigma^2$} 
\label{app:trun_sigma}

Define 
\begin{equation*}
\bar{F}_{0,\sigma^2}^{[a,b]}(x) =
\frac{\Phi(b/\sigma)-
\Phi(x/\sigma)}
{\Phi(b/\sigma)-
\Phi(a/\sigma)}, 
\end{equation*}
the survival function for a normal random variable $Z \sim
N(0,\sigma^2)$, truncated to lie in an interval $[a,b]$, where 
$a \geq 0$.  We will show, following the proof of a similar
monotonicity result in Lemma A.1 of \citet{lee2016exact}, that for 
any $0<\sigma_1^2 < \sigma_2^2$,    
\begin{equation*} 
\bar{F}_{0,\sigma_1^2}^{[a,b]}(x) < \bar{F}_{0,\sigma_2^2}^{[a,b]}(x)
\;\;\; \text{for all $x \in [a,b]$}.
\end{equation*} 
To emphasize, the above property is only true when
the interval $[a,b]$ lies to the right of 0. Without this restriction,
the survival function will not be monotone increasing in
$\sigma^2$ (if $[a,b]$ contains 0, then it will generally be
nonmonotone, and if $[a,b]$ lies to the left of 0, then it will
actually be monotone {\it decreasing}).  

Over $\sigma^2>0$, the family of
distributions \smash{$\bar{F}_{0,\sigma^2}^{[a,b]}$} forms an
exponential family with natural parameter $1/\sigma^2$, as it is just
a family of Gaussian distributions with the carrier measure changed. 
Therefore, it has a monotone likelihood ratio in its sufficient
statistic $-x^2$, i.e., if we denote by 
\smash{$f_{0,\sigma^2}^{[a,b]}$} the truncated Gaussian density
function, and we fix $\sigma_1^2 < \sigma_2^2$, and  
$a \leq x_1<x_2 \leq b$, then
\begin{equation*}
\frac{f_{0,\sigma_1^2}^{[a,b]}(x_2)}{f_{0,\sigma_2^2}^{[a,b]}(x_2)} <  
\frac{f_{0,\sigma_1^2}^{[a,b]}(x_1)}{f_{0,\sigma_2^2}^{[a,b]}(x_1)}.
\end{equation*}
Hence 
\begin{equation*}
f_{0,\sigma_1^2}^{[a,b]}(x_2) \, f_{0,\sigma_2^2}^{[a,b]}(x_1) < 
f_{0,\sigma_1^2}^{[a,b]}(x_1) \, f_{0,\sigma_2^2}^{[a,b]}(x_2),
\end{equation*}
Integrating with respect to $x_1$ over $[a,x)$, for some
$x<x_2$, we obtain
\begin{equation*}
f_{0,\sigma_1^2}^{[a,b]}(x_2)\, 
\Big(1-\bar{F}_{0,\sigma_2^2}^{[a,b]}(x)\Big) <   
\Big(1-\bar{F}_{0,\sigma_1^2}^{[a,b]}(x)\Big)\,
f_{0,\sigma_2^2}^{[a,b]}(x_2).
\end{equation*}
Now integrating with respect to $x_2$, over $(x,b]$, we obtain
\begin{equation*}
\bar{F}_{0,\sigma_1^2}^{[a,b]}(x)\,
\Big(1-\bar{F}_{0,\sigma_2^2}^{[a,b]}(x)\Big) < 
\Big(1-\bar{F}_{0,\sigma_1^2}^{[a,b]}(x)\Big)\,
\bar{F}_{0,\sigma_2^2}^{[a,b]}(x).
\end{equation*}
Rearranging gives the result.

\subsection{P-value examples for correlated predictors}
\label{app:pval_ex_cor}

Here we investigate the consequences of using correlated predictors in
the simulation setup of Section \ref{sec:pval_ex}.   
We constructed a preliminary matrix 
\smash{$X \in \R^{50 \times 10}$} as before: each column was
drawn independently to have either i.i.d.\ $N(0,1)$,
$\mathrm{Bern}(0.5)$, or $SN(0,1,5)$ entries, with equal
probability. We then took as our predictor matrix \smash{$X' = X
  \Sigma^{1/2}$}, where $\Sigma \in \R^{10\times 10}$ has all diagonal  
entries equal to 1 and all off-diagonal entries equal to 0.5 (and 
$\Sigma^{1/2}$ is its symmetric square root). We scaled the
columns of $X'$ to have unit norm.  The rest of the setup is then just 
as in Section \ref{sec:pval_ex}.

Figure \ref{fig:lo_cor} shows the
results, in the same format as Figure \ref{fig:lo}: p-values for
LAR steps 1, 2, and 3, and pivotal statistics aggregated over LAR
steps, from 500 repetitions.  The p-values at steps 1 and 2 were
restricted to repetitions in which either variable 1 or 2 were
selected (now comprising about 70\% and 60\% of the repetitions, 
respectively); the 
p-values at step 3 were restricted to repetitions in which one of
variables 3 through 10 was selected (comprising about 80\% of the 
repetitions). Similar to the display in Figure \ref{fig:lo}, we
see power in the p-values from steps 1 and 2, albeit less power
than in the uncorrelated case, and uniform p-values in step 3, as
well as uniform pivotal statistics.

\begin{figure}[p]
\begin{subfigure}[b]{\textwidth}
\centering
Step 1, p-values \smallskip \\
\includegraphics[width=0.24\textwidth]{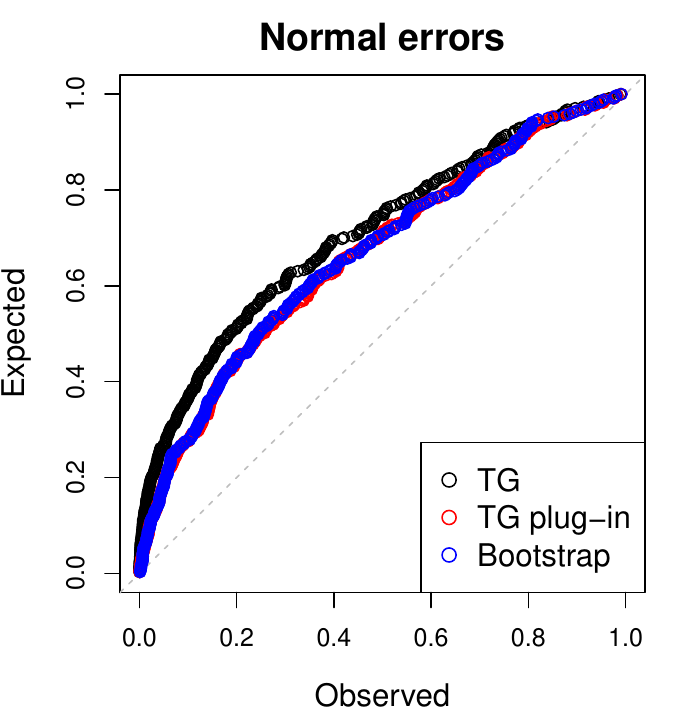}
\includegraphics[width=0.24\textwidth]{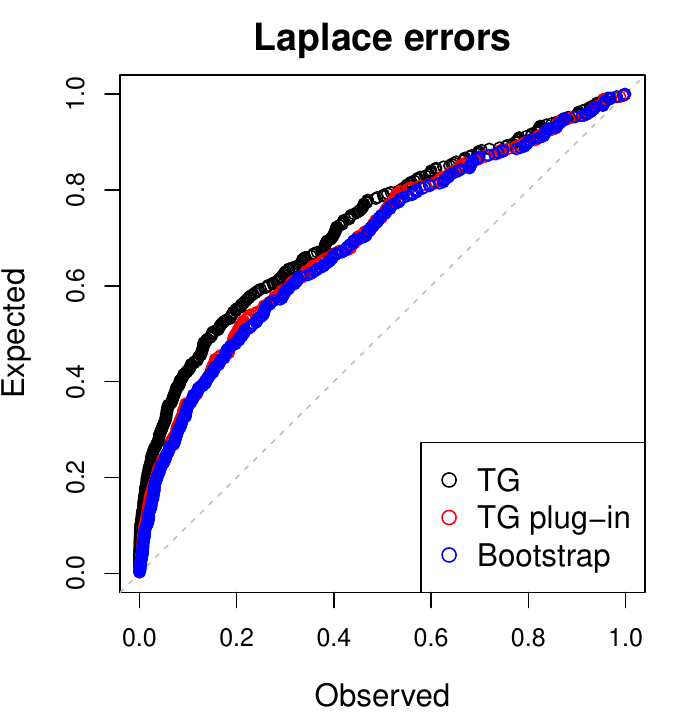}
\includegraphics[width=0.24\textwidth]{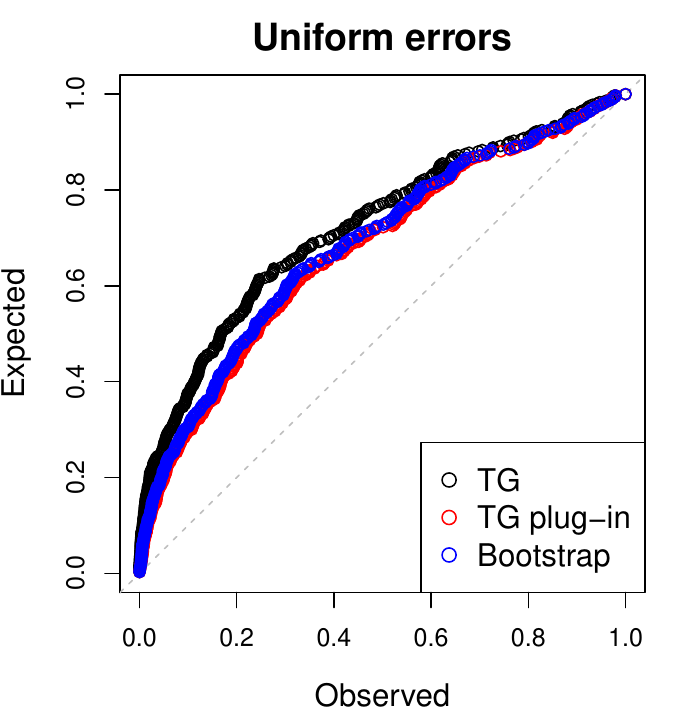}
\includegraphics[width=0.24\textwidth]{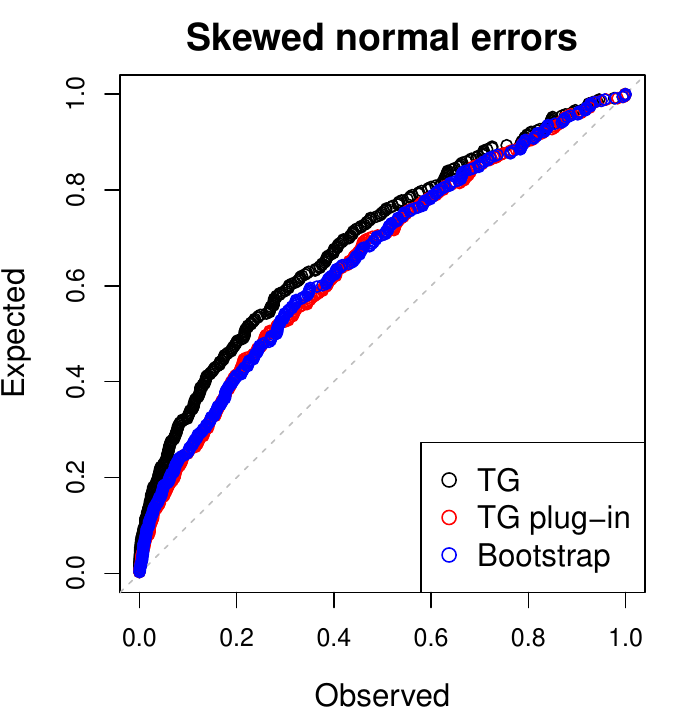} 
\smallskip \\
Step 2, p-values \smallskip \\
\includegraphics[width=0.24\textwidth]{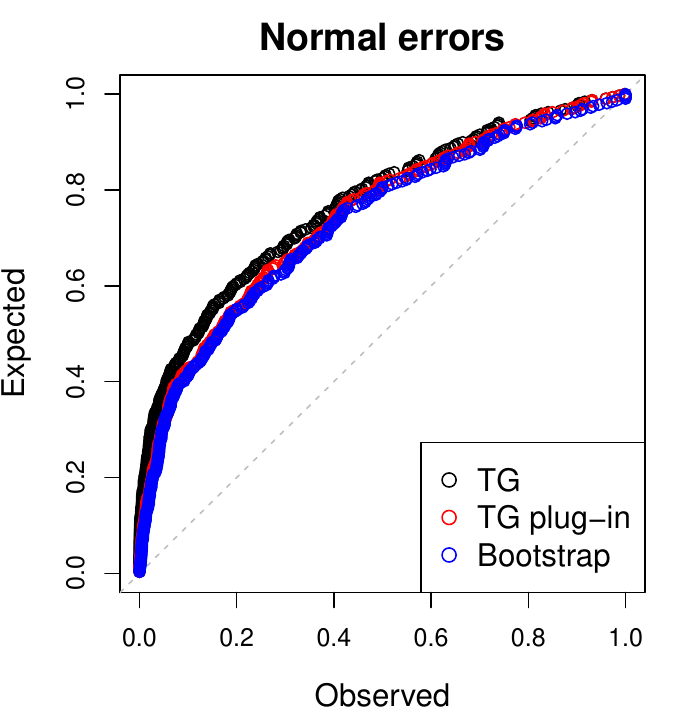}
\includegraphics[width=0.24\textwidth]{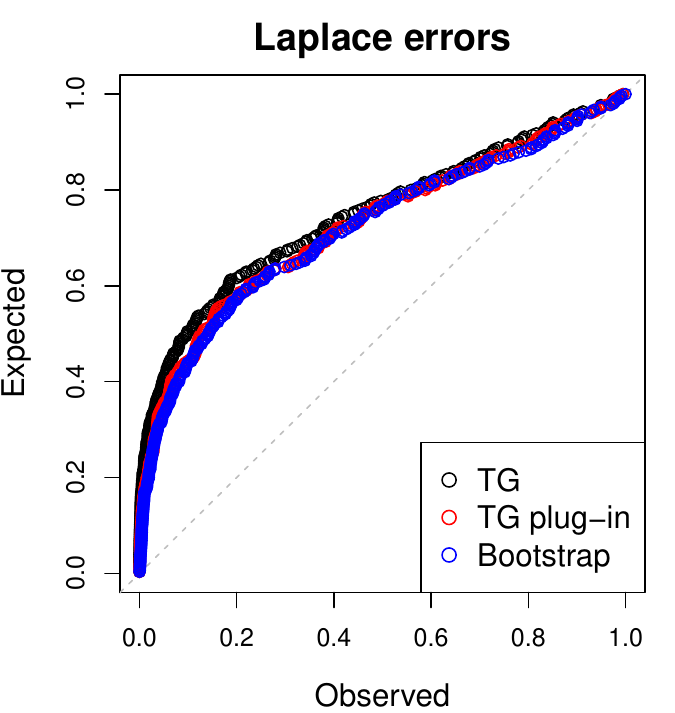} 
\includegraphics[width=0.24\textwidth]{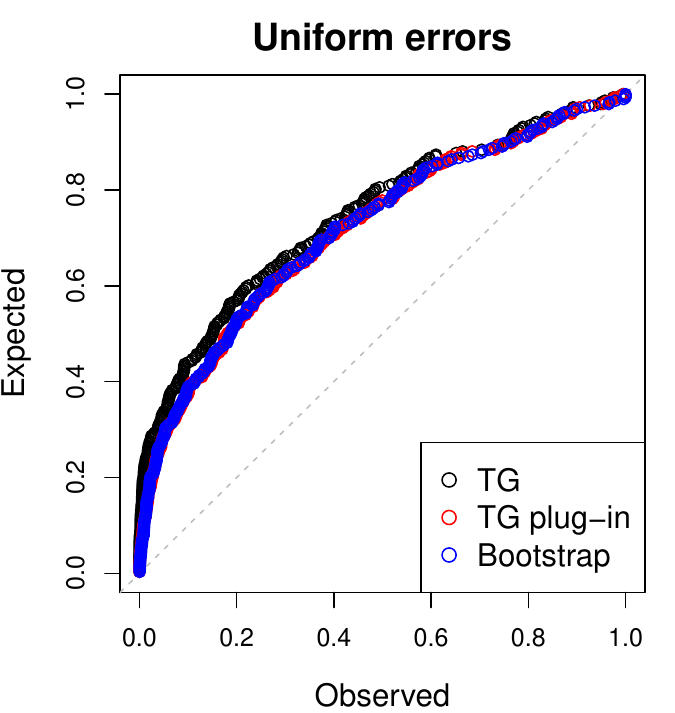} 
\includegraphics[width=0.24\textwidth]{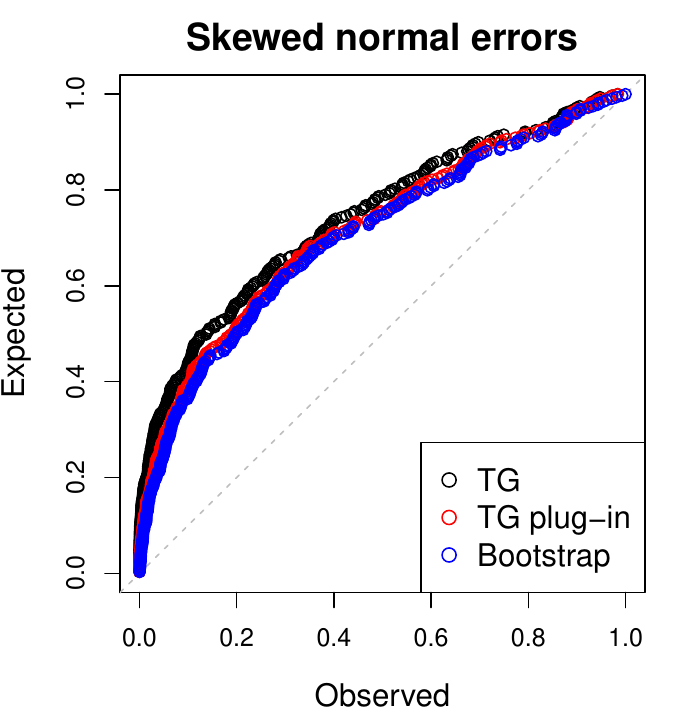} 
\smallskip \\
Step 3, p-values \smallskip \\ 
\includegraphics[width=0.24\textwidth]{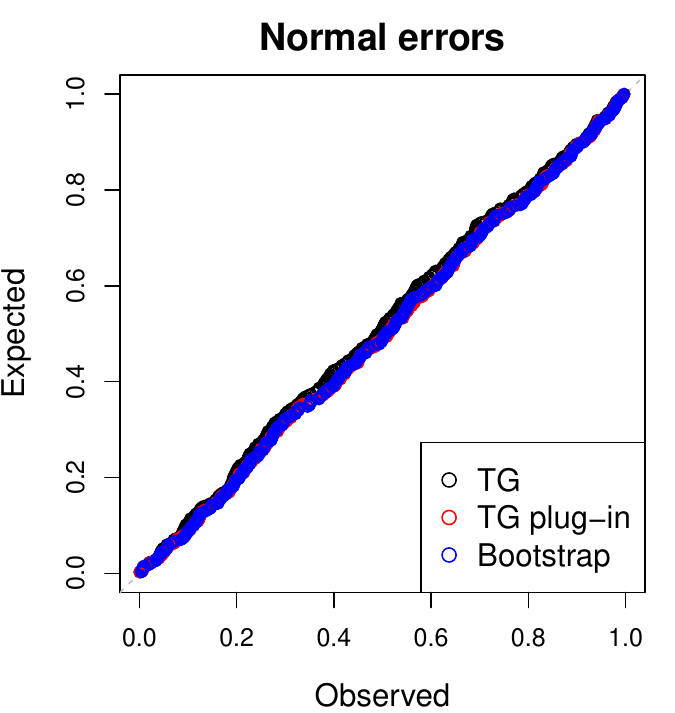}
\includegraphics[width=0.24\textwidth]{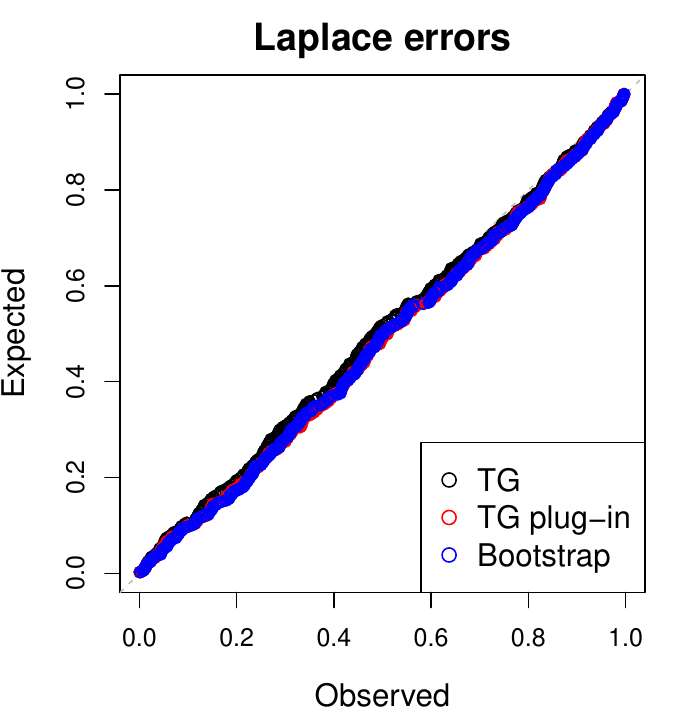}
\includegraphics[width=0.24\textwidth]{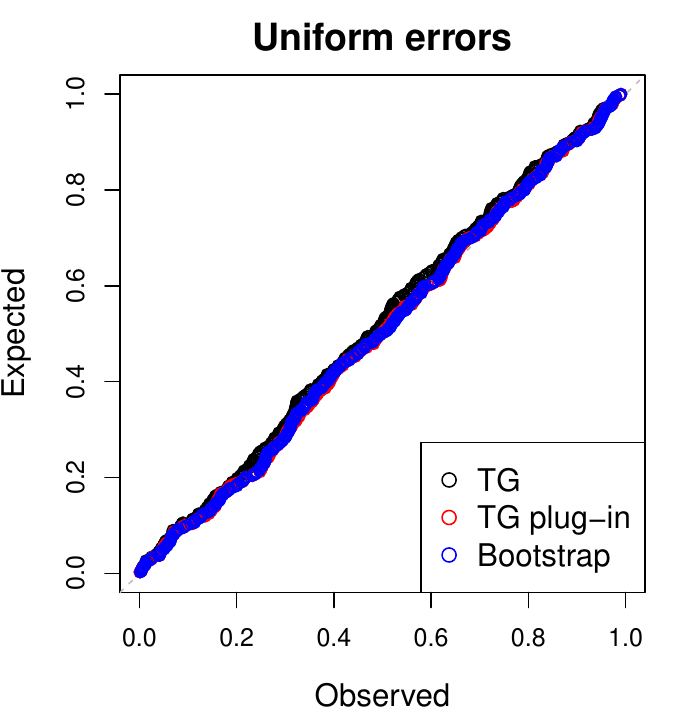}
\includegraphics[width=0.24\textwidth]{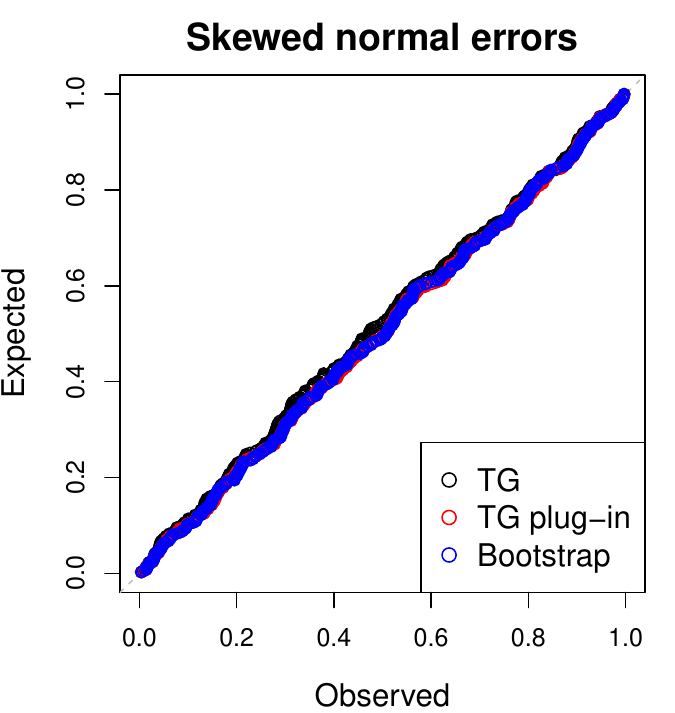}
\caption{\it\small P-values are shown, after each of 3 steps of LAR.}
\label{fig:lo_cor_pval}
\end{subfigure}

\bigskip\smallskip\smallskip
\begin{subfigure}[b]{\textwidth}
\centering
All steps, pivotal statistics \smallskip \\
\includegraphics[width=0.24\textwidth]{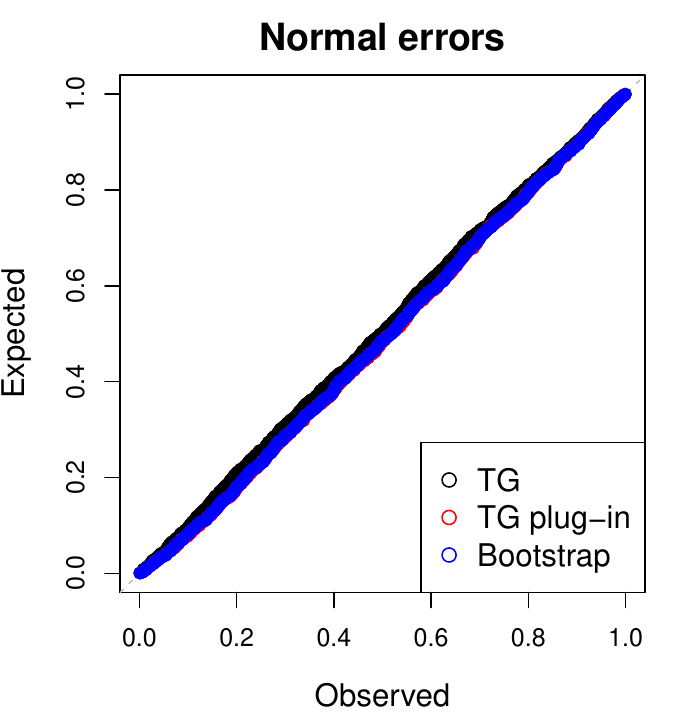}
\includegraphics[width=0.24\textwidth]{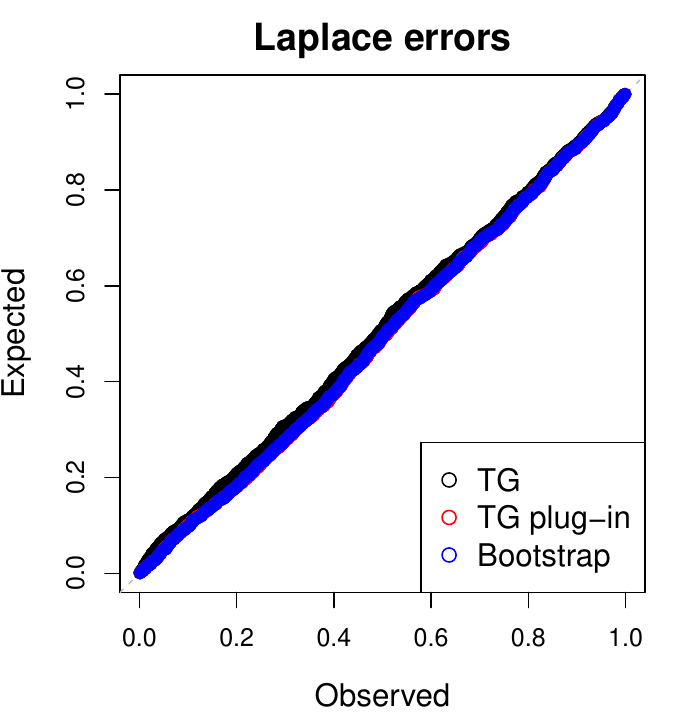}
\includegraphics[width=0.24\textwidth]{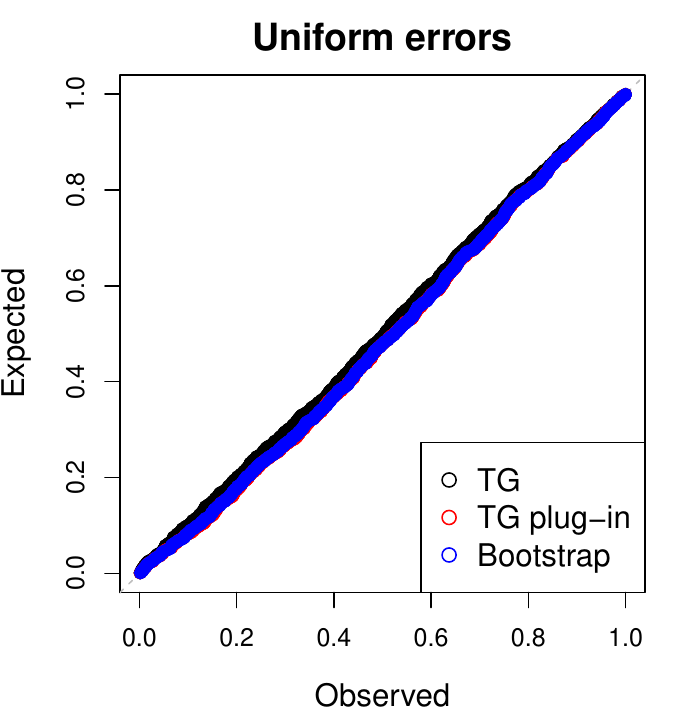}
\includegraphics[width=0.24\textwidth]{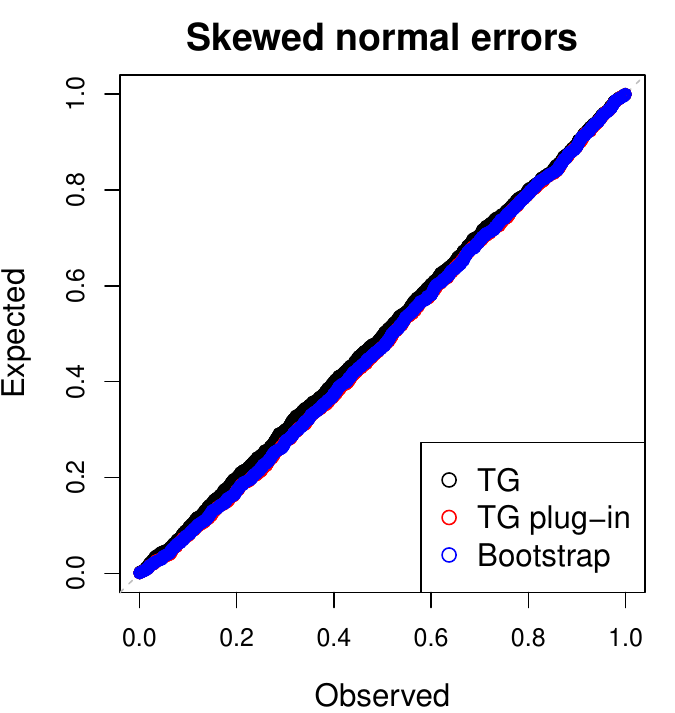}
\caption{\it\small Pivotal statistics are shown, aggregated over all
  3 steps of LAR.}
\label{fig:lo_cor_pivot}
\end{subfigure}

\caption{\it\small QQ plots as in Figure \ref{fig:lo}, but in a setup
  where the predictor variables have pairwise correlation 0.5.} 
\label{fig:lo_cor}
\end{figure}

\subsection{Confidence intervals for uniform, Laplace, and 
  skew normal noise}
\label{app:more_intervals}

Figures \ref{fig:int_e2} through \ref{fig:int_e4}
show sample confidence intervals for the problem setting of
Section \ref{sec:interval_ex},
when the error distribution is uniform, Laplace, and skew normal, 
respectively. 

\begin{figure}[htbp]
\centering
\includegraphics[width=\textwidth]{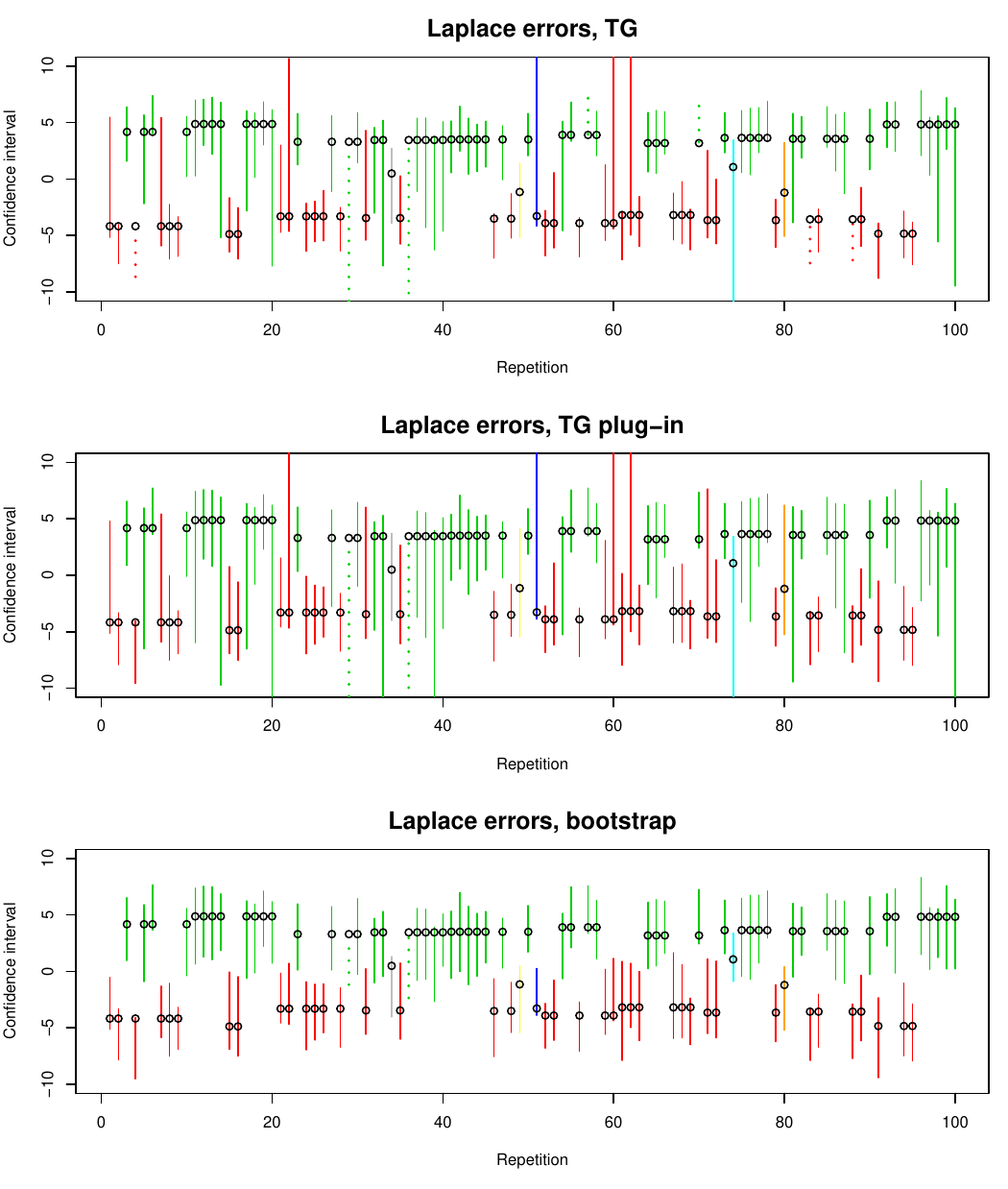}
\caption{\small\it Confidence intervals from 100 draws of $Y$, similar
  to those in Figure \ref{fig:int_e1}, but under a uniform
  noise distribution.}
\label{fig:int_e2}
\end{figure}

\begin{figure}[htbp]
\centering
\includegraphics[width=\textwidth]{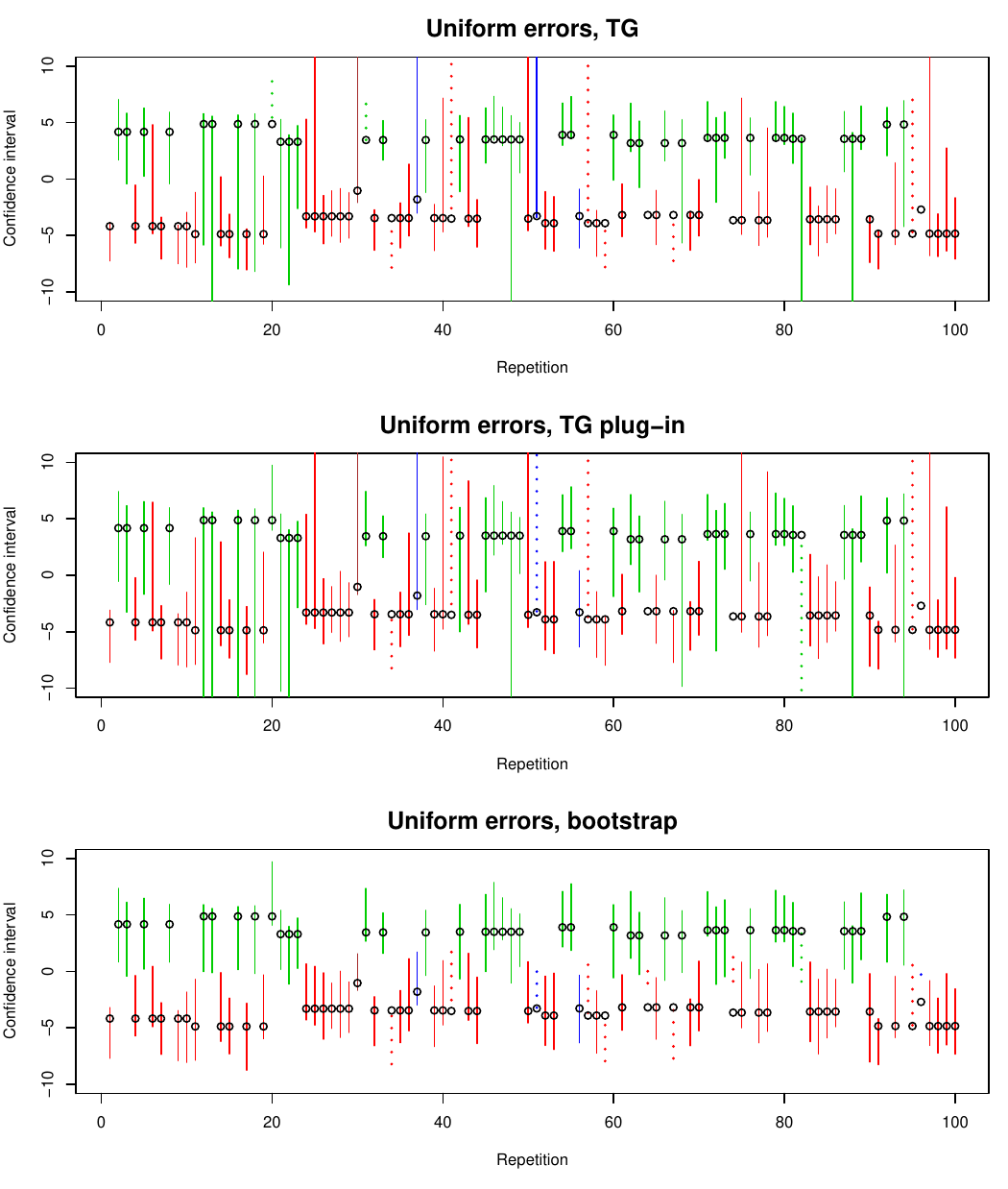}
\caption{\small\it Confidence intervals from 100 draws of $Y$, similar
  to those in Figure \ref{fig:int_e1}, but under a Laplace
  noise distribution.}
\label{fig:int_e3}
\end{figure}

\begin{figure}[htbp]
\centering
\includegraphics[width=\textwidth]{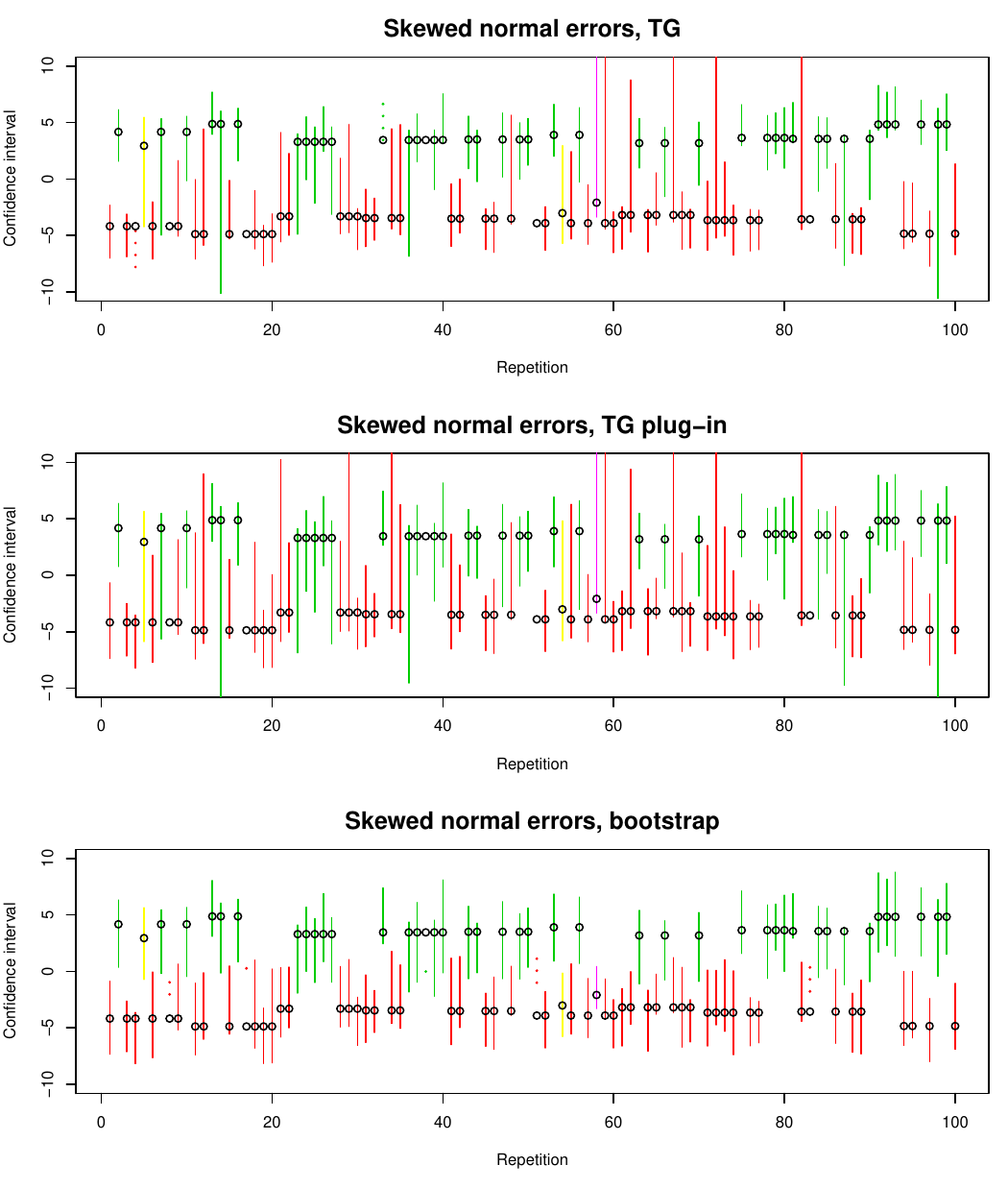}
\caption{\small\it Confidence intervals from 100 draws of $Y$, similar
  to those in Figure \ref{fig:int_e1}, but under a skew
  normal noise distribution.}
\label{fig:int_e4}
\end{figure}

\subsection{Confidence interval summary statistics for correlated
  predictors} 
\label{app:interval_ex_cor}

Table \ref{tab:int_cor} gives summary statistics of confidence
intervals obtained by inverting the original TG, plug-in TG, and 
bootstrap TG statistics, as in Table \ref{tab:int} of Section
\ref{sec:interval_ex}, but for the correlated predictors setup
described in Section \ref{app:pval_ex_cor}.   

\begin{table}[htbp]
\small
\centering
\begin{tabular}{|r|r|}
\multicolumn{2}{r}{} \\ 
\multicolumn{2}{r}{} \\ 
\hline
\multirow{3}{*}{N} 
& \multicolumn{1}{r|}{TG} \\
& \multicolumn{1}{r|}{Plug-in} \\
& \multicolumn{1}{r|}{Boot} \\
\hline
\multirow{3}{*}{L} 
& \multicolumn{1}{r|}{TG} \\
& \multicolumn{1}{r|}{Plug-in} \\
& \multicolumn{1}{r|}{Boot} \\
\hline
\multirow{3}{*}{U}
& \multicolumn{1}{r|}{TG} \\
& \multicolumn{1}{r|}{Plug-in} \\
& \multicolumn{1}{r|}{Boot} \\
\hline
\multirow{3}{*}{S}
& \multicolumn{1}{r|}{TG} \\
& \multicolumn{1}{r|}{Plug-in} \\
& \multicolumn{1}{r|}{Boot} \\
\hline
\end{tabular}
\begin{tabular}{|rrr|}
\hline
\multicolumn{3}{|c|}{Step 1} \\
\hline
Coverage & Power & Width \\
\hline
0.908 & 0.220 & 6.907 \\
0.926 & 0.186 & 8.186 \\
0.924 & 0.192 & 4.973 \\
\hline
0.912 & 0.264 & 6.510 \\
0.928 & 0.182 & 7.341 \\
0.934 & 0.176 & 5.117 \\
\hline
0.910 & 0.226 & 6.826 \\
0.926 & 0.154 & 8.192 \\
0.918 & 0.172 & 4.949 \\
\hline
0.904 & 0.240 & 6.479 \\
0.912 & 0.174 & 7.717 \\
0.908 & 0.192 & 4.973 \\
\hline
\end{tabular}
\begin{tabular}{|rrr|}
\hline
\multicolumn{3}{|c|}{Step 2} \\
\hline
Coverage & Power & Width \\
\hline
0.920 & 0.244 & 25.960 \\
0.922 & 0.210 & 30.113 \\
0.916 & 0.254 & 8.745 \\
\hline
0.886 & 0.290 & 23.668 \\
0.894 & 0.264 & 26.841 \\
0.916 & 0.294 & 8.769 \\
\hline
0.898 & 0.262 & 25.371 \\
0.906 & 0.200 & 29.211 \\
0.886 & 0.280 & 8.817 \\
\hline
0.910 & 0.262 & 24.502 \\
0.920 & 0.218 & 28.979 \\
0.904 & 0.254 & 8.697 \\
\hline
\end{tabular}
\begin{tabular}{|rrr|}
\hline
\multicolumn{3}{|c|}{Step 3} \\
\hline
Coverage & Power & Width \\
\hline
0.904 & 0.110 & 55.614 \\
0.908 & 0.106 & 66.083 \\
0.914 & 0.116 & 10.667 \\
\hline
0.894 & 0.126 & 54.351 \\
0.894 & 0.130 & 60.831 \\
0.884 & 0.148 & 10.583 \\
\hline
0.920 & 0.106 & 52.786 \\
0.922 & 0.098 & 63.915 \\
0.910 & 0.122 & 10.474 \\
\hline
0.892 & 0.136 & 56.700 \\
0.894 & 0.120 & 68.143 \\
0.896 & 0.122 & 10.486 \\
\hline
\end{tabular}
\caption{\small\it Summary statistics for 90\% confidence intervals,
  as in Table \ref{tab:int}, but in a modified problem setting 
  such that the predictor variables have pairwise correlation 0.5. The 
  standard errors are roughly 0.01, 0.02, and 0.87 for the coverage,
  power, and width statistics, respectively.}     
\label{tab:int_cor}
\end{table}

\subsection{Proof of Theorem \ref{thm:highdim}}
\label{app:highdim}

Let us denote by $N_j$ the number of observations in the $j$th column 
of the data array $Y_{ij}$, $i=1,\ldots,m$, $j=1,\ldots,d$ that are
drawn from the $N(B,1)$ mixture component.  Similarly, let $N'_j$
denote the number of observations in the $j$th column drawn from the  
$N(0,1)$ mixture component.  Then we will define $E$ to be the event
\begin{equation*}
E = \Big\{ \text{For some $j=1,\ldots,d$, we have $N_j=m$ and $N'_\ell
  \geq  m-2\pi m d$ for all $\ell \not= j$} \Big\}.
\end{equation*} 
In words, $E$ is the event that exactly one column has all of its
observations drawn from $N(B,1)$, and each of the rest of the $d-1$
columns have at least $m-2\pi m d$ observations from $N(0,1)$.
We calculate
\begin{align*}
\P(E) &= d \pi^m \P\big(N'_1 \geq m- 2\pi m d\big)^{d-1}
\\
&= \Big(1-\P\big(N'_1+\tilde{N}_1 \geq 2\pi m d\big)
\Big)^{d-1} \\
&\geq \Bigg(1-\frac{1}{d}\Bigg)^{d-1} \\
&\to 1/e,
\end{align*}
where in the second line we used that $d \pi^m=1$ by
construction, and introduced 
the notation \smash{$\tilde{N}_j$} for the number of observations in
column $j$ that are drawn from the $N(-B,1)$ mixture component; in the
third line we used Markov's inequality.  

On the event $E$, intersected with an event whose probability
tends to one, we have $W_{(1)},W_{(2)} \to \infty$, and furthermore
\begin{align*}
\sqrt{m}W_{(1)} &\geq \sqrt{m} B + Z_0 \geq
\sqrt{m}B/2, \\
\sqrt{m}W_{(2)} &\leq 2\pi m^{3/2} d B + \max_{j=1,\ldots,d-1} \; Z_j 
\leq 4\pi m^{3/2} d B,
\end{align*}
where $Z_0,Z_1,\ldots,Z_{d-1}$ denote standard normals.  We note that 
the ultimate bounds on the right-hand sides in the two lines above are  
extremely loose, but will suffice for our purposes.  Hence using Mills'
ratio, we can bound the TG statistic on the event in consideration by 
\begin{align*}
\cT(Y;0) &\leq
\exp\Bigg(-\frac{\big(mW_{(1)}^2-mW_{(2)}^2\big)}{4}\Bigg)  
\frac{W_{(2)}}{W_{(1)}} \Bigg(1 + \frac{2}{mW_{(2)}^2}\Bigg) \\
&\leq 2 \exp\Bigg(-\frac{\big(mW_{(1)}^2-mW_{(2)}^2\big)}{4}\Bigg), 
\end{align*}
for sufficiently large $d$.  But on this same event we have that
\begin{align*}
mW_{(1)}^2-mW_{(2)}^2 \geq mB^2 \Bigg(\frac{1}{4} - 16\pi^2 m^2
d^2\Bigg),
\end{align*}
and it is straightforward to check that the right-hand side of
the bound above diverges to $\infty$, given our assumptions on
$m,d,\pi,B$. Therefore, we have shown that on an event whose
probability tends to at least $1/e$, the TG statistic converges to
0. 

As for the conditional result, notice that for any model $(j,s)$, we have
by symmetry (under $\mu=0$) \smash{$\P(T(Y;j,s,0) \leq t \,|\,
  \hat{M}(Y)=(j,s))=\P(\cT(Y;0) \leq t)$}, as well as \smash{$\P(E \,|\, 
  \hat{M}(Y)=(j,s))=\P(E)$}. Hence the conditional TG statistic 
\smash{$T(Y;j,s,0) \,|\,  \hat{M}(Y)=(j,s)$} itself cannot be asymptotically
uniform, and converges to 0 on a event whose limiting probability is at least
$1/e$, conditional on \smash{$\hat{M}(Y)=(j,s)$}.

\subsection{Some thoughts on instability in high dimensions}   
\label{app:instability}

The TG statistic is defined by the ratio of normal tail probabilities. 
If the dimension $d$ is large (in which case we are searching through
a large space of models), or there are some large effects, then we
often find ourselves evaluating the pivot far into the tails. 
The point of evaluation is given by a linear function of the data,
which should 
itself converge to a Gaussian distribution (at least when $d$ is
finite). But even a small amount of non-Gaussianity is magnified 
when we are in the tails.   To see this, consider the function  
\begin{equation*}
H_p(t) = 
\frac{\int_{t+1}^\infty p(z) \, dz}{\int_t^\infty p(z) \, dz}.
\end{equation*}
The left plot in Figure \ref{fig:tail} shows two densities $p$ and $q$
which are nearly indistinguishable. The right plot shows their
corresponding tail functions $H_p$ and $H_q$. Even
though $p$ and $q$ are close, we see that $H_p$ and $H_q$ are quite
different.  
The message is that any inferential method that depends
heavily on extreme tail behavior could be unreliable.

\begin{figure}[tb]
\centering
\includegraphics[width=0.4\textwidth]{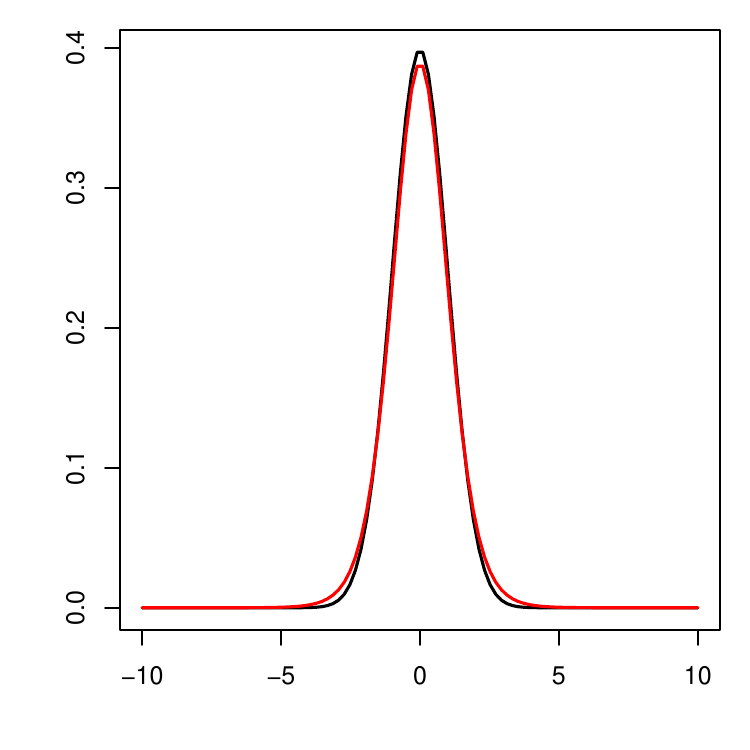}
\includegraphics[width=0.4\textwidth]{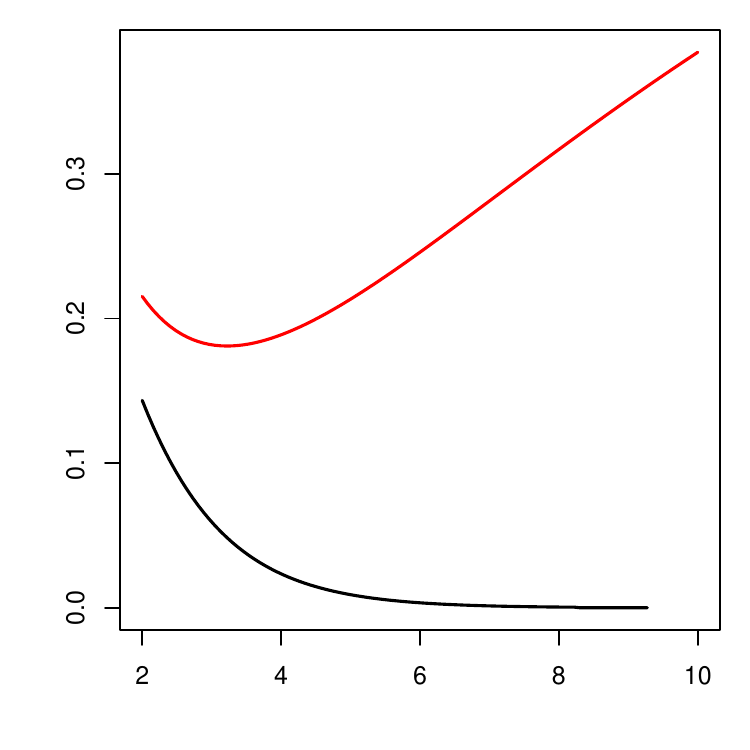}
\caption{\it\small The left plot shows two densities $p,q$, in black
  and red; the right shows their tail functions $H_p,H_q$ (in 
  corresponding colors).}
\label{fig:tail}
\end{figure}

Perhaps more visually striking is a plot of the TG statistic, when
viewed as a function of $y$ (for $X$ fixed).  This is shown in
Figure \ref{fig:nonsmooth}, where the statistic is used to test
$\mu=0$, and we used the same setup---thus the same model selection   
partition elements, and 
even matching colors---as in Figure \ref{fig:partition}.  Here $n=2$,  
so it is possible to fully visualize the TG statistic as a function of 
$y \in \R^2$.  This function is not well-behaved
at the boundaries between partition elements corresponding to different  
model selection events.  Technically, this function is continuous on the
interior of each partition element, which permits an application of the
(uniform) continuous mapping theorem when $d$ is fixed.  But the
derivatives at the boundaries are infinite and, especially in high-dimensional
problem settings, there is a nonnegligible probability of being near a
boundary. Thus a small perturbation to the data could have a dramatic effect on
the value of the pivot.   

\begin{figure}[p]
\centering
\vspace{-70pt}
\includegraphics[width=0.9\textwidth]{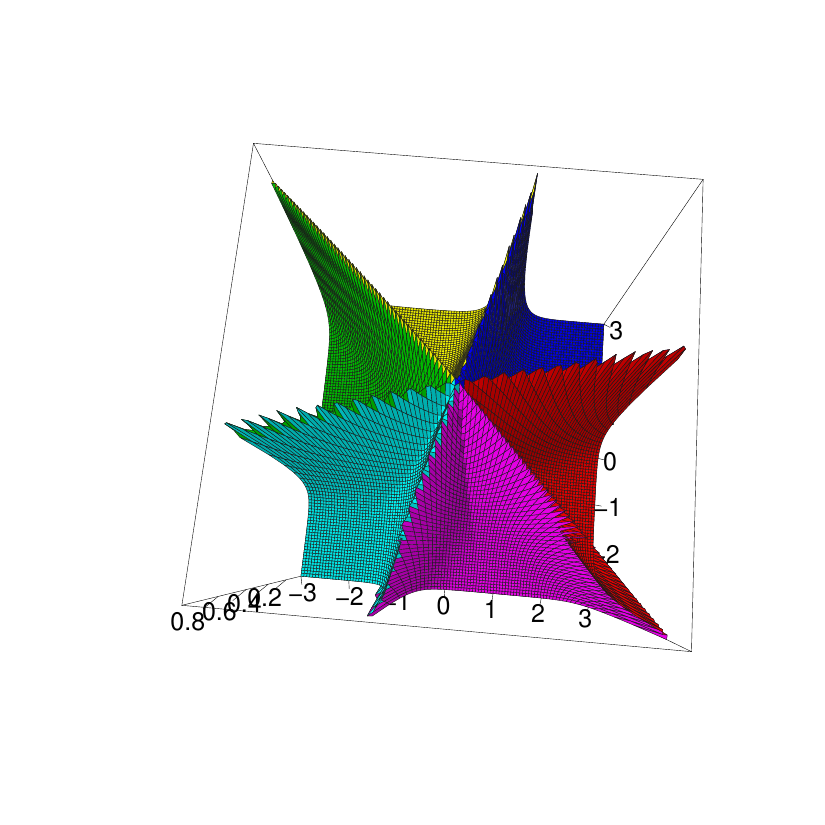} \\
\vspace{-140pt}
\includegraphics[width=0.9\textwidth]{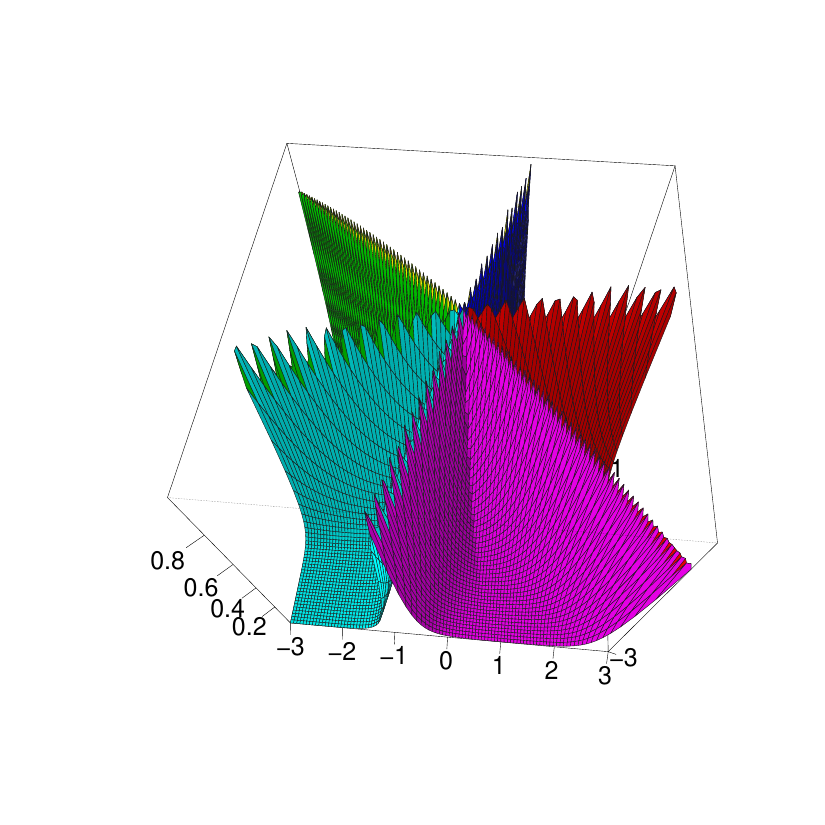}
\vspace{-50pt}
\caption{\it\small Two 3d views of the TG statistic, with the pivot
  value set at $\mu=0$, in same setup as in Figure
  \ref{fig:partition}.  Here $n=2$, and the statistic is plotted as a
  function of $y \in \R^2$.} 
\label{fig:nonsmooth}
\end{figure}

\bibliographystyle{agsm}
\bibliography{ryantibs}

\end{document}